\documentclass[a4paper,reqno,10pt]{amsart}

\usepackage{epsf,graphicx,epsfig}
\usepackage{amscd}
\usepackage{amsmath,latexsym,amssymb,amsthm}
\usepackage[nospace,noadjust]{cite}
\usepackage{textcomp}
\usepackage{amsfonts}
\usepackage{setspace,cite}
\usepackage{lscape,fancyhdr,fancybox}
\usepackage{stmaryrd}
\usepackage{mathrsfs}
\usepackage[all,cmtip]{xy}
\usepackage{tikz}
\usepackage{cancel}
\usetikzlibrary{shapes,arrows,decorations.markings}
\usepackage{graphicx}
\usepackage{tikz-cd}
\usepackage{color}
\usepackage{url}
\usepackage{enumerate}
\usepackage[mathscr]{euscript}

  \theoremstyle{plain}

\swapnumbers
    \newtheorem{thm}{Theorem}[section]
    \newtheorem{prop}[thm]{Proposition}

    \newtheorem{subsec}[thm]{}
\theoremstyle{definition}
    \newtheorem{defn}[thm]{Definition}
        \newtheorem{remark}[thm]{Remark}
    \newtheorem{exam}[thm]{Example}

\theoremstyle{remark}

\title{}
\author{}
\date{}
\usepackage{amssymb}

\usepackage{hyperref}
\hypersetup{
	colorlinks,
	citecolor=blue,
	filecolor=black,
	linkcolor=blue,
	urlcolor=black
}

\begin{document}

\title[]{Cohomology theory of Nijenhuis Lie algebras and (generic) Nijenhuis Lie bialgebras}

\author{Apurba Das}
\address{Department of Mathematics,
Indian Institute of Technology, Kharagpur 721302, West Bengal, India}
\email{apurbadas348@gmail.com, apurbadas348@maths.iitkgp.ac.in}



\begin{abstract}
The aim of this paper is twofold. In the first part, we define the cohomology of a Nijenhuis Lie algebra with coefficients in a suitable representation. Our cohomology of a Nijenhuis Lie algebra governs the simultaneous deformations of the underlying Lie algebra and the Nijenhuis operator. Subsequently, we define homotopy Nijenhuis operators on $2$-term $L_\infty$-algebras and show that in some cases they are related to third cocycles of Nijenhuis Lie algebras. In another part of this paper, we extend our study to (generic) Nijenhuis Lie bialgebras where the Nijenhuis operators on the underlying Lie algebras and Lie coalgebras need not be the same. In due course, we introduce matched pairs and Manin triples of Nijenhuis Lie algebras and show that they are equivalent to Nijenhuis Lie bialgebras. Finally, we consider the admissible classical Yang-Baxter equation whose antisymmetric solutions yield Nijenhuis Lie bialgebras.
\end{abstract}

\maketitle



\medskip

\begin{center}
\noindent {2020 MSC classification:} 17B40, 17B55, 17B56, 17B62.

\noindent  {Keywords:} Nijenhuis Lie algebras, Cohomology, Deformations, Homotopy Nijenhuis operators, Nijenhuis Lie bialgebras.
\end{center}

 



\thispagestyle{empty}

\tableofcontents


\medskip

\section{Introduction}
A traditional approach to studying a mathematical structure is associating some invariants. Among others, the cohomology of an associative algebra is a very classical invariant that controls the deformations and extensions of the given algebra \cite{hoch}, \cite{gers}. Subsequently, cohomology and deformation theory were generalized for Lie algebras by Nijenhuis and Richardson \cite{nij-ric}. As of now, cohomologies for various kinds of algebras have been developed and their applications were also obtained. See also \cite{doubek}, \cite{markl} and the references therein for more details. Recently, people have been very interested in operated algebras (i.e. algebras endowed with distinguished operators). For such an operated algebra, it is worth meaningful to considering the simultaneous deformations of the underlying algebra and the operator. An important instance first appeared in the work of Gerstenhaber and Schack \cite{gers-sch} in their study of algebras with homomorphisms. They developed the cohomology and deformation theory of an associative algebra with a distinguished homomorphism. Later, Loday \cite{loday-der} considered the operad encoding algebras with derivations. In achieving significant progress in Rota-Baxter operators and averaging operators in the last few years, many authors have derived the cohomology and deformations of Rota-Baxter (Lie) algebras and averaging (Lie) algebras. Another interesting operator that appears in the linear deformation theory of algebraic structures \cite{koss}, integrable systems, nonlinear evolution equations and tensor hierarchies \cite{dorfman} \cite{koss}, bi-Hamiltonian systems \cite{gra-bi} and the geometry of vector-valued differential forms \cite{fro-nij-1} is `Nijenhuis operator'. It is important to mention that Nijenhuis operators on Lie algebras are closely related to twisted Rota-Baxter operators and they produce NS-Lie algebras \cite{das-twisted}. A Lie algebra endowed with a distinguished Nijenhuis operator is referred to as a `Nijenhuis Lie algebra'. By considering the importance of Nijenhuis operators in various directions of mathematics and mathematical physics, it is desirable to have the cohomology theory of a Nijenhuis Lie algebra that controls the simultaneous deformations of the underlying Lie algebra and the Nijenhuis operator. Our primary aim in this paper is to develop a suitable cohomology theory for a Nijenhuis Lie algebra that serves the purpose. To construct the cohomology of a Nijenhuis Lie algebra, we adopt the following approach. First, we recall the Fr\"{o}licher-Nijenhuis bracket \cite{fro-nij-1} \cite{yang} associated with a given Lie algebra whose Maurer-Cartan elements are precisely Nijenhuis operators on this Lie algebra. As a consequence of this characterization, we can define the cohomology of a Nijenhuis operator $N$. It is important to remark that the cohomology of a Nijenhuis operator is not the same as the Chevalley-Eilenberg cohomology of the deformed Lie algebra. However, we obtain a homomorphism from the cohomology of a Nijenhuis operator $N$ to the Chevalley-Eilenberg cohomology of the deformed Lie algebra. Next, given a Nijenhuis Lie algebra, we find a suitable homomorphism from the Chevalley-Eilenberg cochain complex of the underlying Lie algebra to the cochain complex of the Nijenhuis operator. The mapping cone corresponding to this homomorphism is defined to be the cochain complex of the given Nijenhuis Lie algebra. The cohomology groups thus obtained are said to be the cohomology groups of the Nijenhuis Lie algebra (with coefficients in the adjoint representation). Subsequently, we generalize this cohomology of a Nijenhuis Lie algebra in the presence of a suitable representation. More specifically, we consider Nijenhuis representations of a Nijenhuis Lie algebra and define cohomology with coefficients in a Nijenhuis representation.

\medskip

To find applications of our cohomology theory, we first consider deformations of a Nijenhuis Lie algebra where we allow the simultaneous deformations of the Lie bracket and the Nijenhuis operator. Among others, we show that the set of all equivalence classes of infinitesimal deformations of a Nijenhuis Lie algebra has a bijection with the second cohomology group of the Nijenhuis Lie algebra with coefficients in the adjoint Nijenhuis representation. Another application of the second cohomology group of a Nijenhuis Lie algebra with coefficients in a Nijenhuis representation is given by the present author in \cite{das-nj}. More precisely, the author has considered the non-abelian cohomology group of a Nijenhuis Lie algebra with values in another Nijenhuis Lie algebra and showed that it classifies the set of all isomorphism classes of non-abelian extensions of Nijenhuis Lie algebras. In a particular case, it gives a bijection between the set of all isomorphism classes of abelian extensions of a Nijenhuis Lie algebra by a given Nijenhuis representation and the second cohomology group. On the other hand, to get another application of our cohomology theory, we first introduce `homotopy Nijenhuis operators' on a $2$-term $L_\infty$-algebra. In this paper, we shall call a $2$-term $L_\infty$-algebra endowed with a homotopy Nijenhuis operator as a $2$-term Nijenhuis $L_\infty$-algebra. We show that `skeletal' $2$-term Nijenhuis $L_\infty$-algebras are characterized by third cocycles of Nijenhuis Lie algebras. We also consider crossed modules of Nijenhuis Lie algebras and show that they characterize `strict' $2$-term Nijenhuis $L_\infty$-algebras.

\medskip

In another part of this paper, we develop the bialgebra theory for Nijenhuis Lie algebras. The notion of Lie bialgebras first appeared in the work of Drinfeld \cite{drin} in the study of deformations of universal enveloping algebras of Lie algebras. A Lie bialgebra is simply a Lie algebra and a Lie coalgebra both defined on a vector space satisfying a compatibility condition. Generalizing this concept, the authors in \cite{lang-sheng}, \cite{bai-bialgebra} have recently developed the bialgebra theory for Rota-Baxter Lie algebras. Among others, they defined the notion of Rota-Baxter Lie bialgebras and found their relations with the classical Yang-Baxter equation. In the context of Nijenhuis Lie algebras, recently, the author of \cite{ravanpak} considered the notion of an NL bialgebra as the analogue of Poisson-Nijenhuis structures studied in the context of integrable systems. According to this, a Nijenhuis Lie bialgebra is given by a Lie bialgebra equipped with a Nijenhuis operator $N$ on the underlying Lie algebra satisfying certain compatibility conditions. It turns out that the map $N$ becomes a Nijenhuis operator on the underlying Lie coalgebra. Since the Nijenhuis operator on the underlying Lie algebra and Lie coalgebra are the same, this definition of an NL bialgebra doesn't fit with the possible theories of matched pairs and Manin triples of Nijenhuis Lie algebras. In this paper, we define a generic Nijenhuis Lie bialgebra where the underlying Nijenhuis Lie algebra and Nijenhuis Lie coalgebra share different Nijenhuis operators in general. To justify our definition, we consider matched pairs and Manin triples of Nijenhuis Lie algebras and show that they are equivalent to Nijenhuis Lie bialgebras. Finally, we consider the admissible classical Yang-Baxter equation whose antisymmetric solutions yield Nijenhuis Lie bialgebras. In the end, we obtain some important results including representations and matched pairs of NS-Lie algebras and relate them with the corresponding notions for Nijenhuis Lie algebras.
 
\medskip

The paper is organized as follows. In Section \ref{sec2}, we recall some necessary background on Nijenhuis operators and the Fr\"{o}licher-Nijenhuis bracket associated with a given Lie algebra. In Section \ref{sec3}, we first introduce and study the cohomology of a Nijenhuis operator and find its relation with the Chevalley-Eilenberg cohomology of the deformed Lie algebra. Subsequently, we also define the cohomology of a Nijenhuis Lie algebra with coefficients in a Nijenhuis representation. As an application of our cohomology theory, in Section \ref{sec4}, we consider deformations of a Nijenhuis Lie algebra. Then in Section \ref{sec5}, we introduce homotopy Nijenhuis operators on $2$-term $L_\infty$-algebras and give characterizations of skeletal and strict $2$-term Nijenhuis $L_\infty$-algebras. Finally, in Section \ref{sec6}, we consider (generic) Nijenhuis Lie bialgebras and show that they are equivalent to matched pairs and Manin triples of Nijenhuis Lie algebras.

\medskip

\medskip

\section{Lie algebras and Nijenhuis operators}\label{sec2}
In this section, we first recall the Nijenhuis-Richardson bracket whose Maurer-Cartan elements are precisely Lie algebra structures on a given vector space. Next, we revise some basic properties of Nijenhuis operators on a Lie algebra and recall the construction of the Fr\"{o}licher-Nijenhuis bracket that characterizes Nijenhuis operators as its Maurer-Cartan elements.

\medskip

Let $\mathfrak{g}$ be a vector space (need not have any additional structure). Then the Nijenhuis-Richardson bracket associated with the space $\mathfrak{g}$ is a graded Lie bracket on the space of all antisymmetric multilinear maps on $\mathfrak{g}$. Explicitly, the {\bf Nijenhuis-Richardson bracket} \cite{nij-ric} is the bracket
\begin{align*}
    [~,~]_\mathsf{NR} : \mathrm{Hom} (\wedge^{m} \mathfrak{g}, \mathfrak{g}) \times \mathrm{Hom} (\wedge^{n} \mathfrak{g}, \mathfrak{g}) \rightarrow \mathrm{Hom} (\wedge^{m+n-1} \mathfrak{g}, \mathfrak{g})
\end{align*}
defined by $[P, Q]_\mathsf{NR} := i_P Q - (-1)^{(m-1)(n-1)} ~\!  i_Q P$, where
\begin{align*}
    (i_P Q ) (x_1, \ldots, x_{m+n-1} ) = \sum_{\sigma \in \mathrm{Sh} (m, n-1)} (-1)^\sigma ~ \! Q \big( P (x_{\sigma (1)}, \ldots, x_{\sigma (m)}), x_{\sigma (m+1)}, \ldots, x_{\sigma (m+n-1)}   \big),
\end{align*}
for $P \in  \mathrm{Hom} (\wedge^{m} \mathfrak{g}, \mathfrak{g})$, $Q \in  \mathrm{Hom} (\wedge^{n} \mathfrak{g}, \mathfrak{g})$ and elements $x_1, \ldots, x_{m+n-1} \in \mathfrak{g}$. Then it turns out that $\big( \oplus_{n = 0}^\infty  \mathrm{Hom} (\wedge^{n+1} \mathfrak{g}, \mathfrak{g}), [~, ~]_\mathsf{NR}   \big)$ is a graded Lie algebra, called the Nijenhuis-Richardson graded Lie algebra associated to the vector space $\mathfrak{g}$. An element $\mu \in \mathrm{Hom} (\wedge^2 \mathfrak{g}, \mathfrak{g})$ is a Maurer-Cartan element in the Nijenhuis-Richardson graded Lie algebra if and only if the bracket $[~,~]_\mathfrak{g}: \mathfrak{g} \times \mathfrak{g} \rightarrow \mathfrak{g} $ defined by $[x, y]_\mathfrak{g}:= \mu (x, y)$
is a Lie bracket on the vector space $\mathfrak{g}$.

Let $(\mathfrak{g}, [~,~]_\mathfrak{g})$ be a Lie algebra and $(\mathcal{V}, \rho)$ be a representation of it. That is, $\mathcal{V}$ is a vector space endowed with a Lie algebra homomorphism $\rho: \mathfrak{g} \rightarrow \mathrm{End}(\mathcal{V})$. Then the Chevalley-Eilenberg cochain complex of the Lie algebra $(\mathfrak{g}, [~,~]_\mathfrak{g})$ with coefficients in the representation $(\mathcal{V}, \rho)$ is the complex $\{ \oplus_{n=0}^\infty \mathrm{Hom} (\wedge^n \mathfrak{g}, \mathcal{V}), \delta_\mathrm{CE} \}$, where the coboundary map $\delta_\mathrm{CE} : \mathrm{Hom} (\wedge^n \mathfrak{g}, \mathcal{V}) \rightarrow \mathrm{Hom} (\wedge^{n+1} \mathfrak{g}, \mathcal{V})$ is given by
\begin{align}\label{ce-diff}
    (\delta_\mathrm{CE} f) (x_1, \ldots, x_{n+1} ) =~& \sum_{i=1}^{n+1} (-1)^{i+1} ~ \! \rho_{x_i} f (x_1, \ldots, \widehat{x_i}, \ldots, x_{n+1}) \\
    &+ \sum_{1 \leq i < j \leq n+1} (-1)^{i+j} ~ \! f ([x_i, x_j]_\mathfrak{g}, x_1, \ldots, \widehat{x_i}, \ldots, \widehat{x_j}, \ldots, x_{n+1}), \nonumber
\end{align}
for $f \in \mathrm{Hom} (\wedge^n \mathfrak{g}, \mathcal{V}) $ and $x_1, \ldots, x_{n+1} \in \mathfrak{g}$. The corresponding cohomology groups are said to be the Chevalley-Eilenberg cohomology groups of the Lie algebra $(\mathfrak{g}, [~,~]_\mathfrak{g})$ with coefficients in the representation $(\mathcal{V}, \rho)$, and they are denoted by $H^\bullet_\mathrm{CE} (\mathfrak{g} ; \mathcal{V})$. When $(\mathcal{V}, \rho) = (\mathfrak{g}, \mathrm{ad}_\mathfrak{g})$ is the adjoint representation, the coboundary map (\ref{ce-diff}) is simply given by $\delta_\mathrm{CE} (f) = - [\mu, f]_\mathsf{NR}$, for any $f \in \mathrm{Hom} (\wedge^n \mathfrak{g}, \mathfrak{g})$.
In this case, the corresponding cohomology groups are denoted by $H^\bullet_\mathrm{CE} (\mathfrak{g})$.


\begin{defn}
    Let $(\mathfrak{g}, [~,~]_\mathfrak{g})$ be a Lie algebra. A {\bf Nijenhuis operator} on this Lie algebra is a linear map $N : \mathfrak{g} \rightarrow \mathfrak{g}$ that satisfies
    \begin{align*}
        [N (x), N (y)]_\mathfrak{g} = N \big( [ N(x), y]_\mathfrak{g} + [x, N(y)]_\mathfrak{g} - N [x, y]_\mathfrak{g} \big), \text{ for all } x, y \in \mathfrak{g}. 
    \end{align*}
\end{defn}

There are various examples of Nijenhuis operators. First and foremost, for any Lie algebra $(\mathfrak{g}, [~,~]_\mathfrak{g})$, the identity map $\mathrm{Id}_\mathfrak{g} : \mathfrak{g} \rightarrow \mathfrak{g}$ is a Nijenhuis operator on it. Moreover, if $N: \mathfrak{g} \rightarrow \mathfrak{g}$ is a Nijenhuis operator then $\lambda N$ is also a Nijenhuis operator, for any $\lambda \in {\bf k}$. In a complex Lie algebra, the complex structure is itself a Nijenhuis operator on the underlying real Lie algebra. Any (relative) Rota-Baxter operator on a Lie algebra (with respect to a representation) can be lifted to a Nijenhuis operator on the semidirect product Lie algebra \cite{sheng-o}. See also \cite{koss}, \cite{das-nij} for more examples of Nijenhuis operators.

Let $(\mathfrak{g}, [~,~]_\mathfrak{g})$ be a Lie algebra and $N : \mathfrak{g} \rightarrow \mathfrak{g}$ be a Nijenhuis operator on it. Then the underlying vector space $\mathfrak{g}$ inherits a new Lie algebra structure with the bracket
\begin{align*}
    [x, y]_\mathfrak{g}^N := [ N(x), y]_\mathfrak{g} + [x, N(y)]_\mathfrak{g} - N [x, y]_\mathfrak{g}, \text{ for } x, y \in \mathfrak{g}.
\end{align*}
The Lie algebra $(\mathfrak{g}, [~,~]_\mathfrak{g}^N)$ which is often denoted by $\mathfrak{g}^N$ is said to be the {\em deformed Lie algebra} associated to the Nijenhuis operator $N$. The following more general result has been proved in \cite{koss}.

\begin{prop}
Let $(\mathfrak{g}, [~,~]_\mathfrak{g})$ be a Lie algebra and $N : \mathfrak{g} \rightarrow \mathfrak{g}$ be a Nijenhuis operator on it.
\begin{itemize}
    \item[(i)] Then for each $k \geq 0$, the map $N^k : \mathfrak{g} \rightarrow \mathfrak{g}$ is also a Nijenhuis operator on the Lie algebra $(\mathfrak{g}, [~,~]_\mathfrak{g})$.
    \item[(ii)] For any $k, l \geq 0$, the map $N^l : \mathfrak{g} \rightarrow \mathfrak{g}$ is a Nijenhuis operator on the deformed Lie algebra $(\mathfrak{g}, [~,~]_\mathfrak{g}^{N^k})$.
    \item[(iii)] Moreover, the deformed Lie algebras $(\mathfrak{g}, ( [~,~]_\mathfrak{g}^{N^k})^{N^l})$ and $(\mathfrak{g},  [~,~]_\mathfrak{g}^{N^{k+l}})$ are the same.
\end{itemize}
\end{prop}

Let $(\mathfrak{g}, [~,~]_\mathfrak{g})$ be a Lie algebra. In \cite{nij-ric2} Nijenhuis and Richardson defined a cup-product (generalizing Gerstenhaber's cup-product from the context of associative algebras)
\begin{align*}
    \vee :  \mathrm{Hom} (\wedge^{m} \mathfrak{g}, \mathfrak{g}) \times \mathrm{Hom} (\wedge^{n} \mathfrak{g}, \mathfrak{g}) \rightarrow \mathrm{Hom} (\wedge^{m+n} \mathfrak{g}, \mathfrak{g}) \text{ given by }
\end{align*}
\begin{align*}
    (P \vee Q) (x_1, \ldots, x_{m+n} ) := \sum_{\sigma \in \mathrm{Sh} (m, n)} (-1)^\sigma ~ \! [ P (x_{\sigma (1)}, \ldots, x_{\sigma (m)}), Q (x_{\sigma (m+1)}, \ldots, x_{\sigma (m+n)}) ]_\mathfrak{g},
\end{align*}
for $P \in \mathrm{Hom} (\wedge^{m} \mathfrak{g}, \mathfrak{g})$, $Q \in \mathrm{Hom} (\wedge^{n} \mathfrak{g}, \mathfrak{g})$ and elements $x_1, \ldots, x_{m+n} \in \mathfrak{g}$. This cup-product makes the pair $\big(  \oplus_{n=1}^\infty \mathrm{Hom} (\wedge^n \mathfrak{g}, \mathfrak{g}), \vee   \big)$ into a graded Lie algebra. Further, it has been shown in \cite{yang}, \cite{baishya-das} that the map 
\begin{align*}
    \mathrm{Hom} (\wedge^{m+1} \mathfrak{g}, \mathfrak{g}) \times \mathrm{Hom} (\wedge^{n} \mathfrak{g}, \mathfrak{g}) \rightarrow \mathrm{Hom} (\wedge^{m+n} \mathfrak{g}, \mathfrak{g}), ~  (P, Q) \mapsto i_P Q
\end{align*}
defines an action of the Nijenhuis-Richardson graded Lie algebra $\big( \oplus_{n = 0}^\infty  \mathrm{Hom} (\wedge^{n+1} \mathfrak{g}, \mathfrak{g}), [~, ~]_\mathsf{NR}   \big)$ on the cup-product graded Lie algebra $\big(  \oplus_{n=1}^\infty \mathrm{Hom} (\wedge^n \mathfrak{g}, \mathfrak{g}), \vee   \big)$. As a result, one shows that the bracket
\begin{align}\label{fn}
    [P, Q]_\mathsf{FN} := P \vee Q + (-1)^m ~\!  i_{\delta_\mathrm{CE} P} Q - (-1)^{(m+1)n}  ~\! i_{\delta_\mathrm{CE} Q} P
\end{align}
makes the graded space $ \oplus_{n=1}^\infty \mathrm{Hom} (\wedge^n \mathfrak{g}, \mathfrak{g})$ into a graded Lie algebra. The bracket defined in (\ref{fn}) is called the {\bf Fr\"{o}licher-Nijenhuis bracket} and the graded Lie algebra $ \big( \oplus_{n=1}^\infty \mathrm{Hom} (\wedge^n \mathfrak{g}, \mathfrak{g}) , [~,~]_\mathsf{FN} \big)$ is called the {\bf Fr\"{o}licher-Nijenhuis algebra} associated to the Lie algebra $(\mathfrak{g}, [~,~]_\mathfrak{g}).$ Moreover, for any $P \in \mathrm{Hom} (\wedge^{m} \mathfrak{g}, \mathfrak{g})$ and $ Q \in \mathrm{Hom} (\wedge^{n} \mathfrak{g}, \mathfrak{g})$,
we have (\cite{yang}, \cite{baishya-das})
\begin{align}\label{bracket-pres}
    \delta_\mathrm{CE} ( [P, Q]_\mathsf{FN}) = [ \delta_\mathrm{CE} P, \delta_\mathrm{CE} Q]_\mathsf{NR} .
\end{align}

For linear maps $N, N' : \mathfrak{g} \rightarrow \mathfrak{g}$ and $x, y \in \mathfrak{g}$, it follows from (\ref{fn}) that
\begin{align*}
    [N, N']_\mathsf{FN} (x, y) =~& [N(x), N'(y)]_\mathfrak{g} + [N'(x), N(y)]_\mathfrak{g} - N'\big( [ N(x), y]_\mathfrak{g} + [x, N(y)]_\mathfrak{g} - N [x, y]_\mathfrak{g} \big) \\
   & ~~ - N \big(  [ N'(x), y]_\mathfrak{g} + [x, N'(y)]_\mathfrak{g} - N' [x, y]_\mathfrak{g} \big).
\end{align*}
This shows that a linear map $N : \mathfrak{g} \rightarrow \mathfrak{g}$ is a Nijenhuis operator on the Lie algebra $(\mathfrak{g}, [~, ~]_\mathfrak{g})$ if and only if $[N, N]_\mathsf{FN} = 0$, i.e. $N$ is a Maurer-Cartan element in the Fr\"{o}licher-Nijenhuis algebra.

\medskip

\medskip

\section{Cohomology theory of Nijenhuis operators and Nijenhuis Lie algebras}\label{sec3}
Given a Lie algebra $(\mathfrak{g}, [~,~]_\mathfrak{g})$, here we first consider the cohomology of a Nijenhuis operator $N$ defined on it. We show that there is a homomorphism from the cohomology of a Nijenhuis operator $N$ to the Chevalley-Eilenberg cohomology of the deformed Lie algebra $(\mathfrak{g}, [~,~]_\mathfrak{g}^N)$. As a byproduct of the Chevalley-Eilenberg cochain complex of the given Lie algebra $(\mathfrak{g}, [~,~]_\mathfrak{g})$ and the cochain complex of the Nijenhuis operator $N$, we define the cochain complex (and hence the cohomology) associated to the Nijenhuis Lie algebra $(\mathfrak{g}, [~,~]_\mathfrak{g}, N)$. Subsequently, we generalize this construction to define the cohomology of a Nijenhuis Lie algebra with coefficients in an arbitrary Nijenhuis representation.

\subsection{Cohomology of Nijenhuis operators}\label{subsec-31} Let  $(\mathfrak{g}, [~,~]_\mathfrak{g})$ be a Lie algebra and $N : \mathfrak{g} \rightarrow \mathfrak{g}$ be a Nijenhuis operator on it. Note that, for each $n \geq 1$, the linear map $N$ induces a map 
\begin{align}\label{dn-map}
    d_N : \mathrm{Hom} (\wedge^n \mathfrak{g}, \mathfrak{g}) \rightarrow \mathrm{Hom} (\wedge^{n+1} \mathfrak{g}, \mathfrak{g}) ~~~ \text{ given by } ~~~ d_N (f) = [N, f]_\mathsf{FN}, \text{ for } f \in \mathrm{Hom} (\wedge^n \mathfrak{g}, \mathfrak{g}).
\end{align}
Since $N$ is a Maurer-Cartan element in the Fr\"{o}licher-Nijenhuis algebra (i.e. $[N, N]_\mathsf{FN} = 0$), it turns out that $(d_N)^2 = 0$. Explicitly, the map $d_N$ is given by
\begin{align}\label{dn-exp}
    &(d_N f) (x_1, \ldots, x_{n+1}) \\
    &= \sum_{i=1}^{n+1} (-1)^{i+1} ~\! [ N (x_i), f (x_1, \ldots, \widehat{x_i}, \ldots, x_{n+1})]_\mathfrak{g} \nonumber \\
    &+ \sum_{1 \leq i < j \leq n+1} (-1)^{i+j }  f ( [N(x_i), x_j]_\mathfrak{g} + [x_i, N (x_j)]_\mathfrak{g} - N [x_i, x_j]_\mathfrak{g}, x_1, \ldots, \widehat{x_i}, \ldots, \widehat{x_j}, \ldots, x_{n+1}) \nonumber \\
    &- N \big(   \sum_{i=1}^{n+1} (-1)^{i+1} [x_i, f (x_1, \ldots, \widehat{x_i}, \ldots, x_{n+1})]_\mathfrak{g} 
    + \sum_{1 \leq i < j \leq n+1} (-1)^{i+j} f ([x_i, x_j]_\mathfrak{g}, x_1, \ldots, \widehat{x_i}, \ldots, \widehat{x_j}, \ldots, x_{n+1})   \big), \nonumber
\end{align}
for $f \in \mathrm{Hom} (\wedge^n \mathfrak{g}, \mathfrak{g})$ and $x_1, \ldots, x_{n+1} \in \mathfrak{g}$. Further, one can extend (\ref{dn-map}) to a map (also denoted by the same notation) $d_N: \mathfrak{g} \rightarrow \mathrm{Hom} (\mathfrak{g}, \mathfrak{g})$ by
\begin{align}\label{dn0}
    d_N (x) (y) = [N (y), x]_\mathfrak{g} - N [y, x]_\mathfrak{g}, \text{ for } x, y \in \mathfrak{g}.
\end{align}
Then it follows that $\{ \oplus_{n=0}^\infty \mathrm{Hom} (\wedge^n \mathfrak{g}, \mathfrak{g}), d_N \}$ is a cochain complex, called the cochain complex associated to the Nijenhuis operator $N$. The corresponding cohomology groups are said to be the {\em cohomology groups} of the Nijenhuis operator $N$, and they are denoted by $H^\bullet (N)$.

\begin{exam}
    Let $(\mathfrak{g}, [~,~]_\mathfrak{g})$ be any Lie algebra. Note that the identity map $\mathrm{Id}_\mathfrak{g}: \mathfrak{g} \rightarrow \mathfrak{g}$ is a Nijenhuis operator on this Lie algebra. For this Nijenhuis operator $N = \mathrm{Id}_\mathfrak{g}$, it follows from the expressions (\ref{dn-exp}) and (\ref{dn0}) that the coboundary map $d_{N = \mathrm{Id}_\mathfrak{g}} : \mathrm{Hom} (\wedge^n \mathfrak{g}, \mathfrak{g}) \rightarrow \mathrm{Hom} (\wedge^{n+1} \mathfrak{g}, \mathfrak{g})$ vanishes identically. Hence the cohomology groups of the Nijenhuis operator $N = \mathrm{Id}_\mathfrak{g}$ are given by $H^n (\mathrm{Id}_\mathfrak{g}) = \mathrm{Hom} (\wedge^n \mathfrak{g}, \mathfrak{g})$, for all $n$. This shows that the cohomology of the identity map $\mathrm{Id}_\mathfrak{g}$ (viewed as a Nijenhuis operator) is independent of the Lie bracket of $\mathfrak{g}$. This is much expected as the identity map $\mathrm{Id}_\mathfrak{g}$ is a Nijenhuis operator for any Lie algebra structure on $\mathfrak{g}$.
\end{exam}

It is important to remark that the authors in \cite{sheng-o} have introduced the cohomology of a (relative) Rota-Baxter operator generalizing the well-known cohomology of a classical {\bf r}-matrix. Here, we shall show that the cochain complex associated with a relative Rota-Baxter operator can be seen as a subcomplex of the cochain complex associated to a suitable Nijenhuis operator. Let $(\mathfrak{g}, [~,~]_\mathfrak{g})$ be a Lie algebra and $(\mathcal{V}, \rho)$ be a representation of it. First, recall that a {\em relative Rota-Baxter operator} (also called an {\em $\mathcal{O}$-operator}) is a linear map $r : \mathcal{V} \rightarrow \mathfrak{g}$ that satisfies
\begin{align}\label{rrb}
    [r(u), r(v) ]_\mathfrak{g} = r (\rho_{r(u)} v - \rho_{r(v)} u), \text{ for } u, v \in \mathcal{V}.
\end{align}
The cochain complex associated to the relative Rota-Baxter operator $r$ is given by $\{ \oplus_{n=0}^\infty \mathrm{Hom} (\wedge^n \mathcal{V}, \mathfrak{g}), d_r \}$, where
\begin{align}
    d_r (x) (v) :=~& [r(v) , x]_\mathfrak{g} + r (\rho_x v), \label{dr-1}\\
    (d_r f)(v_1, \ldots, v_{n+1} ) :=~& \sum_{i=1}^{n+1} (-1)^{i+1} \big\{ [r(v_i), f (v_1, \ldots, \widehat{v_i}, \ldots, v_{n+1})]_\mathfrak{g}  + r ( \rho_{ f (v_1, \ldots, \widehat{v_i}, \ldots, v_{n+1})   } v_i )   \big\} \label{dr-2} \\
   & + \sum_{1 \leq i < j \leq n+1} (-1)^{i+j} ~ \! f \big(  \rho_{ r (v_i)} v_j - \rho_{r (v_j)} v_i, v_1, \ldots, \widehat{v_i}, \ldots, \widehat{v_j}, \ldots, v_{n+1}    \big), \nonumber
\end{align}
for $x \in \mathfrak{g}$, $f \in \mathrm{Hom} (\wedge^n \mathcal{V}, \mathfrak{g})$ and $v, v_1, \ldots, v_{n+1} \in \mathcal{V}$. On the other hand, given a relative Rota-Baxter operator $r: \mathcal{V} \rightarrow \mathfrak{g}$, its lift $\widetilde{r}: \mathfrak{g} \oplus \mathcal{V} \rightarrow \mathfrak{g} \oplus \mathcal{V}$ defined by $\widetilde{r} (x, v) = (r(v), 0)$, for $(x, v) \in \mathfrak{g} \oplus \mathcal{V}$ is a Nijenhuis operator on the semidirect product Lie algebra $(\mathfrak{g} \oplus \mathcal{V}, [~,~]_\ltimes)$, where the Lie bracket $[~,~]_\ltimes$ is given by
\begin{align}\label{semid}
    [(x, u), (y, v)]_\ltimes = ([x, y]_\mathfrak{g} ~\! , ~\! \rho_x v - \rho_y u), \text{ for } (x, u) , (y, v) \in \mathfrak{g} \oplus \mathcal{V}.
\end{align}
 Hence one may consider the cochain complex $\{ \oplus_{n=0}^\infty \mathrm{Hom} (\wedge^n (\mathfrak{g} \oplus \mathcal{V}), \mathfrak{g} \oplus \mathcal{V} ), ~ \! d_{\widetilde{r}} \}$ associated to the Nijenhuis operator $\widetilde{r}$. Then it can be easily checked that 
\begin{align*}
    d_{\widetilde{r}} \big(  \mathrm{Hom} (\wedge^n \mathcal{V}, \mathfrak{g})   \big) \subset  \mathrm{Hom} (\wedge^{n+1} \mathcal{V}, \mathfrak{g}), \text{ for all } n.
\end{align*}
Further, while restricting to the space $\mathrm{Hom} (\wedge^n \mathcal{V}, \mathfrak{g})$, the map $d_{\widetilde{r}}$ coincides with the map $d_r$ given in (\ref{dr-1}), (\ref{dr-2}). Hence the cochain complex $\{  \oplus_{n=0}^\infty \mathrm{Hom} (\wedge^n \mathcal{V}, \mathfrak{g}), ~\!  d_r  \}$ associated to the relative Rota-Baxter operator $r$ is a subcomplex of the cochain complex $\{ \oplus_{n=0}^\infty \mathrm{Hom} (\wedge^n (\mathfrak{g} \oplus \mathcal{V}), \mathfrak{g} \oplus \mathcal{V} ), ~ \! d_{\widetilde{r}} \}$ associated to the Nijenhuis operator $\widetilde{r}$.

\medskip

Let $(\mathfrak{g}, [~, ~]_\mathfrak{g})$ be a Lie algebra and $N: \mathfrak{g} \rightarrow \mathfrak{g}$ be a Nijenhuis operator on it. Then it is easy to see from the expressions (\ref{dn-exp}) and (\ref{dn0})  that the coboundary map $d_N$ cannot be expressed as the Chevalley-Eilenberg coboundary operator of the deformed Lie algebra $\mathfrak{g}^N = (\mathfrak{g}, [~,~]_\mathfrak{g}^N)$ with coefficients in any representation. In particular, the cohomology of the Nijenhuis operator $N$ is not the same as the Chevalley-Eilenberg cohomology of the deformed Lie algebra with coefficients in the adjoint representation. However, in the next result, we show that there is a homomorphism from the cohomology of the Nijenhuis operator $N$ to the cohomology of the deformed Lie algebra. For each $n \geq 0$, we first define a map 
\begin{align*}
    \Phi_n : \mathrm{Hom} (\wedge^n \mathfrak{g}, \mathfrak{g}) \rightarrow \mathrm{Hom} (\wedge^{n+1} \mathfrak{g}, \mathfrak{g})  ~~~\text{ by } ~~~ \Phi_n (f) := (-1)^{n+1} ~\! \delta_{\mathrm{CE}} (f), \text{ for } f \in \mathrm{Hom} (\wedge^n \mathfrak{g}, \mathfrak{g}),
\end{align*}
where $\delta_\mathrm{CE}$ is the Chevalley-Eilenberg coboundary operator of the Lie algebra $(\mathfrak{g}, [~, ~]_\mathfrak{g})$ with coefficients in the adjoint representation.

\begin{prop}\label{prop-mor}
    The collection of maps $\{ \Phi_n \}_{n=0}^\infty$ satisfy $\delta^N_\mathrm{CE} \circ \Phi_n = \Phi_{n+1} \circ d_N$, for all $n$, where $\delta_\mathrm{CE}^N$ is the Chevalley-Eilenberg coboundary operator of the deformed Lie algebra $\mathfrak{g}^N = (\mathfrak{g}, [~, ~]_\mathfrak{g}^N)$ with coefficients in the adjoint representation. As a consequence, there is a homomorphism $ H^\bullet (N) \rightarrow H^{\bullet +1}_\mathrm{CE} (\mathfrak{g}^N)$ from the cohomology of the Nijenhuis operator $N$ to the cohomology of the deformed Lie algebra. 
\end{prop}

\begin{proof}
    For any $f \in \mathrm{Hom} (\wedge^n \mathfrak{g}, \mathfrak{g})$, we have
    \begin{align*}
        (\Phi_{n+1} \circ d_N) (f) = \Phi_{n+1} ([N, f]_\mathsf{FN}) =~& (-1)^{n+2} ~\! \delta_\mathrm{CE} ([N, f]_\mathsf{FN}) \\
        =~& (-1)^{n+2} ~\! [\delta_\mathrm{CE} N, \delta_\mathrm{CE} f]_\mathsf{NR} \quad (\text{by } (\ref{bracket-pres}))  \\
        =~& - [\mu^N , (-1)^{n+1} ~\! \delta_\mathrm{CE} f]_\mathsf{NR} \\ =~& - [\mu^N, \Phi_n (f)]_\mathsf{NR} = (\delta_\mathrm{CE}^N \circ \Phi_n) (f).
    \end{align*}
    Here $\mu^N = \delta_\mathrm{CE} N \in \mathrm{Hom}(\wedge^2 \mathfrak{g}, \mathfrak{g})$ is the element corresponding to the deformed Lie bracket $[~,~]_\mathfrak{g}^N$, and $\delta_\mathrm{CE}^N$ is the Chevalley-Eilenberg coboundary operator of the deformed Lie algebra with coefficients in the adjoint representation.
\end{proof}

Given a Lie algebra endowed with a Nijenhuis operator, next, we consider deformations of the Nijenhuis operator, keeping the underlying Lie algebra intact. Let $(\mathfrak{g}, [~, ~]_\mathfrak{g})$ be a Lie algebra and $N : \mathfrak{g} \rightarrow \mathfrak{g}$ be a Nijenhuis operator on it. Then we have seen that $N$ is a Maurer-Cartan element in the Fr\"{o}licher-Nijenhuis graded Lie algebra $ \big(  \oplus_{n=1}^\infty \mathrm{Hom} (\wedge^n \mathfrak{g}, \mathfrak{g}), [~,~]_\mathsf{FN} \big).$ As a result, the triple  $ \big(  \oplus_{n=1}^\infty \mathrm{Hom} (\wedge^n \mathfrak{g}, \mathfrak{g}), [~,~]_\mathsf{FN} , d_N \big)$ is a differential graded Lie algebra. The following result shows that this differential graded Lie algebra controls the linear deformations of the Nijenhuis operator $N$.

\begin{prop}
    Let $(\mathfrak{g}, [~,~]_\mathfrak{g})$ be a Lie algebra and $N : \mathfrak{g} \rightarrow \mathfrak{g}$ be a Nijenhuis operator on it. Then for any linear map $N': \mathfrak{g} \rightarrow \mathfrak{g}$, the sum $N + N'$ is also a Nijenhuis operator on the Lie algebra $(\mathfrak{g}, [~,~]_\mathfrak{g})$ if and only if $N'$ is a Maurer-Cartan element in the differential graded Lie algebra  $ \big(  \oplus_{n=1}^\infty \mathrm{Hom} (\wedge^n \mathfrak{g}, \mathfrak{g}), [~,~]_\mathsf{FN} , d_N \big)$.
\end{prop}

\begin{proof}
    We observe that
    \begin{align*}
        [N + N', N + N']_\mathsf{FN} =~& [N, N ]_\mathsf{FN} + [N, N']_\mathsf{FN} + [N', N]_\mathsf{FN} + [N', N']_\mathsf{FN} \\
        =~& 2 [N, N']_\mathsf{FN} + [N', N']_\mathsf{FN} = 2 \big(  [N, N']_\mathsf{FN} + \frac{1}{2} [N', N']_\mathsf{FN}  \big).
    \end{align*}
    This shows that $[N + N', N + N']_\mathsf{FN} = 0$ if and only if $N'$ is a Maurer-Cartan element in the above differential graded Lie algebra.
\end{proof}

In the following, we consider finite order deformations of a Nijenhuis operator $N$ and investigate the criterion for their extensions. Let $(\mathfrak{g}, [~, ~]_\mathfrak{g})$ be a Lie algebra and $N : \mathfrak{g} \rightarrow \mathfrak{g}$ be a Nijenhuis operator on it. For any $n \in \{ 1, 2, \ldots \} $, consider the space $\mathfrak{g} [[t]] / (t^{n+1})$ of all polynomials in the variable $t$ of degree $\leq n$ with coefficients from $\mathfrak{g}$. Then $\mathfrak{g} [[t]] / (t^{n+1})$ is obviously a ${\bf k} [[t]]/(t^{n+1})$-module. Note that the Lie bracket $[~, ~ ]_\mathfrak{g} : \mathfrak{g} \times \mathfrak{g} \rightarrow \mathfrak{g}$  on $\mathfrak{g}$ can be extended to a bracket on $\mathfrak{g} [[t]] / (t^{n+1})$ simply by using ${\bf k} [[t]]/(t^{n+1})$-bilinearity. We denote the extended bracket on $\mathfrak{g} [[t]] / (t^{n+1})$ by the same notation $[~,~]_\mathfrak{g}$. Then it turns out that $ \big( \mathfrak{g} [[t]] / (t^{n+1}) , [~,~ ]_\mathfrak{g} \big)$ is a Lie algebra in the category of ${\bf k} [[t]]/(t^{n+1})$-modules.

\begin{defn}
    (i) A {\bf deformation of order $n$} of the Nijenhuis operator $N$ is a sum
    \begin{align*}
        N_t^n := N_0 + t N_1 + \cdots + t^n N_n \in \mathrm{Hom} (\mathfrak{g}, \mathfrak{g}) [[t]] / (t^{n+1}) ~~ \text{ with } N_0 = N
    \end{align*}
    such that its ${\bf k} [[t]]/(t^{n+1})$-linear extension (denoted by the same notation) $N_t : \mathfrak{g} [[t]] / (t^{n+1}) \rightarrow \mathfrak{g} [[t]] / (t^{n+1})$ is a Nijenhuis operator on the Lie algebra $ \big( \mathfrak{g} [[t]] / (t^{n+1}) , [~,~ ]_\mathfrak{g} \big)$ in the category of ${\bf k} [[t]]/(t^{n+1})$-modules.

    \medskip

    (ii) Let $N_t^n := N_0 + t N_1 + \cdots + t^n N_n$ be a deformation of order $n$ of the Nijenhuis operator $N$. It is said to be {\bf extensible} if there exists a linear map $N_{n+1} : \mathfrak{g} \rightarrow \mathfrak{g}$ such that the sum 
    \begin{align*}
            N^{n+1}_t := N_t^n + t^{n+1} N_{n+1} = N_0 + tN_1 + \cdots + t^n N_n + t^{n+1} N_{n+1}
    \end{align*}
    defines a deformation of order $n+1$ of the Nijenhuis operator $N$.
\end{defn}

Let $N_t^n = \sum_{i=0}^n t^i N_i$ be a deformation of order $n$. Then for any $0 \leq p \leq n$, we have $\sum_{\substack{i+ j = p \\ i , j \geq 0}} [N_i, N_j ]_\mathsf{FN} = 0$ or equivalently,
\begin{align}\label{def-def}
    d_N (N_p) = - \frac{1}{2} \sum_{\substack{i+j = p \\ i , j \geq 1}} [N_i ,N_j]_\mathsf{FN}.
\end{align}
We now define an element $\mathrm{Ob} (N_t^n) \in \mathrm{Hom} (\wedge^2 \mathfrak{g}, \mathfrak{g})$ by
  $  \mathrm{Ob} (N_t^n) := - \frac{1}{2} \sum_{\substack{i+j = n+1 \\ i, j \geq 1}} [N_i, N_j]_\mathsf{FN}.$
Note that $\mathrm{Ob} (N^n_t)$ depends only on the deformation $N_t^n$. Moreover, we observe that
\begin{align*}
    d_N \big(   \mathrm{Ob} (N^n_t)  \big) =~& - \frac{1}{2} \sum_{\substack{i+j = n+1 \\ i , j \geq 1}} [ N, [N_i, N_j]_\mathsf{FN} ]_\mathsf{FN} \\
    =~& - \frac{1}{2} \sum_{\substack{i+j = n+1 \\ i , j \geq 1}} \big(    [[N, N_i]_\mathsf{FN} , N_j]_\mathsf{FN} - [N_i , [N, N_j]_\mathsf{FN} ]_\mathsf{FN} \big) \\
    =~& \frac{1}{4} \sum_{\substack{i_1 + i_2 + j = n+1 \\ i_1, i_2 , j \geq 1}} [[N_{i_1}, N_{i_2}]_\mathsf{FN} , N_j]_\mathsf{FN} - \frac{1}{4} \sum_{\substack{i+ j_1 + j_2 = n+1 \\ i, j_1, j_2 \geq 1}} [N_i , [N_{j_1} , N_{j_2} ]_\mathsf{FN} ]_\mathsf{FN} \quad (\text{by } (\ref{def-def})) \\
    =~& \frac{1}{2} \sum_{\substack{i+j + k = n \\ i, j,  k \geq 1}} [[N_i, N_j ]_\mathsf{FN} , N_k]_\mathsf{FN} = 0.
\end{align*}
This shows that $\mathrm{Ob} (N_t^n)$ is a $2$-cocycle in the cochain complex associated to the Nijenhuis operator $N$. Hence it gives rise to a cohomology class $[ \mathrm{Ob} (N_t^n)] \in H^2 (N)$, called the {\em obstruction class}. Then we have the following.

\begin{thm}
    Let $N_t^n$ be a deformation of order $n$ of the Nijenhuis operator $N$. Then it is extensible if and only if the corresponding obstruction class  $[ \mathrm{Ob} (N_t^n)] $ vanishes. In particular, if $H^2 (N) = 0$ then any finite order deformation of $N$ is extensible.
\end{thm}

\begin{proof}
    Suppose $N_t^n$ is extensible. Then there exists a linear map $N_{n+1} : \mathfrak{g} \rightarrow \mathfrak{g}$ such that the sum $N^{n+1}_t := N^n_t + t^{n+1} N_{n+1}$ is a deformation of order $n+1$. Therefore, we get that $\sum_{\substack{i+j = n+1 \\ i , j \geq 0}} [N_i, N_j ]_\mathsf{FN} = 0$ which in turn implies that $d_N (N_{n+1}) = - \frac{1}{2} \sum_{ \substack{i+j = n+1 \\ i , j \geq 1}} [N_i , N_j ]_\mathsf{FN}$. This shows that $\mathrm{Ob} (N_t^n)$ is a coboundary and hence the corresponding cohomology class vanishes.

    Conversely, suppose the cohomology class $[\mathrm{Ob} (N_t^n)]$ is trivial. Then we have $\mathrm{Ob} (N_t^n) = d_N (N_{n+1})$, for some linear map $N_{n+1} : \mathfrak{g} \rightarrow \mathfrak{g}$. As a result, we obtain that $N_t^{n+1} := N_t^n + t^{n+1} N_{n+1}$ is a deformation of order $n+1$. Hence $N^n_t$ is extensible. 
\end{proof}

\subsection{Representations and cohomology of Nijenhuis Lie algebras} In this subsection, we aim to define the cohomology of a Nijenhuis Lie algebra with coefficients in a suitable representation.

\begin{defn}\label{defn-nla}
    A {\bf Nijenhuis Lie algebra} is a Lie algebra $(\mathfrak{g}, [~, ~]_\mathfrak{g})$ endowed with a distinguished Nijenhuis operator $N : \mathfrak{g} \rightarrow \mathfrak{g}$ on it. We denote a Nijenhuis Lie algebra as above simply by the triple $(\mathfrak{g}, [~, ~]_\mathfrak{g}, N)$.
\end{defn}

Let $(\mathfrak{g}, [~, ~]_\mathfrak{g}, N)$ and $(\mathfrak{h}, [~, ~]_\mathfrak{h}, S)$ be two Nijenhuis Lie algebras. A {\em homomorphism} of Nijenhuis Lie algebras from $(\mathfrak{g}, [~, ~]_\mathfrak{g}, N)$ to $(\mathfrak{h}, [~, ~]_\mathfrak{h}, S)$ is a Lie algebra homomorphism $\varphi : \mathfrak{g} \rightarrow \mathfrak{h}$ satisfying additionally $S \circ \varphi = \varphi \circ N$. Further, it is said to be an isomorphism of Nijenhuis Lie algebras if $\varphi$ is also bijective.

\begin{defn}
    Let $(\mathfrak{g}, [~, ~]_\mathfrak{g}, N)$ be a Nijenhuis Lie algebra. A {\bf Nijenhuis representation} of this Nijenhuis Lie algebra is a triple $(\mathcal{V}, \rho, S)$, where $(\mathcal{V}, \rho)$ is a usual representation of the Lie algebra $(\mathfrak{g}, [~, ~]_\mathfrak{g})$ and $S : \mathcal{V} \rightarrow \mathcal{V}$ is a linear map satisfying
    \begin{align*}
        \rho_{N(x)} S (v) = S  ( \rho_{N (x)} v + \rho_x S (v) - S (\rho_x v)), \text{ for } x \in \mathfrak{g}, v \in \mathcal{V}.
    \end{align*}
\end{defn}

\begin{exam}
    Let $(\mathfrak{g}, [~,~]_\mathfrak{g}, N)$ be a Nijenhuis Lie algebra.
    \begin{itemize}
        \item[(i)] Then the triple $(\mathfrak{g}, \mathrm{ad}_\mathfrak{g}, N)$ is a Nijenhuis representation of it, where $\mathrm{ad}_\mathfrak{g} : \mathfrak{g} \rightarrow \mathrm{End}(\mathfrak{g})$ is the adjoint representation given by $(\mathrm{ad}_\mathfrak{g})_x (y) = [x, y]_\mathfrak{g}$, for $x, y \in \mathfrak{g}$. This is called the {\em adjoint Nijenhuis representation} of the Nijenhuis Lie algebra $(\mathfrak{g}, [~,~]_\mathfrak{g}, N)$.
        \item[(ii)] For any representation $(\mathcal{V}, \rho)$ of the Lie algebra $(\mathfrak{g}, [~,~]_\mathfrak{g})$, the triples $(\mathcal{V}, \rho, 0)$, $(\mathcal{V}, \rho, \mathrm{Id}_\mathcal{V})$ and $(\mathcal{V}, \rho,  - \mathrm{Id}_\mathcal{V})$ are all Nijenhuis representations of the Nijenhuis Lie algebra $(\mathfrak{g}, [~,~]_\mathfrak{g}, N)$.
    \end{itemize}
\end{exam}

\begin{exam}\label{power-repn}
    Let $(\mathfrak{g}, [~,~]_\mathfrak{g}, N)$ be a Nijenhuis Lie algebra and $(\mathcal{V}, \rho, S)$ be a Nijenhuis representation of it. Then for any $k \geq 0$, the triple $(\mathcal{V}, \rho, S^k)$ is a Nijenhuis representation of the Nijenhuis Lie algebra $(\mathfrak{g}, [~,~]_\mathfrak{g}, N^k)$. More generally, for any $k, l \geq 0$, the triple $(\mathcal{V}, \rho^l, S^k)$ is a Nijenhuis representation of the Nijenhuis Lie algebra $(\mathfrak{g}, [~,~]_\mathfrak{g}^{N^l}, N^k)$, where 
    \begin{align*}
        \rho^l_x (v) := \rho_{N^l (x)} v + \rho_x S^l (v) - S^l (\rho_x v), \text{ for } x \in \mathfrak{g}, v \in \mathcal{V}.
    \end{align*}
\end{exam}

Given a Nijenhuis Lie algebra and a Nijenhuis representation, one may construct the semidirect product Nijenhuis Lie algebra. The precise statement is given below.

\begin{prop}\label{prop-semid-nlie}
    Let $(\mathfrak{g}, [~,~]_\mathfrak{g}, N)$ be a Nijenhuis Lie algebra and $(\mathcal{V}, \rho, S)$ be a Nijenhuis representation of it. Then $(\mathfrak{g} \oplus \mathcal{V}, [~,~]_\ltimes, N \oplus S)$ is a Nijenhuis Lie algebra, where $[~,~]_\ltimes$ is the semidirect product Lie bracket on $\mathfrak{g} \oplus \mathcal{V}$ given in (\ref{semid}).
\end{prop}

Let $(\mathfrak{g}, [~,~]_\mathfrak{g})$ be a Lie algebra and $(\mathcal{V}, \rho)$ be any representation of it. Define $\rho^* : \mathfrak{g} \rightarrow \mathrm{End}(\mathcal{V}^*)$ by
\begin{align*}
    (\rho^*_x \alpha) (v) = - \langle \alpha, \rho_x v \rangle, \text{ for } x \in \mathfrak{g}, \alpha \in \mathcal{V}^*, v \in \mathcal{V}.
\end{align*}
Then $(\mathcal{V}^*, \rho^*)$ is also a representation of the Lie algebra $(\mathfrak{g}, [~,~]_\mathfrak{g})$. This representation is called the dual of the representation $(\mathcal{V}, \rho)$. However, if we have a Nijenhuis representation $(\mathcal{V}, \rho, S)$ of a Nijenhuis Lie algebra $(\mathfrak{g}, [~,~]_\mathfrak{g}, N)$, then $(\mathcal{V}^*, \rho^*, S^*)$ need not be a Nijenhuis representation in general. Let $(\mathfrak{g}, [~, ~]_\mathfrak{g}, N)$ be a Nijenhuis Lie algebra and $(\mathcal{V}, \rho)$ be a representation of the underlying Lie algebra $(\mathfrak{g}, [~,~]_\mathfrak{g})$. A linear map $S : \mathcal{V} \rightarrow \mathcal{V}$ is said to be {\bf admissible} to the Nijenhuis Lie algebra $(\mathfrak{g}, [~,~]_\mathfrak{g}, N)$ and the Lie algebra representation $(\mathcal{V}, \rho)$ if 
\begin{align*}
    S (\rho_{N (x)} v) + \rho_x S^2 (v) = \rho_{N (x)} S (v) + S (\rho_x S (v)), \text{ for all } x \in \mathfrak{g}, v \in \mathcal{V}.
\end{align*}
Then the triple $(\mathcal{V}^*, \rho^*, S^*)$ is a Nijenhuis representation of the Nijenhuis Lie algebra $(\mathfrak{g}, [~,~]_\mathfrak{g}, N)$. In particular, suppse a linear map $S : \mathfrak{g} \rightarrow \mathfrak{g}$ satisfies
\begin{align}\label{adm-first}
  S [N(x), y]_\mathfrak{g} + [x, S^2 (y)]_\mathfrak{g} = [ N (x), S (y)]_\mathfrak{g} + S [x, S(y)]_\mathfrak{g}, \text{ for } x, y \in \mathfrak{g}
\end{align}
(i.e. $S$ is admissible to the Nijenhuis Lie algebra $(\mathfrak{g}, [~, ~]_\mathfrak{g}, N)$ and the adjoint Lie algebra representation $(\mathfrak{g}, \mathrm{ad}_\mathfrak{g})$) then the triple $(\mathfrak{g}^*, \mathrm{ad}^*_\mathfrak{g}, S^*)$ is a Nijenhuis representation of the Nijenhuis Lie algebra $(\mathfrak{g}, [~,~]_\mathfrak{g}, N)$. In this case, we simply say that $S$ is admissible to the Nijenhuis Lie algebra $(\mathfrak{g}, [~,~]_\mathfrak{g}, N)$.

\medskip

We will now define the cohomology of a Nijenhuis Lie algebra (with coefficients in the adjoint Nijenhuis representation). First, given a Nijenhuis Lie algebra $(\mathfrak{g}, [~, ~]_\mathfrak{g}, N)$, we observe that there are two important cochain complexes, namely,
\begin{itemize}
    \item the Chevalley-Eilenberg cochain complex $\{ \oplus_{n=0}^\infty \mathrm{Hom} (\wedge^n \mathfrak{g}, \mathfrak{g}), \delta_\mathrm{CE} \}$ of the Lie algebra $(\mathfrak{g}, [~, ~]_\mathfrak{g})$ with coefficients in the adjoint representation,
    \item the cochain complex $\{ \oplus_{n=0}^\infty \mathrm{Hom} (\wedge^n \mathfrak{g}, \mathfrak{g}), d_N \}$ of the Nijenhuis operator $N$ defined on the Lie algebra $(\mathfrak{g}, [~,~]_\mathfrak{g})$.
\end{itemize}
For each $n \geq 0$, we now define a map $\partial^N : \mathrm{Hom} (\wedge^n \mathfrak{g}, \mathfrak{g}) \rightarrow \mathrm{Hom} (\wedge^n \mathfrak{g}, \mathfrak{g})$ by
\begin{align*}
    & \qquad \qquad \partial^N (x) = x, \text{ for } x \in \mathfrak{g},\\
    (\partial^N f) &(x_1, \ldots, x_n) = f (N (x_1), \ldots, N (x_n )) - \sum_{i=1}^n N \big(   f ( N (x_1), \ldots, x_i, \ldots, N (x_n)) \big) \\
   & + \sum_{1 \leq i < j \leq n} N^2 \big(    f (  N (x_1), \ldots, x_i, \ldots, x_j, \ldots, N (x_n)    )  \big) - \cdots + (-1)^n N^n (f (x_1, \ldots, x_n)),
\end{align*}
for $f \in \mathrm{Hom} (\wedge^n \mathfrak{g}, \mathfrak{g})$ and $x_1, \ldots, x_n \in \mathfrak{g}$. Then for any $f \in \mathrm{Hom} (\wedge^n \mathfrak{g}, \mathfrak{g})$ with $n \geq 0$, it is straightforward to verify that 
\begin{align}\label{useful}
   ( d_N \circ \partial^N )(f) = (\partial^N \circ \delta_\mathrm{CE}) (f).
\end{align}
This shows that the collection of maps $\{ \partial^N : \mathrm{Hom} (\wedge^n \mathfrak{g}, \mathfrak{g}) \rightarrow \mathrm{Hom} (\wedge^n \mathfrak{g}, \mathfrak{g}) \}_{n \geq 0}$ defines a homomorphism of cochain complexes from $\{  \oplus_{n=0}^\infty \mathrm{Hom} (\wedge^n \mathfrak{g}, \mathfrak{g}), \delta_\mathrm{CE} \}$ to the complex $\{ \oplus_{n=0}^\infty \mathrm{Hom} (\wedge^n \mathfrak{g}, \mathfrak{g}), d_N \}$. We will now consider the mapping cone induced by this homomorphism to define the cochain complex of the given Nijenhuis Lie algebra $(\mathfrak{g}, [~,~]_\mathfrak{g}, N)$. More precisely, we set
\begin{align*}
    C^n_\mathrm{NLie} (\mathfrak{g}, N) := \begin{cases} 0  & \text{ if } n =0,\\
    \mathrm{Hom} (\mathfrak{g}, \mathfrak{g}) & \text{ if } n =1,\\
    \mathrm{Hom} (\wedge^n \mathfrak{g}, \mathfrak{g}) \oplus \mathrm{Hom} (\wedge^{n-1} \mathfrak{g}, \mathfrak{g})  & \text{ if } n \geq 2
    \end{cases}
\end{align*}
and define a map $\delta_\mathrm{NLie} :  C^n_\mathrm{NLie} (\mathfrak{g}, N) \rightarrow  C^{n+1}_\mathrm{NLie} (\mathfrak{g}, N)$ by
\begin{align*}
    \delta_\mathrm{NLie} (f) := \big(    \delta_\mathrm{CE} (f) ~ \! , ~\! - \partial^N (f) \big)    ~~~ \text{ and } ~~~ \delta_\mathrm{NLie} (\chi, F) := \big(   \delta_\mathrm{CE} (\chi) ~ \! , ~\! d_N (F) + (-1)^n ~\! \partial^N (\chi)  \big),
\end{align*}
for $f \in C^1_\mathrm{NLie} (\mathfrak{g}, N) = \mathrm{Hom} (\mathfrak{g}, \mathfrak{g})$ and $(\chi, F) \in C^{n \geq 2}_\mathrm{NLie} (\mathfrak{g}, N) = \mathrm{Hom} (\wedge^n \mathfrak{g}, \mathfrak{g}) \oplus \mathrm{Hom} (\wedge^{n-1} \mathfrak{g}, \mathfrak{g})$. Then we have the following.

  \begin{prop}\label{d-nlie}
  The map $\delta_\mathrm{NLie}$ is a differential, i.e. $(\delta_\mathrm{NLie})^2 = 0$.
\end{prop}

\begin{proof}
    For any $f \in C^1_\mathrm{NLie} (\mathfrak{g}, N),$ we have
    \begin{align*}
        (\delta_\mathrm{NLie})^2 (f) =~& \delta_\mathrm{NLie} (\delta_\mathrm{CE} (f) , ~ \! - \partial^N (f) ) \\
        =~& \big(   (\delta_\mathrm{CE})^2 f ~ \! , ~ \! - (d_N \circ \partial^N) (f) + (\partial^N \circ \delta_\mathrm{CE}) (f)  \big) = 0 \quad  (\text{since }(\delta_\mathrm{CE})^2 = 0 \text{ and by using } (\ref{useful})).
    \end{align*}
    On the other hand, if $(\chi, F) \in C^{n \geq 2}_\mathrm{NLie} (\mathfrak{g}, N)$ then
    \begin{align*}
        (\delta_\mathrm{NLie})^2 (\chi, F) =~& \big(  \delta_\mathrm{CE} (\chi) ~ \! , ~ \! d_N (F) + (-1)^n ~ \! \partial^N (\chi)   \big) \\
        =~& \big(  (\delta_\mathrm{CE})^2 \chi ~\!, ~\! (d_N)^2 F + (-1)^n (d_N \circ \partial^N ) (\chi) + (-1)^{n+1} (\partial^N \circ \delta_\mathrm{CE}) (\chi)   \big) \\
        =~& 0 \quad  (\text{since } (\delta_\mathrm{CE})^2 = 0, (d_N)^2 = 0 \text{ and by using } (\ref{useful})).
    \end{align*}
    This completes the proof.
\end{proof}

It follows from Proposition \ref{d-nlie}  that $\{ \oplus_{n=0}^\infty C^n_\mathrm{NLie} (\mathfrak{g}, N), \delta_\mathrm{NLie} \}$ is a cochain complex. The corresponding cohomology groups are denoted by $H^\bullet_\mathrm{NLie} (\mathfrak{g}, N)$. They are said to be the {\bf cohomology groups} of the Nijenhuis Lie algebra $(\mathfrak{g}, [~, ~]_\mathfrak{g}, N)$ with coefficients in the adjoint Nijenhuis representation.

\medskip

In the following, we shall generalize the above construction to define the cohomology of a Nijenhuis Lie algebra with coefficients in a given Nijenhuis representation. Let $(\mathfrak{g}, [~, ~]_\mathfrak{g}, N)$ be a Nijenhuis Lie algebra and $(\mathcal{V}, \rho, S)$ be a Nijenhuis representation of it. At first, we consider the Chevalley-Eilenberg cochain complex $\{ \oplus_{n=0}^\infty \mathrm{Hom} (\wedge^n \mathfrak{g}, \mathcal{V}), \delta_\mathrm{CE} \}$ of the Lie algebra $(\mathfrak{g}, [~, ~]_\mathfrak{g})$ with coefficients in the representation $(\mathcal{V}, \rho)$, where the map $\delta_\mathrm{CE} : \mathrm{Hom} (\wedge^n \mathfrak{g}, \mathcal{V}) \rightarrow \mathrm{Hom} (\wedge^{n+1} \mathfrak{g}, \mathcal{V})$ is given in (\ref{ce-diff}). On the other hand, for each $n \geq 0$, we define another map $d_{N, S } : \mathrm{Hom} (\wedge^n \mathfrak{g}, \mathcal{V}) \rightarrow \mathrm{Hom} (\wedge^{n+1} \mathfrak{g}, \mathcal{V})$ by
\begin{align*}
   & \qquad \qquad \qquad  (d_{N, S} v)(x) = \rho_{N (x)} v - S (\rho_x v),\\
    &(d_{N, S} f) (x_1, \ldots, x_{n+1}) = \sum_{i=1}^{n+1} (-1)^{i+1} \rho_{N (x_i)} f (x_1, \ldots, \widehat{x_i}, \ldots, x_{n+1}) \\
    &+ \sum_{1 \leq i < j \leq n+1} (-1)^{i+j} f (   [N(x_i), x_j]_\mathfrak{g} + [x_i, N (x_j)]_\mathfrak{g} - N [x_i, x_j]_\mathfrak{g}, x_1, \ldots, \widehat{x_i}, \ldots, \widehat{x_j}, \ldots, x_{n+1}) \nonumber \\
    &- S \big(   \sum_{i=1}^{n+1} (-1)^{i+1}  \rho_{x_i} f (x_1, \ldots, \widehat{x_i}, \ldots, x_{n+1})
    + \sum_{1 \leq i < j \leq n+1} (-1)^{i+j} f ([x_i, x_j]_\mathfrak{g}, x_1, \ldots, \widehat{x_i}, \ldots, \widehat{x_j}, \ldots, x_{n+1})   \big), \nonumber 
\end{align*}
for $v \in \mathcal{V}$, $f \in \mathrm{Hom} (\wedge^n \mathfrak{g}, \mathcal{V})$ and $x, x_1, \ldots, x_{n+1} \in \mathfrak{g}$. Then we have the following result.
\begin{prop}
    The map $d_{N, S}$ is a differential, i.e. $(d_{N, S})^2 = 0$.
\end{prop}

\begin{proof}
Since $(\mathfrak{g}, [~, ~]_\mathfrak{g}, N)$ is a Nijenhuis Lie algebra and $(\mathcal{V}, \rho, S)$ is a Nijenhuis representation, it follows from Proposition \ref{prop-semid-nlie} that the triple $(\mathfrak{g} \oplus \mathcal{V}, [~, ~]_\ltimes, N \oplus S)$ is a Nijenhuis Lie algebra. In other words, the map $N \oplus S : \mathfrak{g} \oplus \mathcal{V} \rightarrow \mathfrak{g} \oplus \mathcal{V}$ is a Nijenhuis operator on the semidirect product Lie algebra $(\mathfrak{g} \oplus \mathcal{V}, [~, ~]_\ltimes)$. Hence one may consider the cochain complex $\{ \oplus_{n=0}^\infty \mathrm{Hom} ( \wedge^n (\mathfrak{g} \oplus \mathcal{V}), \mathfrak{g} \oplus \mathcal{V}) , d_{N \oplus S} \}$ associated to the Nijenhuis operator $N \oplus S$ on the semidirect product Lie algebra. Then it is easy to verify that the differential map $d_{N \oplus S}$ satisfies
\begin{align*}
    d_{N \oplus S} \big(  \mathrm{Hom} (\wedge^{n} \mathfrak{g}, \mathcal{V})   \big) \subset \mathrm{Hom} (\wedge^{n+1} \mathfrak{g}, \mathcal{V}), \text{ for all } n.
\end{align*}
Moreover, while restricting to the subspace $ \mathrm{Hom} (\wedge^{n} \mathfrak{g}, \mathcal{V})$, the map $d_{N \oplus S}$ coincides with the map $d_{N, S}$. As $(d_{N \oplus S})^2 = 0$, it follows that $(d_{N, S})^2 = 0$.
\end{proof}

It follows from the above proposition that $\{ \oplus_{n=0}^\infty \mathrm{Hom} (\wedge^n \mathfrak{g}, \mathcal{V}), d_{N, S} \}$ is a cochain complex. This can be regarded as the cochain complex of the Nijenhuis operator $N$ relative to the operator $S$. The corresponding cohomology groups are denoted by $H^\bullet (N; S)$.

\begin{remark}
    When $(\mathcal{V}, \rho, S) = (\mathfrak{g}, \mathrm{ad}_\mathfrak{g}, N)$ is the adjoint Nijenhuis representation, the above cochain complex $\{ \oplus_{n=0}^\infty \mathrm{Hom} (\wedge^n \mathfrak{g}, \mathcal{V}), d_{N, S} \}$ coincides with the cochain complex $\{ \oplus_{n=0}^\infty \mathrm{Hom} (\wedge^n \mathfrak{g}, \mathfrak{g}), d_N \}$ associated to the Nijenhuis operator $N$ (see Subsection \ref{subsec-31}).
\end{remark}

\begin{remark}
    Let $(\mathfrak{g}, [~, ~]_\mathfrak{g}, N)$ be a Nijenhuis Lie algebra and $(\mathcal{V}, \rho, S)$ be any Nijenhuis representation. Then we have seen in Example \ref{power-repn} that the deformed Lie algebra $\mathfrak{g}^N = (\mathfrak{g}, [~, ~]_\mathfrak{g}^N)$ has a representation on the vector space $\mathcal{V}$ with the action map $\rho^1 : \mathfrak{g}^N \rightarrow \mathrm{End} (\mathcal{V})$ given by
    \begin{align*}
        \rho_x^1  (v) := \rho_{N (x)} v + \rho_x S (v) - S (\rho_x v), \text{ for } x \in \mathfrak{g}^N, v \in \mathcal{V}.
    \end{align*}
    We denote this representation $(\mathcal{V}, \rho^1)$  simply by $\mathcal{V}^S$. Next, for each $n \geq 0$, we define a linear map 
    \begin{align*}
    \Phi_n : \mathrm{Hom} (\wedge^n \mathfrak{g}, \mathcal{V}) \rightarrow \mathrm{Hom} (\wedge^{n+1} \mathfrak{g}, \mathcal{V}) ~~~ \text{ by } ~~~ \Phi_n (f) := (-1)^{n+1} ~ \! \delta_\mathrm{CE} (f), \text{ for } f \in \mathrm{Hom} (\wedge^n \mathfrak{g}, \mathcal{V}).
    \end{align*}
    Then similar to Proposition \ref{prop-mor}, one can show that $\delta_\mathrm{CE}^{N, S} \circ \Phi_n = \Phi_{n+1} \circ d_{N, S}$, where $\delta_\mathrm{CE}^{N, S}$ is the Chevalley-Eilenberg coboundary operator of the deformed Lie algebra $\mathfrak{g}^N$ with coefficients in the representation $\mathcal{V}^S$. Hence there is a morphism $H^\bullet (N ; S) \rightarrow H^{\bullet +1}_\mathrm{CE} (\mathfrak{g}^N ; \mathcal{V}^S)$ at the level of cohomology groups.
\end{remark}

For each $n \geq 0$, we now define a map $\partial^{N,S} : \mathrm{Hom} (\wedge^n \mathfrak{g}, \mathcal{V}) \rightarrow  \mathrm{Hom} (\wedge^n \mathfrak{g}, \mathcal{V})$ by
\begin{align*}
     & \qquad \qquad \partial^{N, S} (v) = v, \text{ for } v \in \mathcal{V},\\
    (\partial^{N, S} f) &(x_1, \ldots, x_n) = f (N (x_1), \ldots, N (x_n )) - \sum_{i=1}^n S \big(   f ( N (x_1), \ldots, x_i, \ldots, N (x_n)) \big) \\
   & + \sum_{1 \leq i < j \leq n} S^2 \big(    f (  N (x_1), \ldots, x_i, \ldots, x_j, \ldots, N (x_n)    )  \big) - \cdots + (-1)^n S^n (f (x_1, \ldots, x_n)),
\end{align*}
for $f \in \mathrm{Hom} (\wedge^n \mathfrak{g}, \mathcal{V})$ and $x_1, \ldots, x_n \in \mathfrak{g}$. Then it turns out that $(d_{N, S} \circ \partial^{N, S}) (f) = (\partial^{N, S} \circ \delta_\mathrm{CE}) (f)$, for $f \in \mathrm{Hom} (\wedge^n \mathfrak{g}, \mathcal{V})$. As a result, we obtain the cochain complex $\{  \oplus_{n=0}^\infty C^n_\mathrm{NLie} ((\mathfrak{g}, N); (\mathcal{V}, S)), \delta_\mathrm{NLie}  \}$, where
\begin{align*}
    C^n_\mathrm{NLie} ((\mathfrak{g}, N); (\mathcal{V}, S)) = \begin{cases}
        0 & \text{ if } n =0,\\
        \mathrm{Hom} (\mathfrak{g}, \mathcal{V}) & \text{ if } n =1, \\
        \mathrm{Hom}(\wedge^n \mathfrak{g}, \mathcal{V}) \oplus \mathrm{Hom}(\wedge^{n-1} \mathfrak{g}, \mathcal{V}) & \text{ if } n \geq 2 \\
    \end{cases}
\end{align*}
and the map $\delta_\mathrm{NLie} : C^n_\mathrm{NLie} ((\mathfrak{g}, N); (\mathcal{V}, S)) \rightarrow C^{n+1}_\mathrm{NLie} ((\mathfrak{g}, N); (\mathcal{V}, S))$ is given by
\begin{align*}
\delta_\mathrm{NLie} (f) = (\delta_\mathrm{CE} (f) , ~ \! - \partial^{N, S} (f)) ~~~~ \text{ and } ~~~~ \delta_\mathrm{NLie} (\chi, F) = ( \delta_\mathrm{CE} (\chi), ~ \! d_{N, S} (F) + (-1)^n ~\! \partial^{N, S } (\chi)),
\end{align*}
for $f \in C^1_\mathrm{NLie} ((\mathfrak{g}, N); (\mathcal{V}, S))$ and $(\chi, F) \in C^{n \geq 2}_\mathrm{NLie} ((\mathfrak{g}, N); (\mathcal{V}, S)) $. Then the cohomology groups of the cochain complex $\{  \oplus_{n=0}^\infty   C^n_\mathrm{NLie} ((\mathfrak{g}, N); (\mathcal{V}, S)), \delta_\mathrm{NLie} \}$ are simply denoted by $H^\bullet_\mathrm{NLie}  ((\mathfrak{g}, N); (\mathcal{V}, S))$ which are called the {\bf cohomology groups} of the Nijenhuis Lie algebra $(\mathfrak{g}, [~, ~]_\mathfrak{g}, N)$ with coefficients in the Nijenhuis representation $(\mathcal{V}, \rho, S)$.

\medskip

It follows from the above definition that a pair $(\chi, F) \in C^2_\mathrm{NLie} ((\mathfrak{g}, N); (\mathcal{V}, S)) = \mathrm{Hom} (\wedge^2 \mathfrak{g}, \mathcal{V}) \oplus \mathrm{Hom} (\mathfrak{g}, \mathcal{V})$ is a $2$-cocycle if and only if $\delta_\mathrm{CE} (\chi) = 0$ and $d_{N, S} (F) + ~\!  \partial^{N, S } (\chi) = 0$. These two conditions can be explicitly written as 
\begin{align*}
    &\rho_x \chi (y, z) + \rho_y \chi (z, x) + \rho_z \chi (x, y) - \chi ([x, y]_\mathfrak{g}, z) - \chi ([y, z]_\mathfrak{g}, x) - \chi ([z, x]_\mathfrak{g}, y) = 0, \\\\ 
    \rho_{N (x)} F (y) ~\!  - ~\! & \rho_{N (y)} F (x) - F (   [N (x), y]_\mathfrak{g} + [x, N (y)]_\mathfrak{g} -N [x, y]_\mathfrak{g} ) - S (\rho_x F (y) - \rho_y F (x) - F [x, y]_\mathfrak{g}) \\
    &+ \chi (N (x), N (y)) - S \big( \chi (N (x), y) + \chi (x, N (y)) - S \chi (x, y) \big) = 0,
\end{align*}
for all $x, y , z \in \mathfrak{g}$. Further, a $2$-cocycle $(\chi, F)$ is a $2$-coboundary if there exists a linear map $\varphi \in \mathrm{Hom} (\mathfrak{g}, \mathcal{V})$ such that
\begin{align*}
    \chi (x, y) = \rho_x \varphi (y) - \rho_y \varphi (x) - \varphi ([x, y]_\mathfrak{g}) ~~~~ \text{ and } ~~~~ F (x) = ( S \circ \varphi - \varphi \circ N) (x), \text{ for all } x, y \in \mathfrak{g}.
\end{align*}

\medskip

In the following result, we connect the cohomology of a Nijenhuis Lie algebra $(\mathfrak{g}, [~, ~]_\mathfrak{g}, N)$ and the Chevalley-Eilenberg cohomology of the underlying Lie algebra $(\mathfrak{g}, [~, ~]_\mathfrak{g})$ and also the cohomology of the Nijenhuis operator $N$.

\begin{thm}
    Let $(\mathfrak{g}, [~, ~]_\mathfrak{g}, N)$ be a Nijenhuis Lie algebra and $(\mathcal{V}, \rho, S)$ be a Nijenhuis representation of it. Then there is a short exact sequence of cochain complexes
    \begin{align*}
        0 \rightarrow \{ \oplus_{n=2}^\infty \mathrm{Hom} (\wedge^{n -1} \mathfrak{g}, \mathcal{V}), d_{N, S} \} &\xrightarrow{i} \{ \oplus_{n=2}^\infty C^n_\mathrm{NLie} ((\mathfrak{g}, N); (\mathcal{V}, S)), \delta_{\mathrm{NLie}} \} \xrightarrow{p} \{ \oplus_{n=2}^\infty \mathrm{Hom} ( \wedge^{n} \mathfrak{g}, \mathcal{V}), \delta_\mathrm{CE} \} \rightarrow 0, \\
      &  \text{ where } i (F) = (0, F) \text{ and } p (\chi, F) = \chi
    \end{align*}
    which yields a long exact sequence in the cohomology:
    \begin{align*}
        H^2 (N; S) \rightarrow H^3_\mathrm{NLie} ((\mathfrak{g}, N); (\mathcal{V}, S)) \rightarrow H^3_\mathrm{CE} (\mathfrak{g}; \mathcal{V}) \rightarrow H^3 (N ; S) \rightarrow \cdots .
    \end{align*}
    In particular, when $(\mathcal{V}, \rho, S) = (\mathfrak{g}, \mathrm{ad}_\mathfrak{g}, N)$ is the adjoint Nijenhuis representation, we obtain the following long exact sequence connecting various cohomology groups:
    \begin{align*}
        H^2 (N) \rightarrow H^3_\mathrm{NLie} (\mathfrak{g}, N) \rightarrow H^3_\mathrm{CE} (\mathfrak{g}) \rightarrow H^3 (N) \rightarrow \cdots .
    \end{align*}
\end{thm}

\begin{remark}\label{mc-char}
    In this section, we mainly developed the cohomology of a Nijenhuis Lie algebra. For this, we essentially require the Chevalley-Eilenberg complex of the underlying Lie algebra and the cochain complex of the Nijenhuis operator. Although, both the above complexes carry graded Lie algebra structures, it is not clear how to obtain a higher structure (possibly an $L_\infty$-algebra) on the cochain complex of a Nijenhuis Lie algebra. In future work, we aim to find the Maurer-Cartan characterization of a Nijenhuis Lie algebra and obtain the higher structure on the cochain complex.
\end{remark}

\medskip

\medskip

\section{Deformations of Nijenhuis Lie algebras}\label{sec4}
In this section, we study formal and infinitesimal deformations of a Nijenhuis Lie algebra $(\mathfrak{g}, [~,~]_\mathfrak{g}, N)$ in terms of its cohomology. Among others, we show that the set of all equivalence classes of infinitesimal deformations of the Nijenhuis Lie algebra $(\mathfrak{g}, [~,~]_\mathfrak{g}, N)$ has a bijection with its second cohomology group.

\medskip

Let $\mathsf{R}$ be a commutative unital ring with unity $1_\mathsf{R}$. An {\em augmentation} of $\mathsf{R}$ is a homomorphism $\varepsilon : \mathsf{R} \rightarrow {\bf k}$ satisfying $\varepsilon (1_\mathsf{R}) = 1_{\bf k}$. In the following, we shall always assume that $\mathsf{R}$ is a commutative unital ring with an augmentation $\varepsilon$. A Nijenhuis Lie algebra in the category of $\mathsf{R}$-modules can be defined as of Definition \ref{defn-nla} by replacing the vector space $\mathfrak{g}$ by an $\mathsf{R}$-module and all the (bi)linear operations on $\mathfrak{g}$ by $\mathsf{R}$-(bi)linear operations on the $\mathsf{R}$-module. Morphisms and isomorphisms between Nijenhuis Lie algebras in the category of $\mathsf{R}$-modules can be defined similarly. Note that any Nijenhuis Lie algebra $(\mathfrak{g}, [~, ~]_\mathfrak{g}, N)$ can be regarded as a Nijenhuis Lie algebra in the category of $\mathsf{R}$-modules, where the $\mathsf{R}$-module structure on $\mathfrak{g}$ is given by $r \cdot x = \varepsilon (r) x$, for $r \in \mathsf{R}$ and $x \in \mathfrak{g}$.

\begin{defn}
An {\bf {\sf R}-deformation} of a Nijenhuis Lie algebra $(\mathfrak{g}, [~,~]_\mathfrak{g}, N)$ consists of a pair $(\mu_\mathsf{R}, N_\mathsf{R})$ of an antisymmetric {\sf R}-bilinear map $\mu_\mathsf{R} : (\mathsf{R} \otimes_{\bf k} \mathfrak{g}) \times (\mathsf{R} \otimes_{\bf k} \mathfrak{g}) \rightarrow \mathsf{R} \otimes_{\bf k} \mathfrak{g}$ and a {\sf R}-linear map $N_\mathsf{R} : \mathsf{R} \otimes_{\bf k} \mathfrak{g} \rightarrow \mathsf{R} \otimes_{\bf k} \mathfrak{g}$ such that the following conditions hold:
\begin{itemize}
\item $( \mathsf{R} \otimes_{\bf k} \mathfrak{g}, \mu_\mathsf{R}, N_\mathsf{R})$ is a Nijenhuis Lie algebra in the category of $\mathsf{R}$-modules,
\item the map $\varepsilon \otimes_\mathbf{k} \mathrm{Id}_\mathfrak{g}: \mathsf{R} \otimes_{\bf k} \mathfrak{g} \rightarrow \mathfrak{g}$ is a morphism of Nijenhuis Lie algebras in the category of $\mathsf{R}$-modules from $( \mathsf{R} \otimes_{\bf k} \mathfrak{g}, \mu_\mathsf{R}, N_\mathsf{R})$ to $(\mathfrak{g}, [~,~]_\mathfrak{g}, N)$.
\end{itemize}
\end{defn}

\begin{defn}
Let $(\mathfrak{g}, [~,~]_\mathfrak{g}, N)$ be a Nijenhuis Lie algebra. Two ${\sf R}$-deformations $(\mu_\mathsf{R}, N_\mathsf{R})$ and $(\mu'_\mathsf{R}, N'_\mathsf{R})$ are said to be {\bf equivalent} if there exists an ${\sf R}$-linear isomorphism $\varphi: \mathsf{R} \otimes_{\bf k} \mathfrak{g} \rightarrow \mathsf{R} \otimes_{\bf k} \mathfrak{g}$ which is a morphism of Nijenhuis Lie algebras in the category of ${\sf R}$-modules from $( \mathsf{R} \otimes_{\bf k} \mathfrak{g}, \mu_\mathsf{R}, N_\mathsf{R})$ to 
$( \mathsf{R} \otimes_{\bf k} \mathfrak{g}, \mu'_\mathsf{R}, N'_\mathsf{R})$,
 satisfying additionally 
 $( \varepsilon \otimes_{\bf k} \mathrm{Id}_\mathfrak{g}) \circ \varphi = (\varepsilon \otimes_{\bf k} \mathrm{Id}_\mathfrak{g})$.
\end{defn}

In the following, we shall consider the cases when ${\sf R} = {\bf k} [[t]]$ (the ring of formal power series) and ${\sf R} = {\bf k}[[t]]/ (t^2)$ (the local Artinian ring of dual numbers) with the obvious augmentations. The corresponding {\sf R}-deformations are respectively called formal deformations and infinitesimal deformations. We now briefly discuss the formal deformations.

\begin{defn}
(i) Let $(\mathfrak{g}, [~, ~]_\mathfrak{g}, N)$ be a Nijenhuis Lie algebra. A {\em formal deformation} of this Nijenhuis Lie algebra is a pair $(\mu_t, N_t)$ of formal sums $\mu_t = \sum_{i=0}^\infty t^i \mu_i$ and $N_t = \sum_{i=0}^\infty t^i N_i$ (where each $\mu_i : \mathfrak{g} \times \mathfrak{g} \rightarrow \mathfrak{g}$ are bilinear antisymmetric maps and $N_i : \mathfrak{g} \rightarrow \mathfrak{g}$ are linear maps with $\mu_0 = [~,~]_\mathfrak{g}$ and $N_0 = N$) that makes the triple $(\mathfrak{g}[[t]], \mu_t, N_t)$ into a Nijenhuis Lie algebra in the category of ${\bf k} [[t]]$-modules.

\medskip

(ii) Two formal deformations $(\mu_t = \sum_{i=0}^\infty t^i \mu_i, N_t = \sum_{i=0}^\infty t^i N_i)$ and  $(\mu'_t = \sum_{i=0}^\infty t^i \mu'_i, N'_t = \sum_{i=0}^\infty t^i N'_i)$ are {\em equivalent} if there exists a ${\bf k}[[t]]$-linear map $\varphi
_t : \mathfrak{g}[[t]] \rightarrow \mathfrak{g} [[t]]$ of the form $\varphi_t = \sum_{i=0}^\infty t^i \varphi_i$ (where each $\varphi_i:\mathfrak{g} \rightarrow \mathfrak{g}$ are linear maps with $\varphi_0 = \mathrm{Id}_\mathfrak{g}$) that defines an isomorphism of Nijenhuis Lie algebras in the category of ${\bf k}[[t]]$-modules from $(\mathfrak{g}[[t]], \mu_t, N_t)$ to $(\mathfrak{g}[[t]], \mu'_t, N'_t)$.
\end{defn} 

It follows from the above definition that a pair $(\mu_t = \sum_{i=0}^\infty t^i \mu_i, N_t = \sum_{i=0}^\infty t^i N_i)$ is a formal deformation of the Nijenhuis Lie algebra $(\mathfrak{g}, [~, ~ ]_\mathfrak{g}, N)$ if and only if the following set of identities are hold: For each $p \geq 0$ and $x, y, z \in \mathfrak{g}$,
\begin{align*}
&\sum_{i+j = p} \big\{ \mu_i ( \mu_j (x, y), z) +   \mu_i ( \mu_j ( y, z), x) + \mu_i ( \mu_j (z, x), y) \big\} = 0, \\
\sum_{i+j+k = p} & \mu_i ( N_j (x) ,N_k (y)) = \sum_{i+j + k = p} N_i \big( \mu_j ( N_k (x), y) + \mu_j (x, N_k (y)) - N_k (\mu_j (x, y))    \big).
\end{align*}
Both the above identities are automatically hold for $p= 0$ (as $\mu_0 = [~,~]_\mathfrak{g}$ and $N_0 =N$). However, for $p= 1$, we get that
\begin{align}
[\mu_1 (x, y), z]_\mathfrak{g} + [\mu_1 (y, z), x]_\mathfrak{g} + [\mu_1 (z, x), y]_\mathfrak{g} + \mu_1  ([x, y]_\mathfrak{g}, z) + \mu_1 ([y, z]_\mathfrak{g}, x) + \mu_1 ([z, x]_\mathfrak{g}, y) = 0, \label{def1}
\end{align}
\begin{align}
\mu_1 (N (x), N (y)) + [N_1 (x), N (y)]_\mathfrak{g} + [N (x), N_1 (y)]_\mathfrak{g} = N_1 ( [N(x), y]_\mathfrak{g} + [x, N (y)]_\mathfrak{g} - N [x, y]_\mathfrak{g}) \label{def2} \\
+ N \big(  \mu_1 (N (x) ,y) + \mu_1 (x, N (y)) - N \mu_1 (x, y) \big) + N ( [N_1(x), y]_\mathfrak{g} + [x, N_1 (y)]_\mathfrak{g} - N_1 [x, y]_\mathfrak{g}  ), \nonumber
\end{align}
for all $x, y, z \in \mathfrak{g}$. The identity (\ref{def1}) simply means that $(\delta_\mathrm{CE} \mu_1) (x, y, z) = 0$ while the identity (\ref{def2}) is equivalent to $( d_N (N_1) + \partial^N (\mu_1)) (x, y) = 0$. As a result, we obtain that
\begin{align*}
\delta_\mathrm{NLie} (\mu_1, N_1) = ( \delta_\mathrm{CE} \mu_1 ~ \!, ~ \! d_N (N_1) + \partial^N (\mu_1) ) = 0.
\end{align*}
This shows that $(\mu_1, N_1)$ is a $2$-cocycle of the Nijenhuis Lie algebra $(\mathfrak{g}, [~,~]_\mathfrak{g}, N)$ with coefficients in the adjoint Nijenhuis representation. In general, if $(\mu_1, N_1) = \cdots = (\mu_l, N_l) = 0$ then $(\mu_{l+1}, N_{l+1})$ is a $2$-cocycle.

Two formal deformations $(\mu_t = \sum_{i=0}^\infty t^i \mu_i, N_t = \sum_{i=0}^\infty t^i N_i)$ and  $(\mu'_t = \sum_{i=0}^\infty t^i \mu'_i, N'_t = \sum_{i=0}^\infty t^i N'_i)$ are equivalent if and only if 
\begin{align*}
\sum_{i+j = p} \varphi_i (\mu_j (x, y)) = \sum_{i+j + k = p} \mu_i' ( \varphi_j (x) , \varphi_k (y)) ~~~~ \text{ and } ~~~~ \sum_{i+ j = p } N_i' \circ \varphi_j = \sum_{i+j = p } \varphi_i \circ N_j,
\end{align*}
for any $p \geq 0$ and $x, y \in \mathfrak{g}$. As before, both the above identities are held automatically as $\mu_0 = \mu_0' = [~,~]_\mathfrak{g}$, $N_0 = N_0' = N$ and $\varphi_0 = \mathrm{Id}_\mathfrak{g}$. However, for $p =1$, we obtain
\begin{align*}
 \mu_1 (x, y) - \mu_1' (x, y) =~& [x, \varphi_1 (y)]_\mathfrak{g} - \varphi_1 [x, y]_\mathfrak{g} + [\varphi_1 (x), y]_\mathfrak{g} = (\delta_\mathrm{CE} \varphi_1) (x, y), \\
 N_1 - N_1' =~& N \circ \varphi_1 - \varphi_1 \circ N.
\end{align*}
for $x, y \in \mathfrak{g}$. These two identities can be simply expressed as
\begin{align*}
(\mu_1, N_1) - (\mu_1' , N_1') = (\delta_\mathrm{CE} (\varphi_1) , - \partial^N (\varphi_1) ) = \delta_\mathrm{NLie} (\varphi_1).
\end{align*}
As a conclusion of the above discussions, we get the following.
\begin{thm}
Let $(\mathfrak{g}, [~,~]_\mathfrak{g}, N)$ be a Nijenhuis Lie algebra. Then the infinitesimal of any formal deformation is a $2$-cocycle of the Nijenhuis Lie algebra $(\mathfrak{g}, [~,~]_\mathfrak{g}, N)$  with coefficients in the adjoint Nijenhuis representation. Moreover, the infinitesimals corresponding to equivalent formal deformations are cohomologous, i.e. they correspond to the same cohomology class in $H^2_\mathrm{NLie} (\mathfrak{g}, N)$.
\end{thm}

We have already mentioned earlier that an infinitesimal deformation of a Nijenhuis Lie algebra is an ${\sf R}$-deformation for ${\sf R} = {\bf k}[[t]] /(t^2)$. That is, an infinitesimal deformation can be regarded as a truncated version (module $t^2$) of formal deformation. Equivalences between infinitesimal deformations can be defined similarly.

\begin{thm}
Let $(\mathfrak{g}, [~,~]_\mathfrak{g}, N)$ be a Nijenhuis Lie algebra. Then the set of all equivalence classes of infinitesimal deformations has a bijection with the second cohomology group $H^2_\mathrm{NLie} (\mathfrak{g}, N)$. 
\end{thm}

\begin{proof}
Let $(\mu_t = [~,~]_\mathfrak{g} + t \mu_1, N_t  = N + t N_1)$ be an infinitesimal deformation of the Nijenhuis Lie algebra $(\mathfrak{g}, [~,~]_\mathfrak{g}, N)$. Then similar to the case of formal deformation, one can show that $(\mu_1, N_1)$ is a $2$-cocycle. Moreover, equivalent infinitesimal deformations correspond to cohomologous $2$-cocycles. Hence there is a well-defined map from the set of all equivalence classes of infinitesimal deformations of $(\mathfrak{g}, [~,~]_\mathfrak{g}, N)$ to the second cohomology group $H^2_\mathrm{NLie} (\mathfrak{g}, N)$. To obtain a map in the other direction, we first take a $2$-cocycle $(\mu_1, N_1)$ of the Nijenhuis Lie algebra $(\mathfrak{g}, [~,~]_\mathfrak{g}, N)$ with coefficients in the adjoint Nijenhuis representation. Then it is easy to show that the pair $(\mu_t = [~,~]_\mathfrak{g} + t \mu_1, N_t  = N + t N_1)$ is an infinitesimal deformation. Next, suppose that $(\mu_1, N_1)$ and $(\mu_1' , N_1')$ are two cohomologous $2$-cocycles, say $(\mu_1, N_1) - (\mu_1' , N_1') = \delta_\mathrm{NLie} (\varphi_1)$. Then it turns out that the corresponding infinitesimal deformations $(\mu_t, N_t)$ and $(\mu_t' , N_t')$ are equivalent and an equivalence is given by the map $\varphi_t = \mathrm{Id}_\mathfrak{g} + t \varphi_1$. This shows the existence of a well-defined map from $H^2_\mathrm{NLie} (\mathfrak{g}, N)$ to the set of all equivalence classes of infinitesimal deformations of $(\mathfrak{g}, [~, ~]_\mathfrak{g}, N)$. Finally, the above two constructed maps are inverses to each other. This completes the proof.
\end{proof}

One may also consider finite order deformations of a Nijenhuis Lie algebra and discuss the obstructions for their extensibility. We hope that the obstructions must be third cocycles of the Nijenhuis Lie algebra with coefficients in the adjoint Nijenhuis representation. We ended up with very long computations but couldn't derive. The Maurer-Cartan characterization of a Nijenhuis Lie algebra could be useful to do so (see also Remark \ref{mc-char}). 

\medskip

\medskip

\section{Homotopy Nijenhuis operators on $2$-term $L_\infty$-algebras}\label{sec5}

In this section, we introduce homotopy Nijenhuis operators on $2$-term $L_\infty$-algebras. We shall call a $2$-term $L_\infty$-algebra endowed with a homotopy Nijenhuis operator as a $2$-term Nijenhuis $L_\infty$-algebra. We show that `skeletal' $2$-term Nijenhuis $L_\infty$-algebras are characterized by third cocycles of Nijenhuis Lie algebras. Subsequently, we also consider crossed modules of Nijenhuis Lie algebras that are equivalent to `strict' $2$-term Nijenhuis $L_\infty$-algebras.

\begin{defn}\cite{baez-crans}
    A {\bf $2$-term $L_\infty$-algebra} is a triple $(\mathcal{L}_1 \xrightarrow{\partial} \mathcal{L}_0, l_2, l_3)$ consisting of a $2$-term chain complex $\mathcal{L}_1 \xrightarrow{\partial} \mathcal{L}_0$ endowed with an antisymmetric bilinear map $l_2 : \mathcal{L}_i \times \mathcal{L}_j \rightarrow \mathcal{L}_{i+j}$ (for $0 \leq i, j, i+j \leq 1$) and an antisymmetric trilinear operation $l_3 : \mathcal{L}_0 \times \mathcal{L}_0 \times \mathcal{L}_0 \rightarrow \mathcal{L}_1$ such that for all $x, y, z , w \in \mathcal{L}_0$ and $h, k \in \mathcal{L}_1$, the following set of identities are satisfied:
    \begin{align}
        \partial l_2 (x, h) = ~& l_2 (x, \partial h), \label{2term1}\\
        l_2 (\partial h, k) =~& l_2 (h, \partial k ), \label{2term2}\\
        \partial l_3 (x, y, z) =~& l_2 (x, l_2 (y, z)) + l_2 (y, l_2 (z, x)) + l_2 (z, l_2 (x, y)), \label{2term3}\\
        l_3 (x, y, \partial h) =~& l_2 (x, l_2 (y, h)) + l_2 (y, l_2 (h, x)) + l_2 (h, l_2 (x, y)), \label{2term4}\\ 
        l_2 (x, l_3 (y, z, w)) ~-~& l_2 (y, l_3 (x, z, w)) + l_2 (z, l_3 (x, y, w)) - l_2 (w, l_3 (x, y, z))  - l_3 (l_2 (x, y), z, w) \label{2term5} \\
        +~ \! l_3 (l_2 (x, z), y, w) &- l_3 (l_2 (x, w), y , z) - l_3 (l_2 (y, z), x, w) + l_3 (l_2 (y, w), x, z) - l_3 (l_2 (z, w), x, y) = 0. \nonumber
    \end{align}
\end{defn}

A $2$-term $L_\infty$-algebra $(\mathcal{L}_1 \xrightarrow{\partial} \mathcal{L}_0, l_2, l_3)$ is said to be {\em skeletal} if $\partial = 0$ and {\em strict} if $l_3  = 0$. Baez and Crans \cite{baez-crans} have shown that skeletal $2$-term $L_\infty$-algebras can be characterized by Chevalley-Eilenberg $3$-cocycles of Lie algebras and strict  $2$-term $L_\infty$-algebras are characterized by crossed modules of Lie algebras.

\begin{defn}
    Let $(\mathcal{L}_1 \xrightarrow{\partial} \mathcal{L}_0, l_2, l_3)$ be a $2$-term $L_\infty$-algebra. A {\bf homotopy Nijenhuis operator} on this $2$-term $L_\infty$-algebra is a triple $\mathcal{N} = (\mathcal{N}_0, \mathcal{N}_1, \mathcal{N}_2)$ consisting of linear maps $ \mathcal{N}_0 : \mathcal{L}_0 \rightarrow \mathcal{L}_0$,  $ \mathcal{N}_1 : \mathcal{L}_1 \rightarrow \mathcal{L}_1 $ and an antisymmetric bilinear operation $\mathcal{N}_2 : \mathcal{L}_0 \times \mathcal{L}_0 \rightarrow \mathcal{L}_1$ subject to satisfy the following set of identities:
    \begin{align}
       \partial \circ \mathcal{N}_1 = \mathcal{N}_0 \circ \partial, \label{hn1} 
       \end{align}
       \begin{align}
       \mathcal{N}_0 \big(  l_2 ( \mathcal{N}_0 (x), y) + l_2 (x, \mathcal{N}_0 (y)) - \mathcal{N}_0 (l_2 (x, y))  \big) - l_2 ( \mathcal{N}_0 (x), \mathcal{N}_0 (y))= \partial (\mathcal{N}_2 (x, y)), \label{hn2}
       \end{align}
       \begin{align}
       \mathcal{N}_1 \big(   l_2 ( \mathcal{N}_0 (x), h) + l_2 (x, \mathcal{N}_1 (h)) - \mathcal{N}_1 (l_2 (x, h))    \big) - l_2 (\mathcal{N}_0 (x), \mathcal{N}_1 (h) )= \mathcal{N}_2 (x, \partial h), \label{hn3}
       \end{align}
       \begin{align}
       & \qquad \qquad \qquad \qquad \qquad l_2 (\mathcal{N}_0 (x), \mathcal{N}_2 (y, z)) + l_2  (\mathcal{N}_0 (y), \mathcal{N}_2 ( z, x)) + l_2 (\mathcal{N}_0 (z), \mathcal{N}_2 (x, y)) \label{hn4}\\
       &-\mathcal{N}_2 \big(  l_2 ( \mathcal{N}_0 (x), y) + l_2 (x, \mathcal{N}_0(y)) - \mathcal{N}_0 l_2 (x, y), ~ \! z    \big) - \mathcal{N}_2 \big(  l_2 ( \mathcal{N}_0 (y), z) + l_2 (y, \mathcal{N}_0(z)) - \mathcal{N}_0 l_2 ( y, z), ~ \! x    \big) \nonumber \\
       &  \qquad \qquad \qquad \qquad \qquad  - \mathcal{N}_2 \big(  l_2 ( \mathcal{N}_0 (z), x) + l_2 (z, \mathcal{N}_0(x)) - \mathcal{N}_0 l_2 (z, x),~ \!  y    \big) \nonumber \\
       &- \mathcal{N}_1 \big( l_2 (x, \mathcal{N}_2 (y, z)) +  l_2 (y, \mathcal{N}_2 ( z, x)) +  l_2 (z, \mathcal{N}_2 (x, y))   - \mathcal{N}_2 ( l_2 (x, y), z) - \mathcal{N}_2 ( l_2 (y, z), x) - \mathcal{N}_2 ( l_2 (z, x), y)  \big) \nonumber  \\
       & \quad = l_3 ( \mathcal{N}_0 (x), \mathcal{N}_0 (y), \mathcal{N}_0 (z)) - \mathcal{N}_1 l_3 ( \mathcal{N}_0 (x), \mathcal{N}_0 (y), z) - \mathcal{N}_1 l_3 ( \mathcal{N}_0 (x), y, \mathcal{N}_0 (z)) - \mathcal{N}_1 l_3 ( x, \mathcal{N}_0 (y), \mathcal{N}_0 (z)) \nonumber \\
       &  \qquad \qquad \qquad + \mathcal{N}_1^2 l_3 ( \mathcal{N}_0 (x),  y, z) + \mathcal{N}_1^2 l_3 (x, \mathcal{N}_0 (y), z) + \mathcal{N}_1^2 l_3 (x, y, \mathcal{N}_0 (z)) - \mathcal{N}_1^3 l_3 (x, y, z), \nonumber 
    \end{align}
    for all $x, y, z \in \mathcal{L}_0$ and $h \in \mathcal{L}_1$.
\end{defn}

In \cite{jiang-sheng} the authors have introduced the notion of homotopy relative Rota-Baxter operators on $2$-term $L_\infty$-algebras over some representations. Let $(\mathcal{L}_1 \xrightarrow{\partial} \mathcal{L}_0, l_2, l_3)$ be a $2$-term $L_\infty$-algebra. A {\em representation} of $(\mathcal{L}_1 \xrightarrow{\partial} \mathcal{L}_0, l_2, l_3)$ is a $2$-term chain complex $\mathcal{V}_1 \xrightarrow{\overline{\partial}} \mathcal{V}_0$ endowed with a bilinear map $m_2: \mathcal{L}_i \times \mathcal{V}_j \rightarrow \mathcal{V}_{i+j}$ (for $0 \leq i, j, i+j \leq 1$) and a trilinear operation $m_3: \mathcal{L}_0 \times \mathcal{L}_0 \times \mathcal{V}_0 \rightarrow \mathcal{V}_1$ that is antisymmetric on the first two inputs such that for all $x, y, z \in \mathcal{L}_0$, $h \in \mathcal{L}_1$, $v \in \mathcal{V}_0$ and $p \in \mathcal{V}_1$,
\begin{align*}
    \overline{\partial} m_2 (x, p) =~& m_2 (x, \overline{\partial} p),\\
    m_2 (\partial h, p) =~& m_2 (h, \overline{\partial} p),\\
    \overline{\partial} m_3 (x, y, v ) =~& m_2 (x, m_2 (y, v)) - m_2 (y, m_2 (x, v)) - m_2 (l_2 (x, y), v) ,\\
    m_3 (x, y, \overline{\partial} p) =~& m_2 (x, m_2 (y, p)) - m_2 (y, m_2 (x, p)) - m_2 (l_2 (x, y), p), \\
    m_2 (x, m_3 (y, z, v)) - ~& m_2 (y, m_3 (x, z, v)) + m_2 (z, m_3 (x, y, v)) + m_2 (l_3 (x, y, z), v) - m_3 (l_2 (x, y), z, v)\\
    + m_3 (l_2 (x, z), y, v) &- m_3 (y, z, m_2 (x, v)) - m_3 (l_2 (y, z), x, v) + m_3 (x, z, m_2 (y, v)) - m_3 (x, y, m_2 (z, v)) = 0.
\end{align*}
In this case, the triple $\big( \mathcal{L}_1 \oplus \mathcal{V}_1 \xrightarrow{ \partial + \overline{\partial}} \mathcal{L}_0 \oplus \mathcal{V}_0, l_2 \ltimes m_2, l_3 \ltimes m_3  \big)$ turns out to be a $2$-term $L_\infty$-algebra, where
\begin{align*}
    &(l_2 \ltimes m_2) ((x, u),  (y, v)) := \big(  l_2 (x, y) , ~ \! m_2 (x, v) - m_2 (y, u) \big), \text{ for } (x, u), (y, v) \in \mathcal{L}_0 \oplus \mathcal{V}_0 \text{ or } \mathcal{L}_1 \oplus \mathcal{V}_1,\\
    & \qquad (l_3 \ltimes m_3) ((x, u), (y, v), (z, w)) := \big(   l_3 (x, y, z), ~ \! m_3 (x, y, w) + m_3 (y, z, u) + m_3 (z, x, v) \big),
\end{align*}
for $(x, u), (y, v), (z, w) \in \mathcal{L}_0 \oplus \mathcal{V}_0$. This is called the {\em semidirect product} $2$-term $L_\infty$-algebra.

\medskip

Let $(\mathcal{L}_1 \xrightarrow{\partial} \mathcal{L}_0, l_2, l_3)$ be a $2$-term $L_\infty$-algebra and $(\mathcal{V}_1 \xrightarrow{\overline{\partial}} \mathcal{V}_0, m_2, m_3)$ be a representation of it. Then a {\em homotopy relative Rota-Baxter operator} or a {\em homotopy $\mathcal{O}$-operator} \cite{jiang-sheng} is a triple $r = (r_0 , r_1, r_2)$ of linear maps $r_0 : \mathcal{V}_0 \rightarrow \mathcal{L}_0$, $r_1 : \mathcal{V}_1 \rightarrow \mathcal{L}_1$ and an antisymmetric bilinear map $r_2 : \mathcal{V}_0 \times \mathcal{V}_0 \rightarrow \mathcal{L}_1$ that satisfy
\begin{align*}
& \qquad \qquad \qquad \qquad \partial \circ r_1 = r_0 \circ \overline{\partial},\\
& r_0 \big( m_2 (r_0 (u), v) - m_2 (r_0 (v) , u) \big) - l_2 (r_0 (u), r_0 (v)) = {\partial} (r_2 (u, v)),\\
& r_1 \big(    m_2 (r_0 (u), p) - m_2 (r_1 (p), u) \big) - l_2 (r_0 (u), r_1 (p)) = r_2 (u, \overline{\partial} p),\\
& \big\{ l_2 ( r_0 (u) , r_2 (v, w)) - r_2 \big(  m_2 (r_0 (u), v) - m_2 (r_0 (v), u) , w \big) + r_1 \big(  m_2 ( r_2 (u, v) , w) + m_3 (r_0 (u), r_0 (v), w) \big) \big\} + c. p. \\
& \qquad \qquad \qquad \qquad \qquad = l_3 (r_0 (u), r_0 (v), r_0 (w)), \text{ for } u, v, w \in \mathcal{V}_0, p \in \mathcal{V}_1.
\end{align*}

\begin{prop}
Let $(\mathcal{L}_1 \xrightarrow{\partial} \mathcal{L}_0, l_2, l_3)$ be a $2$-term $L_\infty$-algebra and $(\mathcal{V}_1 \xrightarrow{\overline{\partial}} \mathcal{V}_0, m_2, m_3)$ be a representation of it. Let $r = (r_0, r_1, r_2)$ be a triple consisting of linear maps $r_0 : \mathcal{V}_0 \rightarrow \mathcal{L}_0$, $r_1 : \mathcal{V}_1 \rightarrow \mathcal{L}_1$ and an antisymmetric bilinear map $r_2 : \mathcal{V}_0 \times \mathcal{V}_0 \rightarrow \mathcal{L}_1$. Then $r = (r_0, r_1, r_2)$ is a homotopy relative Rota-Baxter operator if and only if the triple $\widetilde{r} = (\widetilde{r_0}, \widetilde{r_1}, \widetilde{r_2})$ is a homotopy Nijenhuis operator on the semidirect product $2$-term $L_\infty$-algebra $\big( \mathcal{L}_1 \oplus \mathcal{V}_1 \xrightarrow{ \partial + \overline{\partial}} \mathcal{L}_0 \oplus \mathcal{V}_0, l_2 \ltimes m_2, l_3 \ltimes m_3  \big)$, where
\begin{align*}
    &\widetilde{r_0} : \mathcal{L}_0 \oplus \mathcal{V}_0 \rightarrow \mathcal{L}_0 \oplus \mathcal{V}_0 ~ \text{ given by } ~ \widetilde{r_0} (x, u) = (r_0 (u), 0),\\
    &\widetilde{r_1} : \mathcal{L}_1 \oplus \mathcal{V}_1 \rightarrow \mathcal{L}_1 \oplus \mathcal{V}_1 ~ \text{ given by } ~ \widetilde{r_0} (h, p) = (r_1 (p), 0), \\
    &\widetilde{r_2} : (\mathcal{L}_0 \oplus \mathcal{V}_0 ) \times (\mathcal{L}_0 \oplus \mathcal{V}_0) \rightarrow \mathcal{L}_1 \oplus \mathcal{V}_1 ~ \text{ given by } ~ \widetilde{r_2} ((x, u), (y, v)) = (r_2 (u, v), 0).
\end{align*}
\end{prop}

Like a Nijenhuis Lie algebra is a Lie algebra equipped with a distinguished Nijenhuis operator, we define a {\bf $2$-term Nijenhuis $L_\infty$-algebra} as pair $ ( (\mathcal{L}_1 \xrightarrow{\partial} \mathcal{L}_0, l_2, l_3), (\mathcal{N}_0, \mathcal{N}_1, \mathcal{N}_2) )$ consisting of a $2$-term $L_\infty$-algebra endowed with a homotopy Nijenhuis operator on it. A $2$-term Nijenhuis $L_\infty$-algebra as above is said to be {\em skeletal} if the underlying $2$-term $L_\infty$-algebra is skeletal (i.e. $\partial = 0$). On the other hand, it is said to be {\em strict} if the underlying $2$-term $L_\infty$-algebra is strict (i.e. $l_3 = 0$) and additionally $\mathcal{N}_2 = 0$.

\begin{thm}
   There is a 1-1 correspondence between skeletal $2$-term Nijenhuis $L_\infty$-algebras and third cocycles of Nijenhuis Lie algebras with coefficients in Nijenhuis representations. 
\end{thm}

\begin{proof}
    Let $ ( (\mathcal{L}_1 \xrightarrow{\partial = 0} \mathcal{L}_0, l_2, l_3), (\mathcal{N}_0, \mathcal{N}_1, \mathcal{N}_2) )$ be a skeletal $2$-term Nijenhuis $L_\infty$-algebra. Since $\partial = 0$, it follows from (\ref{2term3}) that the vector space $\mathcal{L}_0$ with the bilinear antisymmetric operation $l_2 : \mathcal{L}_0 \times \mathcal{L}_0 \rightarrow \mathcal{L}_0$ is a Lie algebra. Further, the identity (\ref{hn2}) then implies that the linear map $\mathcal{N}_0 : \mathcal{L}_0 \rightarrow \mathcal{L}_0$ is a Nijenhuis operator on the Lie algebra $(\mathcal{L}_0, l_2)$. In other words, $(\mathcal{L}_0, l_2, \mathcal{N}_0 )$ is a Nijenhuis Lie algebra. On the other hand, it follows from the identities (\ref{2term4}) and (\ref{hn3}) that the pair $(\mathcal{L}_1, \rho, \mathcal{N}_1)$ is a Nijenhuis representation of the Nijenhuis Lie algebra $(\mathcal{L}_0, l_2, \mathcal{N}_0 )$, where $\rho: \mathcal{L}_0 \rightarrow \mathrm{End} (\mathcal{L}_1)$ is given by $\rho_x h:= l_2 (x, h)$, for $x \in \mathcal{L}_0$ and $h \in \mathcal{L}_1$. Finally, the identity (\ref{2term5}) is same as $(\delta_\mathrm{CE} l_3) (x, y, z, w) = 0$ and the identity (\ref{hn4}) can be equivalently rephrased as $d_{\mathcal{N}_0, \mathcal{N}_1} (\mathcal{N}_2) (x, y, z ) = \partial^{\mathcal{N}_0, \mathcal{N}_1} (l_3) (x, y, z)$. Here $\delta_\mathrm{CE}$ is the Chevalley-Eilenberg coboundary operator of the Lie algebra $(\mathcal{L}_0, l_2)$ with coefficients in the representation $(\mathcal{L}_1, \rho)$. Thus, we obtain that
    \begin{align*}
        \delta_\mathrm{NLie} (l_3, \mathcal{N}_2) = \big( \delta_\mathrm{CE} (l_3) ~ \!, ~ \! d_{\mathcal{N}_0, \mathcal{N}_1} (\mathcal{N}_2) - \partial^{\mathcal{N}_0, \mathcal{N}_1} (l_3)  \big) = 0.
    \end{align*}
    This shows that the element $(l_3, \mathcal{N}_2) \in \mathrm{Hom} (\wedge^3 \mathcal{L}_0, \mathcal{L}_1) \oplus \mathrm{Hom} (\wedge^2 \mathcal{L}_0, \mathcal{L}_1)$ is a $3$-cocycle of the Nijenhuis Lie algebra $(\mathcal{L}_0, l_2, \mathcal{N}_0)$ with coefficients in the Nijenhuis representation $(\mathcal{L}_1, \rho, \mathcal{N}_1)$.

    Conversely, let $(\mathfrak{g}, [~,~]_\mathfrak{g}, N)$ be a Nijenhuis Lie algebra, $(\mathcal{V}, \rho, S)$ be a Nijenhuis representation and $(\chi, F)$ be a $3$-cocycle. Then it is straightforward to verify that the pair
    \begin{align*}
         \big(  (\mathcal{V} \xrightarrow{\partial = 0} \mathfrak{g}, l_2, l_3 = \chi), (N, S, F) \big)
    \end{align*}
    is a skeletal $2$-term Nijenhuis $L_\infty$-algebra, where the map $l_2$ is given by
    \begin{align*}
        l_2 (x, y) = [x, y]_\mathfrak{g}, \quad l_2 (x, v) = - l_2 (v, x) = \rho_x v, \text{ for } x, y \in \mathfrak{g} \text{ and } v \in \mathcal{V}.
    \end{align*}
    The above two correspondences are inverses to each other. This completes the proof.
\end{proof}

The notion of crossed modules of Lie algebras was introduced in \cite{baez-crans} while studying strict $2$-term $L_\infty$-algebras. Here we shall generalize this notion in the context of Nijenhuis Lie algebras.

\begin{defn}
    A {\bf crossed module of Nijenhuis Lie algebras} is a quadruple 
    \begin{align*}
        \big(  (\mathfrak{g}, [~,~]_\mathfrak{g}, N), (\mathfrak{h}, [~,~]_\mathfrak{h}, S), t, \rho \big)
    \end{align*}
    consisting of two Nijenhuis Lie algebras $(\mathfrak{g}, [~,~]_\mathfrak{g}, N)$ and $ (\mathfrak{h}, [~,~]_\mathfrak{h}, S)$ endowed with a homomorphism $t : \mathfrak{h} \rightarrow \mathfrak{g}$ of Nijenhuis Lie algebras and a Lie algebra homomorphism $\rho : \mathfrak{g} \rightarrow \mathrm{Der} (\mathfrak{h})$ that satisfy the following conditions:
    \begin{itemize}
        \item[(i)] $(\mathfrak{h}, \rho, S)$ is a Nijenhuis representation of the Nijenhuis Lie algebra $(\mathfrak{g}, [~,~]_\mathfrak{g}, N)$,
        \item[(ii)] for any $x \in \mathfrak{g}$ and $h, k \in \mathfrak{h}$,
        \begin{align*}
            t (\rho_x h) = [x, t (h)]_\mathfrak{g} ~~~ \text{ and } ~~~ \rho_{ t(h)} k = [h, k]_\mathfrak{h}.
        \end{align*}
    \end{itemize}
\end{defn}

\medskip

Let $\big(  (\mathfrak{g}, [~,~]_\mathfrak{g}, N), (\mathfrak{h}, [~,~]_\mathfrak{h}, S), t, \rho \big)$ be a crossed module of Nijenhuis Lie algebras. Then for any $h, k \in \mathfrak{h}$, we observe that
\begin{align*}
    t \big(  [h, k]_\mathfrak{h}^S  \big) =~& t ( [S(h), k]_\mathfrak{h} + [h, S (k)]_\mathfrak{h} - S [h, k]_\mathfrak{h} ) \\
    =~& [t S (h), t(k)]_\mathfrak{g} + [t(h) , t S (k)]_\mathfrak{g} - tS [h, k]_\mathfrak{h} = [t(h), t(k)]^N_\mathfrak{g} \quad (\because t S = Nt)
\end{align*}
which shows that $t : \mathfrak{h}^S \rightarrow \mathfrak{g}^N$ is a homomorphism of deformed Lie algebras. Next, we consider the map $\rho^1 : \mathfrak{g}^N \rightarrow \mathrm{Der} (\mathfrak{h}^S)$ by $\rho^1_x (h) := \rho_{N (x)} h + \rho_x S (h) - S (\rho_x h)$, for $x \in \mathfrak{g}^N$, $h \in \mathfrak{h}^S$. It is easy to see that $\rho^1$ defines a representation of the deformed Lie algebra $\mathfrak{g}^N$ on the space $\mathfrak{h}$. Moreover, for any $x \in \mathfrak{g}$ and $h, k \in \mathfrak{h}$, we have
\begin{align*}
    t (\rho^1_x (h)) = t \big(  \rho_{N (x)} h + \rho_x S (h) - S (\rho_x h)  \big)
    = [N (x), t(h)]_\mathfrak{g} + [x, tS (h)]_\mathfrak{g} - N [x, t (h)]_\mathfrak{g} = [x, t (h)]^N_\mathfrak{g},
    \end{align*}
    \begin{align*}
    \rho^1_{t (h)} (k) = \rho_{N t (h)} k + \rho_{t (h)} S (k) - S (\rho_{t (h)} k)
    = [S (h) , k]_\mathfrak{h} + [h, S (k)]_\mathfrak{h} - S [h, k]_\mathfrak{h} = [h, k]_\mathfrak{h}^S.
\end{align*}
This shows that the quadruple $(\mathfrak{g}^N, \mathfrak{h}^S, t, \rho^1)$ is a crossed module of Lie algebras in the sense of \cite{baez-crans}. More generally, for any $l \geq 0$, one can show that the quadruple $(\mathfrak{g}^{N^l}, \mathfrak{h}^{S^l}, t, \rho^l)$ is a crossed module of Lie algebras, where $\rho^l_x (h) = \rho_{N^l (x) } h + \rho_x S^l (h) - S^l (\rho_x h)$, for $x \in \mathfrak{g}^{N^l}$ and $h \in \mathfrak{h}^{S^l}$.

\begin{thm}
    There is a 1-1 correspondence between strict $2$-term Nijenhuis $L_\infty$-algebras and crossed modules of Nijenhuis Lie algebras.
\end{thm}

\begin{proof}
    Let $( (\mathcal{L}_1 \xrightarrow{\partial} \mathcal{L}_0, l_2, l_3 =0), (\mathcal{N}_0, \mathcal{N}_1, \mathcal{N}_2 = 0))$ be a strict $2$-term Nijenhuis $L_\infty$-algebra. Then it follows from (\ref{2term3}) and (\ref{hn2}) that $(\mathcal{L}_0, l_2, \mathcal{N}_0)$ is a Nijenhuis Lie algebra. We define a bilinear operation $[~,~]_1 : \mathcal{L}_1 \times \mathcal{L}_1 \rightarrow \mathcal{L}_1$ by $[h, k]_1 := l_2 (\partial h , k) = \partial (h, \partial k)$, for $h, k \in \mathcal{L}_1$. This operation is antisymmetric as $l_2$ is so. Moreover, the identities in (\ref{2term4}) and (\ref{hn3}) then implies that $(\mathcal{L}_1, [~,~]_1, \mathcal{N}_1)$ is a Nijenhuis Lie algebra. Further, from (\ref{2term1}) and (\ref{hn1}), we get that the map $\partial : \mathcal{L}_1 \rightarrow \mathcal{L}_0$ is a homomorphism of Nijenhuis Lie algebras. Finally, we set a map $\rho : \mathcal{L}_0 \rightarrow \mathrm{End} (\mathcal{L}_1)$ by $\rho_x h := l_2 (x, h)$, for $x \in \mathcal{L}_0$ and $h \in \mathcal{L}_1$. Then it is easy to see from (\ref{2term4}) that $\rho$ is a Lie algebra homomorphism and additionally $\rho_x \in \mathrm{Der} (\mathcal{L}_1),$ for $x \in \mathcal{L}_0$. Further, it follows from (\ref{hn3}) that 
    \begin{align*}
        \rho_{\mathcal{N}_0 (x)} \mathcal{N}_1 (h) = \mathcal{N}_1 \big(   \rho_{\mathcal{N}_0 (x)} h + \rho_x \mathcal{N}_1 (h) - \mathcal{N}_1 (\rho_x h)  \big), \text{ for } x \in \mathcal{L}_0, h \in \mathcal{L}_1.
    \end{align*}
    This shows that $(\mathcal{L}_1, \rho, \mathcal{N}_1)$ is a Nijenhuis representation of the Nijenhuis Lie algebra $(\mathcal{L}_0, l_2, \mathcal{N}_0)$. For any $x \in \mathcal{L}_0$ and $h, k \in \mathcal{L}_1$, we also have
    \begin{align*}
        \partial (\rho_x h) = \partial l_2 (x, h) = l_2 (x, \partial h) ~~~~ \text{ and } ~~~~ \rho_{\partial (h)} k = l_2 (\partial h, k) = [h, k]_1
    \end{align*}
    which concludes that the quadruple $(  (\mathcal{L}_0, l_2, \mathcal{N}_0), (\mathcal{L}_1, [~,~]_1, \mathcal{N}_1), \partial, \rho )$ is a crossed module of Nijenhuis Lie algebras.

    Conversely, let $ \big(  (\mathfrak{g}, [~,~]_\mathfrak{g}, N), (\mathfrak{h}, [~,~]_\mathfrak{h}, S), t, \rho \big)$ be a crossed module of Nijenhuis Lie algebras. Then it is straightforward to verify that $( (\mathfrak{h} \xrightarrow{t} \mathfrak{g}, l_2, l_3 = 0), (N, S, \mathcal{N}_2 = 0))$ is a strict $2$-term Nijenhuis $L_\infty$-algebra, where $l_2 (x, y) := [x, y]_\mathfrak{g}$ and $l_2 (x, h) = - l_2 (h, x) := \rho_x h$, for all $x, y \in \mathfrak{g}$ and $h \in \mathfrak{h}$. This completes the proof.
\end{proof}

\medskip

\medskip

\section{Nijenhuis Lie bialgebras}\label{sec6}

In this section, we first introduce matched pairs and Manin triples of Nijenhuis Lie algebras. We show that they are equivalent to generic Nijenhuis Lie bialgebras where the Nijenhuis operators on the underlying Lie algebras and Lie coalgebras need not be the same. Subsequently, we consider the admissible classical Yang-Baxter equation (admissible CYBE) whose antisymmetric solutions give rise to Nijenhuis Lie bialgebras. Finally, we define relative Rota-Baxter operators or $\mathcal{O}$-operators on Nijenhuis Lie algebras that yield antisymmetric solutions of the admissible CYBE, and hence produce Nijenhuis Lie bialgebras.

 First, recall that a {\em matched pair of Lie algebras} \cite{majid-lie} is a quadruple $((\mathfrak{g}, [~,~]_\mathfrak{g}), (\mathfrak{h}, [~,~]_\mathfrak{h}), \rho, \nu )$ consisting of two Lie algebras $(\mathfrak{g}, [~,~]_\mathfrak{g})$ and $(\mathfrak{g}, [~,~]_\mathfrak{g})$ endowed with linear maps $\rho : \mathfrak{g} \rightarrow \mathrm{End} (\mathfrak{h})$ and $\nu : \mathfrak{h} \rightarrow \mathrm{End} (\mathfrak{g})$ such that
        \begin{itemize}
            \item $\rho$ defines a representation of the Lie algebra $(\mathfrak{g}, [~,~]_\mathfrak{g})$ on the vector space $\mathfrak{h}$,
            \item $\nu$ defines a representation of the Lie algebra $(\mathfrak{h}, [~,~]_\mathfrak{h})$ on the vector space $\mathfrak{g}$
        \end{itemize}
        satisfying additionally
        \begin{align*}
            \rho_x ([h, k]_\mathfrak{h}) =~& [\rho_x h , k ]_\mathfrak{h} + [h , \rho_x k]_\mathfrak{h} + \rho_{\nu_k x} h - \rho_{\nu_h x} k,\\
            \nu_h ([x, y]_\mathfrak{g}) =~& [\nu_h x, y]_\mathfrak{g} + [x, \nu_h y]_\mathfrak{g} + \nu_{\rho_y h} x - \nu_{\rho_x h} y, 
        \end{align*}
        for all $x, y \in \mathfrak{g}$ and $h, k \in \mathfrak{h}$. It follows that if  $((\mathfrak{g}, [~,~]_\mathfrak{g}), (\mathfrak{h}, [~,~]_\mathfrak{h}), \rho, \nu)$ is a matched pair of Lie algebras then the direct sum $\mathfrak{g} \oplus \mathfrak{h}$ inherits a Lie bracket given by
        \begin{align}\label{bicrossed-bracket}
            [ (x, h) , (y, k)]_\Join := \big(  [x, y]_\mathfrak{g} + \nu_h y - \nu_k x ~ \! , ~ \! [h, k]_\mathfrak{h} + \rho_x k - \rho_y x \big),
        \end{align}
for $(x, h), (y, k) \in \mathfrak{g} \oplus \mathfrak{h}$. The Lie algebra $(\mathfrak{g} \oplus \mathfrak{h}, [~,~]_\Join)$ is said to be the {\em bicrossed product} of the given matched pair of Lie algebras. 

\begin{defn}
    A {\bf matched pair of Nijenhuis Lie algebras} is a tuple $( (\mathfrak{g}, [~,~]_\mathfrak{g}, N), (\mathfrak{h}, [~,~]_\mathfrak{h}, S), \rho, \nu)$ of two Nijenhuis Lie algebras $(\mathfrak{g}, [~,~]_\mathfrak{g}, N)$ and $(\mathfrak{h}, [~,~]_\mathfrak{h}, S)$ with linear maps $\rho : \mathfrak{g} \rightarrow \mathrm{End}(\mathfrak{h})$ and $\nu : \mathfrak{h} \rightarrow \mathrm{End} (\mathfrak{g})$ such that
    \begin{itemize}
        \item $((\mathfrak{g}, [~,~]_\mathfrak{g}), (\mathfrak{h}, [~,~]_\mathfrak{h}), \rho, \nu )$ is a matched pair of Lie algebras,
        \item $(\mathfrak{h}, \rho, S)$ is a representation of the Nijenhuis Lie algebra $(\mathfrak{g}, [~,~]_\mathfrak{g}, N)$,
        \item $(\mathfrak{g}, \nu , N)$ is a representation of the Nijenhuis Lie algebra $(\mathfrak{h}, [~,~]_\mathfrak{h}, S)$.
    \end{itemize}
 \end{defn}

 The following result shows that the bicrossed product construction can be generalized in a matched pair of Nijenhuis Lie algebras.

 \begin{prop}\label{prop-bicrossed-nlie}
     Let $( (\mathfrak{g}, [~,~]_\mathfrak{g}, N), (\mathfrak{h}, [~,~]_\mathfrak{h}, S), \rho, \nu)$ be a matched pair of Nijenhuis Lie algebras. Then the tuple $(\mathfrak{g} \oplus \mathfrak{h}, [~,~]_\Join, N \oplus S)$ is a Nijenhuis Lie algebra, where $[~,~]_\Join$ is the bicrossed product Lie bracket on $\mathfrak{g} \oplus \mathfrak{h}$ given in (\ref{bicrossed-bracket}).
 \end{prop}

 \begin{proof}
     It is enough to show that the map $N \oplus S : \mathfrak{g} \oplus \mathfrak{h} \rightarrow \mathfrak{g} \oplus \mathfrak{h}$ is a Nijenhuis operator on the bicrossed product Lie algebra $(\mathfrak{g} \oplus \mathfrak{h}, [~,~]_\Join)$. For any $(x, h), (y, k) \in \mathfrak{g} \oplus \mathfrak{h}$, we observe that
     \begin{align*}
         &[ (N \oplus S ) (x, h), (N \oplus S) (y, k)]_\Join \\
         &= [ ( N (x), S (h)) , (N (y), S (k))]_\Join  \\
         &= \big(  [N (x), N (y)]_\mathfrak{g} + \nu_{S (h)} N (y) - \nu_{S (k)} N (x) ~ \! , ~ \! [S (h), S(k)]_\mathfrak{h} + \rho_{N (x)} S (k) - \rho_{N (y)} S (h)   \big) \\
         &= \big(  N \big( [N (x), y]_\mathfrak{g} + [x, N (y)]_\mathfrak{g} - N [x, y]_\mathfrak{g} + \nu_{S (h)} y + \nu_{h} N (y) - N (\nu_h y) - \nu_{S (k)} x - \nu_{k } N (x) + N (\nu_k x) \big), \\
         & \qquad  S \big( [ S (h), k]_\mathfrak{h} + [h, S (k)]_\mathfrak{h} - S [h, k]_\mathfrak{h} + \rho_{N (x)} k + \rho_x S (k) - S (\rho_x k) - \rho_{N (y)} h - \rho_y S (h) + S (\rho_y h)  \big) \big)\\
         &= (N \oplus S ) \big(  \big( [N (x), y]_\mathfrak{g} + \nu_{S(h)} y - \nu_k N (x) ~ \!, ~ \! [S(h), k]_\mathfrak{h} + \rho_{N (x)} k - \rho_y S (h) \big) \\
         & \qquad \qquad \qquad  + \big(  [x, N (y)]_\mathfrak{g} + \nu_h N (y) - \nu_{S(k)} x ~ \! , ~ \! [h, S(k)]_\mathfrak{h} + \rho_x S (k) - \rho_{N(y)} h  \big) \\
         & \qquad \qquad \qquad - (N \oplus S) \big(  [x, y]_\mathfrak{g} + \nu_h y - \nu_k x ~ \! , ~ \! [h, k]_\mathfrak{h} + \rho_x k - \rho_y h \big) \big) \\
         &= (N \oplus S) \big(  [(N \oplus S) (x, h), (y, k)]_\Join + [ (x, h), (N \oplus S) (y, k)]_\Join - (N \oplus S) [(x, h), (y, k)]_\Join   \big).
     \end{align*}
     This proves the desired result.
 \end{proof}

 Let $((\mathfrak{g}, [~,~]_\mathfrak{g}, N) , (\mathfrak{h}, [~, ~ ]_\mathfrak{h}, S), \rho, \nu)$ be a matched pair of Nijenhuis Lie algebras. Then it follows from the previous proposition that $(\mathfrak{g} \oplus \mathfrak{h}, [~,~]_\Join^{N \oplus S})$ is a Lie algebra, where
 \begin{align*}
     [(x, h), (y, k)]_\Join^{N \oplus S} = [(N (x), S(h) ), (y, k)]_\Join + [(x, h), (N (y), S (k))]_\Join - (N \oplus S) [(x, h), (y, k)]_\Join.
 \end{align*}
 Further, it turns out that the deformed Lie algebras $(\mathfrak{g}, [~,~]_\mathfrak{g}^N)$ and $(\mathfrak{h}, [~,~]_\mathfrak{h}^S)$ are both Lie subalgebras of $(\mathfrak{g} \oplus \mathfrak{h}, [~,~]_\Join^{N \oplus S})$. More generally, the quadruple $( (\mathfrak{g}, [~,~]_\mathfrak{g}^N) , (\mathfrak{h}, [~,~]_\mathfrak{h}^S), \rho^1, \nu^1)$ constitute a matched pair of Lie algebras, where
 \begin{align*}
     \rho^1_x (h) = \rho_{N(x)} h + \rho_x S (h) - S (\rho_x h) ~~~~ \text{ and } ~~~~ \nu_h^1 (x) = \nu_{S(h)} x + \nu_h N (x) - N (\nu_h x), \text{ for } x \in \mathfrak{g}, h \in \mathfrak{h}.
 \end{align*}
 In general, for any $l \geq 0$, the quadruple $((\mathfrak{g}, [~,~]_\mathfrak{g}^{N^l}), (\mathfrak{h}, [~,~]_\mathfrak{h}^{S^l}), \rho^l, \nu^l )$ is a matched pair of Lie algebras and the corresponding bicrossed product is $(\mathfrak{g} \oplus \mathfrak{h}, [~,~]_\Join^{(N \oplus S)^l}).$

\medskip

 In the following, we consider Manin triples of Nijenhuis Lie algebras generalizing the well-known Manin triples of Lie algebras. First, we recall that \cite{chari} a (standard) {\em Manin triple of Lie algebras} is a triple $((\mathfrak{g} \oplus \mathfrak{g}^*, [~, ~ ]_{\mathfrak{g} \oplus \mathfrak{g}^*} ), (\mathfrak{g}, [~,~]_\mathfrak{g}), (\mathfrak{g}^*, [~,~]_{\mathfrak{g}^*}) )$ consisting of a Lie algebra $(\mathfrak{g} \oplus \mathfrak{g}^*, [~, ~ ]_{\mathfrak{g} \oplus \mathfrak{g}^*} )$ with two Lie subalgebras $(\mathfrak{g}, [~,~]_\mathfrak{g})$ and $(\mathfrak{g}^*, [~,~]_{\mathfrak{g}^*})$ such that the natural nondegenerate symmetric bilinear form $\mathcal{B}$ on the Lie algebra $(\mathfrak{g} \oplus \mathfrak{g}^*, [~, ~ ]_{\mathfrak{g} \oplus \mathfrak{g}^*} )$ given by
 \begin{align*}
     \mathcal{B} (x + \alpha, y +\beta ) = \alpha (y) + \beta (x), \text{ for } x +\alpha, y + \beta \in \mathfrak{g} \oplus \mathfrak{g}^*
 \end{align*}
 is ad-invariant.

 \begin{defn}
     A {\bf Manin triple of Nijenhuis Lie algebras} is a triple 
     \begin{align*}
         ((\mathfrak{g} \oplus \mathfrak{g}^*, [~, ~ ]_{\mathfrak{g} \oplus \mathfrak{g}^*}, N_{\mathfrak{g} \oplus \mathfrak{g}^*} ), (\mathfrak{g}, [~,~]_\mathfrak{g}, N), (\mathfrak{g}^*, [~,~]_{\mathfrak{g}^*}, S^*) )
     \end{align*}
 consisting of a Nijenhuis Lie algebra $(\mathfrak{g} \oplus \mathfrak{g}^*, [~, ~ ]_{\mathfrak{g} \oplus \mathfrak{g}^*}, N_{\mathfrak{g} \oplus \mathfrak{g}^*} )$ with two Nijenhuis Lie subalgebras $ (\mathfrak{g}, [~,~]_\mathfrak{g}, N)$ and $ (\mathfrak{g}^*, [~,~]_{\mathfrak{g}^*}, S^*)$ such that $((\mathfrak{g} \oplus \mathfrak{g}^*, [~, ~ ]_{\mathfrak{g} \oplus \mathfrak{g}^*} ), (\mathfrak{g}, [~,~]_\mathfrak{g}), (\mathfrak{g}^*, [~,~]_{\mathfrak{g}^*}) )$ is a Manin triple of Lie algebras.
  \end{defn}

  In a Manin triple of Nijenhuis Lie algebras, since $(\mathfrak{g}, [~,~]_\mathfrak{g}, N) $ and $ (\mathfrak{g}^*, [~,~]_{\mathfrak{g}^*}, S^*)$ are both Nijenhuis Lie subalgebras of $(\mathfrak{g} \oplus \mathfrak{g}^*, [~, ~ ]_{\mathfrak{g} \oplus \mathfrak{g}^*}, N_{\mathfrak{g} \oplus \mathfrak{g}^*} )$, it turns out that $N_{\mathfrak{g} \oplus \mathfrak{g}^*} = N \oplus S^*$.

  \begin{prop}\label{last-prop1}
    Let $(\mathfrak{g}, [~,~]_\mathfrak{g}, N)$ be a Nijenhuis Lie algebra. Suppose there is a Nijenhuis Lie algebra structure $(\mathfrak{g}^*, [~,~]_{\mathfrak{g}^*}, S^*)$ on the dual vector space $\mathfrak{g}^*$. Then $( (\mathfrak{g}, [~,~]_\mathfrak{g}, N), (\mathfrak{g}^*, [~,~]_{\mathfrak{g}^*}, S^*), \mathrm{ad}^*_\mathfrak{g}, \mathrm{ad}^*_{\mathfrak{g}^*})$ is a matched pair of Nijenhuis Lie algebras if and only if $((\mathfrak{g} \oplus \mathfrak{g}^*, [~,~]_\Join, N \oplus S^*), (\mathfrak{g}, [~,~]_\mathfrak{g}, N), (\mathfrak{g}^*, [~,~]_{\mathfrak{g}^*}, S^*)  )$ is a Manin triple of Nijenhuis Lie algebras.
\end{prop}

\begin{proof}
    It is well-known that $((\mathfrak{g}, [~,~]_\mathfrak{g}), (\mathfrak{g}^*, [~,~]_{\mathfrak{g}^*}, \mathrm{ad}^*_\mathfrak{g}, \mathrm{ad}^*_{\mathfrak{g}^*})$ is a matched pair of Lie algebras if and only if $((\mathfrak{g} \oplus \mathfrak{g}^*, [~,~]_\Join ), (\mathfrak{g}, [~,~]_\mathfrak{g}), (\mathfrak{g}^*, [~,~]_{\mathfrak{g}^*} )$ is a Manin triple of Lie algebras \cite{chari}. Next, the triple $(\mathfrak{g}^*, \mathrm{ad}^*_\mathfrak{g}, S^*)$ is a Nijenhuis representation of the Nijenhuis Lie algebra $(\mathfrak{g}, [~,~]_\mathfrak{g}, N)$ if and only if
    \begin{align}\label{indd1}
         (\mathrm{ad}^*_\mathfrak{g})_{N(x)} S^* (\alpha) = S^* \big(  (\mathrm{ad}^*_\mathfrak{g})_{N(x)} \alpha + (\mathrm{ad}^*_\mathfrak{g})_x S^* (\alpha) - S^* ( (\mathrm{ad}^*_\mathfrak{g})_x \alpha )    \big),
    \end{align}
    for all $x \in \mathfrak{g}$ and $\alpha \in \mathfrak{g}^*$. On the other hand, the triple $(\mathfrak{g}, \mathrm{ad}^*_{\mathfrak{g}^*}, N)$ is a Nijenhuis representation of the Nijenhuis Lie algebra $(\mathfrak{g}^*, [~,~]_{\mathfrak{g}^*}, S^*)$ if and only if
    \begin{align}\label{indd2}
        (\mathrm{ad}^*_{\mathfrak{g}^*} )_{S^* (\alpha)} N (x) = N \big(   (\mathrm{ad}^*_{\mathfrak{g}^*} )_{S^* (\alpha)} x + (\mathrm{ad}^*_{\mathfrak{g}^*} )_{\alpha} N (x) - N ( (\mathrm{ad}^*_{\mathfrak{g}^*} )_{\alpha} x)   \big),
    \end{align}
    for $x \in \mathfrak{g}$ and $\alpha \in \mathfrak{g}^*$. Hence $( (\mathfrak{g}, [~,~]_\mathfrak{g}, N), (\mathfrak{g}^*, [~,~]_{\mathfrak{g}^*}, S^*), \mathrm{ad}^*_\mathfrak{g}, \mathrm{ad}^*_{\mathfrak{g}^*})$ is a matched pair of Nijenhuis Lie algebras if and only if $((\mathfrak{g} \oplus \mathfrak{g}^*, [~,~]_\Join ), (\mathfrak{g}, [~,~]_\mathfrak{g}), (\mathfrak{g}^*, [~,~]_{\mathfrak{g}^*} )$ is a Manin triple of Lie algebras and the identities (\ref{indd1}), (\ref{indd2}) are hold. These two identities are equivalent to say that $N \oplus S^* : \mathfrak{g} \oplus \mathfrak{g}^* \rightarrow \mathfrak{g} \oplus \mathfrak{g}^*$ is a Nijenhuis operator on the Lie algebra $(\mathfrak{g} \oplus \mathfrak{g}^*, [~,~]_\Join)$. Hence the result follows.
\end{proof}

\medskip

In the following, we aim to consider Nijenhuis Lie bialgebras. Before that, we need to understand Nijenhuis Lie coalgebras which are the dual version of Nijenhuis Lie algebras. First, recall that a {\bf Lie coalgebra} \cite{chari} is a pair $(\mathfrak{g}, \delta)$ of a vector space $\mathfrak{g}$ with a linear map $\delta : \mathfrak{g} \rightarrow \mathfrak{g} \otimes \mathfrak{g}$ that satisfy the following conditions:
\begin{itemize}
    \item[(i)] $\delta$ is co-antisymmetric, i.e. $\delta = - \tau \delta$, where $\tau : \mathfrak{g} \otimes \mathfrak{g} \rightarrow \mathfrak{g} \otimes \mathfrak{g}$ is the flip map,
    \item[(ii)] co-Jacobian identity: for any $x \in \mathfrak{g}$,
    \begin{align*}
        (\mathrm{Id}_{\mathfrak{g}^{\otimes 3}} + \sigma + \sigma^2) (\mathrm{Id}_\mathfrak{g} \otimes \delta) \delta (x) = 0,
    \end{align*}
\end{itemize}
where $\sigma : \mathfrak{g}^{\otimes 3} \rightarrow \mathfrak{g}^{\otimes 3}$ is the map $\sigma (x \otimes y \otimes z) = y \otimes z \otimes x$. Let $(\mathfrak{g}, \delta)$ and $(\mathfrak{g}', \delta')$ be two Lie coalgebras. A {\em homomorphism} of Lie coalgebras from $(\mathfrak{g}, \delta)$ to $(\mathfrak{g}', \delta')$ is a linear map $f : \mathfrak{g} \rightarrow \mathfrak{g}'$ satisfying $\delta' \circ f = (f \otimes f) \circ \delta$. The collection of all Lie coalgebras and homomorphisms between them forms a category.

\begin{defn}
    Let $(\mathfrak{g}, \delta)$ be a Lie coalgebra. A {\bf Nijenhuis operator} of $(\mathfrak{g}, \delta)$ is a linear map $S: \mathfrak{g} \rightarrow \mathfrak{g}$ that satisfies
    \begin{align}\label{nij-liecoalg}
        (S \otimes S) \delta (x) = (S \otimes \mathrm{Id}_\mathfrak{g} + \mathrm{Id}_\mathfrak{g} \otimes S) \delta (S (x)) - \delta (S^2 (x)), \text{ for all } x \in \mathfrak{g}.
    \end{align}
    A Lie coalgebra $(\mathfrak{g}, \delta)$ endowed with a Nijenhuis operator $S$ is said to be a {\bf Nijenhuis Lie coalgebra}. We denote a Nijenhuis Lie coalgebra as above simply by $(\mathfrak{g}, \delta, S).$
\end{defn}

Let $(\mathfrak{g}, \delta, S)$ be a Nijenhuis Lie coalgebra. Then it can be checked that $(\mathfrak{g}^*, [~,~]_{\mathfrak{g}^*}, S^*)$ is a Nijenhuis Lie algebra, where the bracket $[~,~]_{\mathfrak{g}^*} : \mathfrak{g}^* \times \mathfrak{g}^* \rightarrow \mathfrak{g}^*$ is given by
\begin{align*}
    \langle [\alpha, \beta]_{\mathfrak{g}^*}, x \rangle := \langle \delta (x) , \alpha \otimes \beta \rangle, \text{ for } \alpha, \beta \in \mathfrak{g}^*, x \in \mathfrak{g}.
\end{align*}
The converse is not true in general. However, if the underlying vector space $\mathfrak{g}$ is finite-dimensional then $(\mathfrak{g}, \delta, S)$ is a Nijenhuis Lie coalgebra if and only if $(\mathfrak{g}^*, [~,~]_{\mathfrak{g}^*}, S^*)$ is a Nijenhuis Lie algebra.

\medskip

In the following results, we show that an arbitrary Nijenhuis Lie coalgebra gives rise to a hierarchy of Lie coalgebra structures.
         
\begin{prop}
    Let $(\mathfrak{g}, \delta, S)$ be a Nijenhuis Lie coalgebra. Then the vector space $\mathfrak{g}$ can be given a new Lie coalgebra structure with the coproduct $\delta_S: \mathfrak{g} \rightarrow \mathfrak{g} \otimes \mathfrak{g}$ given by
    \begin{align*}
        \delta_S (x) :=  (S \otimes \mathrm{Id}_\mathfrak{g} + \mathrm{Id}_\mathfrak{g} \otimes S) \delta (x) - \delta (S (x)), \text{ for } x \in \mathfrak{g}.
    \end{align*}
    Moreover, $S : \mathfrak{g} \rightarrow \mathfrak{g}$ is a homomorphism of Lie coalgebras from $(\mathfrak{g}, \delta)$ to $(\mathfrak{g}, \delta_S).$
\end{prop}

The Lie coalgebra $(\mathfrak{g}, \delta_S)$ constructed in the above proposition is said to be the {\em deformed Lie coalgebra} of the given Nijenhuis Lie coalgebra $(\mathfrak{g}, \delta, S)$.

\begin{prop}
    Let $(\mathfrak{g}, \delta, S)$ be a Nijenhuis Lie coalgebra.
    \begin{itemize}
        \item[(i)] Then for each $k \geq 0$, the map $S^k : \mathfrak{g} \rightarrow \mathfrak{g}$ is also a Nijenhuis operator on the Lie coalgebra $(\mathfrak{g}, \delta).$ In other words, $(\mathfrak{g}, \delta, S^k)$ is a Nijenhuis Lie coalgebra.
        \item[(ii)] For any $k, l \geq 0$, the map $S^l : \mathfrak{g} \rightarrow \mathfrak{g}$ is a Nijenhuis operator on the deformed Lie coalgebra $(\mathfrak{g}, \delta_{S^k})$. That is, $(\mathfrak{g}, \delta_{S^k}, S^l)$ is  a Nijenhuis Lie coalgebra.
        \item[(iii)] Moreover, the deformed Lie coalgebras $(\mathfrak{g}, (\delta_{S^k})_{S^l})$ and $(\mathfrak{g}, \delta_{S^{k+l}})$ are the same.
    \end{itemize}
\end{prop}

Let $(\mathfrak{g}, \delta, S)$ be a Nijenhuis Lie coalgebra. Consider the dual Nijenhuis Lie algebra $(\mathfrak{g}^*, [~,~]_{\mathfrak{g}^*}, S^*)$. Note that, for a linear map $N: \mathfrak{g} \rightarrow \mathfrak{g}$, the map $N^*  : \mathfrak{g}^* \rightarrow \mathfrak{g}^*$ is admissible to the Nijenhuis Lie algebra $(\mathfrak{g}^*, [~,~]_{\mathfrak{g}^*}, S^*)$ if 
\begin{align*}
N^* [S^* (\alpha), \beta]_{\mathfrak{g}^*} + [\alpha, (N^*)^2 (\beta)]_{\mathfrak{g}^*} = [S^* (\alpha ), N^* (\beta)]_{\mathfrak{g}^*} + N^* [\alpha, N^* (\beta)]_{\mathfrak{g}^*}, \text{ for } \alpha, \beta \in \mathfrak{g}^*.
\end{align*}
This condition can be equivalently written as 
\begin{align}\label{adm-sec}
    (N \otimes S) \delta (x) + (N \otimes \mathrm{Id} - \mathrm{Id} \otimes S) \delta (N (x)) - (N^2 \otimes \mathrm{Id}) \delta (x) = 0, \text{ for any } x \in \mathfrak{g}.
\end{align}

\medskip

We are now ready to introduce the notion of a Nijenhuis Lie bialgebra. First, we recall the following \cite{chari}.

\begin{defn}
    A {\bf Lie bialgebra} is a triple $(\mathfrak{g}, [~,~]_\mathfrak{g}, \delta)$ consisting of a vector space $\mathfrak{g}$ endowed with a Lie algebra structure $(\mathfrak{g}, [~,~]_\mathfrak{g})$ and a Lie coalgebra structure $(\mathfrak{g}, \delta)$ satisfying the following compatibility:
    \begin{align}\label{lie-bialg-comp}
        \delta ( [x, y]_\mathfrak{g}) = (\mathrm{ad}_x \otimes \mathrm{Id}_\mathfrak{g} + \mathrm{Id}_\mathfrak{g} \otimes \mathrm{ad}_x) \delta (y) - (\mathrm{ad}_y \otimes \mathrm{Id}_\mathfrak{g} + \mathrm{Id}_\mathfrak{g} \otimes \mathrm{ad}_y) \delta (x), \text{ for all } x, y \in \mathfrak{g}.
    \end{align}
\end{defn}

\begin{defn}\label{defn-nij-lie-bi}
    A {\bf Nijenhuis Lie bialgebra} is a tuple $(\mathfrak{g}, [~,~]_\mathfrak{g}, N, \delta, S)$ consisting of a Nijenhuis Lie algebra $(\mathfrak{g}, [~,~]_\mathfrak{g}, N)$ and a Nijenhuis Lie coalgebra $(\mathfrak{g}, \delta, S)$ both defined on a same vector space $\mathfrak{g}$ such that the following compatibility conditions are hold:
    \begin{itemize}
        \item[(i)] $(\mathfrak{g}, [~,~]_\mathfrak{g}, \delta)$ is a Lie bialgebra, i.e. the identity (\ref{lie-bialg-comp}) holds,
        \item[(ii)] $S : \mathfrak{g} \rightarrow \mathfrak{g}$ is admissible to the Nijenhuis Lie algebra $(\mathfrak{g}, [~,~]_\mathfrak{g}, N)$, i.e. the identity (\ref{adm-first}) holds,
        \item[(iii)] $N^* : \mathfrak{g}^* \rightarrow \mathfrak{g}^*$ is admissible to the Nijenhuis Lie algebra $(\mathfrak{g}^*, [~,~]_{\mathfrak{g}^*}, S^*)$, i.e. the identity (\ref{adm-sec}) holds.
    \end{itemize}
\end{defn}

In the following result, we show that Nijenhuis Lie algebras can be equivalently characterized by matched pairs of Nijenhuis Lie algebras. More precisely, we have the following result.

\begin{prop}\label{last-prop2}
    Let $(\mathfrak{g}, [~,~]_\mathfrak{g}, N)$ be a Nijenhuis Lie algebra and $(\mathfrak{g}, \delta, S)$ be a Nijenhuis Lie coalgebra both defined on a same finite-dimensional vector space $\mathfrak{g}$. Then $(\mathfrak{g}, [~,~]_\mathfrak{g}, N, \delta, S)$ is a Nijenhuis Lie bialgebra if and only if $ ( (\mathfrak{g}, [~,~]_\mathfrak{g}, N), (\mathfrak{g}^*, [~,~]_{\mathfrak{g}^*}, S^*), \mathrm{ad}^*_\mathfrak{g}, \mathrm{ad}^*_{\mathfrak{g}^*})$ is a matched pair of Nijenhuis Lie algebras.
\end{prop}

\begin{proof}
    It is well-known that $(\mathfrak{g}, [~,~]_\mathfrak{g}, \delta)$ is a Lie bialgebra if and only if $( (\mathfrak{g}, [~,~]_\mathfrak{g}), (\mathfrak{g}^*, [~,~]_{\mathfrak{g}^*}), \mathrm{ad}^*_\mathfrak{g}, \mathrm{ad}^*_{\mathfrak{g}^*})$ is a matched pair of Lie algebras \cite{chari}. Next, the linear map $S : \mathfrak{g} \rightarrow \mathfrak{g}$ is admissible to the Nijenhuis Lie algebra $(\mathfrak{g}, [~,~]_\mathfrak{g}, N)$, i.e. the identity (\ref{adm-first}) holds if and only if
    \begin{align*}
        (\mathrm{ad}^*_\mathfrak{g})_{N(x)} S^* (\alpha) = S^* \big(  (\mathrm{ad}^*_\mathfrak{g})_{N(x)} \alpha + (\mathrm{ad}^*_\mathfrak{g})_x S^* (\alpha) - S^* ( (\mathrm{ad}^*_\mathfrak{g})_x \alpha )    \big),
    \end{align*}
    for all $x \in \mathfrak{g}$ and $\alpha \in \mathfrak{g}^*$. This is equivalent that the triple $(\mathfrak{g}^*, \mathrm{ad}^*_\mathfrak{g}, S^*)$ is a Nijenhuis representation of the Nijenhuis Lie algebra $(\mathfrak{g}, [~,~]_\mathfrak{g}, N)$. Similarly, the map $N^* : \mathfrak{g}^* \rightarrow \mathfrak{g}^*$ is admissible to the Nijenhuis Lie algebra $(\mathfrak{g}^*, [~,~]_{\mathfrak{g}^*}, S^*)$ if and only if the triple $(\mathfrak{g}, \mathrm{ad}^*_{\mathfrak{g}^*}, N)$ is a Nijenhuis representation of the Nijenhuis Lie algebra $(\mathfrak{g}^*, [~,~]_{\mathfrak{g}^*}, S^*)$. Hence the result follows.
\end{proof}

Combining Propositions \ref{last-prop1} and \ref{last-prop2}, we get the following equivalent characterizations of Nijenhuis Lie bialgebras.

\begin{thm}
    Let $(\mathfrak{g}, [~,~]_\mathfrak{g}, N)$ be a Nijenhuis Lie algebra and $(\mathfrak{g}, \delta, S)$ be a Nijenhuis Lie coalgebra both defined on a finite-dimensional vector space $\mathfrak{g}$. Then the following are equivalent:
    \begin{itemize}
        \item[(i)] $(\mathfrak{g}, [~,~]_\mathfrak{g}, N, \delta, S)$ is a Nijenhuis Lie bialgebra,
        \item[(ii)] $ ( (\mathfrak{g}, [~,~]_\mathfrak{g}, N), (\mathfrak{g}^*, [~,~]_{\mathfrak{g}^*}, S^*), \mathrm{ad}^*_\mathfrak{g}, \mathrm{ad}^*_{\mathfrak{g}^*})$ is a matched pair of Nijenhuis Lie algebras,
        \item[(iii)] $((\mathfrak{g} \oplus \mathfrak{g}^*, [~,~]_\Join, N \oplus S^*), (\mathfrak{g}, [~,~]_\mathfrak{g}, N), (\mathfrak{g}^*, [~,~]_{\mathfrak{g}^*}, S^*)  )$ is a Manin triple of Nijenhuis Lie algebras.
    \end{itemize}
\end{thm}

\begin{remark}
    In \cite{ravanpak} Ravanpak introduced the notion of an NL bialgebra as the algebraic analogue of Poisson-Nijenhuis structures. More precisely, an NL bialgebra is a quadruple $(\mathfrak{g}, [~,~]_\mathfrak{g}, \delta, N)$ consisting of a Lie bialgebra $(\mathfrak{g}, [~,~]_\mathfrak{g}, \delta)$ endowed with a linear map $N : \mathfrak{g} \rightarrow \mathfrak{g}$ such that $N$ is a Nijenhuis operator on both the Lie algebra $(\mathfrak{g}, [~,~]_\mathfrak{g})$ and the Lie coalgebra $(\mathfrak{g}, \delta)$ satisfying some compatibility conditions. It turns out that if $(\mathfrak{g}, [~,~]_\mathfrak{g}, \delta, N)$ is an NL bialgebra then the tuple $(\mathfrak{g}, [~,~]_\mathfrak{g}, N, \delta, N)$ is a Nijenhuis Lie bialgebra in the sense of Definition \ref{defn-nij-lie-bi}.  Thus, our notion of a Nijenhuis Lie bialgebra is more general than NL bialgebra considered in \cite{ravanpak}.
\end{remark}

\subsection{Coboundary Nijenhuis Lie bialgebras and admissible CYBE} In this subsection, we consider a particular class of Nijenhuis Lie bialgebras, called {\em coboundary} Nijenhuis Lie bialgebras. In particular, we introduce the admissible classical Yang-Baxter equation (admissible CYBE) whose antisymmetric solutions can be used to construct Nijenhuis Lie bialgebras. In the end, we define relative Rota-Baxter operators or $\mathcal{O}$-operators on Nijenhuis Lie algebras that yield antisymmetric solutions of the admissible CYBE, and hence produce Nijenhuis Lie bialgebras.

\begin{defn}
    Let $(\mathfrak{g}, [~, ~]_\mathfrak{g}, N, \delta, S)$ be a Nijenhuis Lie bialgebra. It is said to be {\bf coboundary} if the underlying Lie bialgebra is coboundary \cite{chari}, i.e. there exists an element $r \in \mathfrak{g} \otimes \mathfrak{g}$ such that
    \begin{align}\label{coboun}
        \delta (x) = \delta_r (x) := ( (\mathrm{ad}_\mathfrak{g})_x \otimes \mathrm{Id}_\mathfrak{g} + \mathrm{Id}_\mathfrak{g} \otimes (\mathrm{ad}_\mathfrak{g})_x ) r, \text{ for all } x \in \mathfrak{g}.
    \end{align}
\end{defn}

\medskip

Let $(\mathfrak{g}, [~,~]_\mathfrak{g})$ be a Lie algebra and $r = \sum_i p_i \otimes q_i$ be any element of $\mathfrak{g} \otimes \mathfrak{g}$. We define a map $\delta : \mathfrak{g} \rightarrow \mathfrak{g} \otimes \mathfrak{g}$ by the equation (\ref{coboun}). Then it is easy to see that the map $\delta$ satisfies the condition (\ref{lie-bialg-comp}). Further, it is well-known that $(\mathfrak{g}, \delta)$ is a Lie coalgebra which in turn implies that $(\mathfrak{g}, [~,~]_\mathfrak{g}, \delta)$ is a Lie bialgebra if and only if for all $x \in \mathfrak{g}$,
\begin{align}
     &( (\mathrm{ad}_\mathfrak{g})_x \otimes \mathrm{Id}_\mathfrak{g} + \mathrm{Id}_\mathfrak{g} \otimes (\mathrm{ad}_\mathfrak{g})_x ) (r + \tau (r)) = 0, \label{ad-inf1}\\
     \big(  (\mathrm{ad}_\mathfrak{g})_x \otimes \mathrm{Id}_\mathfrak{g} \otimes \mathrm{Id}_\mathfrak{g} ~ \! + & ~ \! \mathrm{Id}_\mathfrak{g} \otimes (\mathrm{ad}_\mathfrak{g})_x  \otimes \mathrm{Id}_\mathfrak{g} ~ \! + ~ \!  \mathrm{Id}_\mathfrak{g}  \otimes  \mathrm{Id}_\mathfrak{g}  \otimes   (\mathrm{ad}_\mathfrak{g})_x   \big)  ( \llbracket r_{12}, r_{13} \rrbracket_\mathfrak{g} + \llbracket r_{12}, r_{23} \rrbracket_\mathfrak{g} + \llbracket r_{13}, r_{23} \rrbracket_\mathfrak{g} )= 0, \label{ad-inf2}
\end{align}
where
\begin{align*}
 \llbracket r_{12}, r_{13} \rrbracket_\mathfrak{g} :=~&  \sum_{i, j} [p_i, p_j ]_\mathfrak{g} \otimes q_i \otimes q_j, \qquad 
    \llbracket r_{12}, r_{23} \rrbracket_\mathfrak{g} := \sum_{i, j} p_i \otimes [q_i, p_j]_\mathfrak{g} \otimes q_j \\
 &\text{ and } \quad     \llbracket r_{13}, r_{23} \rrbracket_\mathfrak{g} :=  \sum_{i, j} p_i \otimes p_j \otimes [q_i, q_j ]_\mathfrak{g}.
\end{align*}

\begin{thm}
    Let $(\mathfrak{g}, [~,~]_\mathfrak{g}, N)$ be a Nijenhuis Lie algebra and $S : \mathfrak{g} \rightarrow \mathfrak{g}$ be a linear map admissible to the Nijenhuis Lie algebra $(\mathfrak{g}, [~,~]_\mathfrak{g}, N)$. For any element $r \in \mathfrak{g} \otimes \mathfrak{g}$, define a map $\delta : \mathfrak{g} \rightarrow \mathfrak{g} \otimes \mathfrak{g}$ by the equation (\ref{coboun}). Then $(\mathfrak{g}, [~,~]_\mathfrak{g}, N, \delta, S)$ is a Nijenhuis Lie bialgebra if and only if (\ref{ad-inf1}), (\ref{ad-inf2}) and for all $x \in \mathfrak{g}$, the following conditions hold:
    \begin{align}\label{three-star}
       & \big( \mathrm{Id}_\mathfrak{g} \otimes S(\mathrm{ad}_\mathfrak{g})_x - \mathrm{Id}_\mathfrak{g} \otimes (\mathrm{ad}_\mathfrak{g} )_{S (x)} \big) (S \otimes \mathrm{Id}_\mathfrak{g} - \mathrm{Id}_\mathfrak{g} \otimes N) (r)  \\
        & \qquad  = \big( S (\mathrm{ad}_\mathfrak{g})_x \otimes \mathrm{Id}_\mathfrak{g} - (\mathrm{ad}_\mathfrak{g} )_{S (x)} \otimes \mathrm{Id}_\mathfrak{g} \big) (N \otimes \mathrm{Id}_\mathfrak{g} - \mathrm{Id}_\mathfrak{g} \otimes S) (r), \nonumber
    \end{align}
    \begin{align}\label{four-star}
        \big(  \mathrm{Id}_\mathfrak{g} \otimes (\mathrm{ad}_\mathfrak{g})_{N (x)} +~& (\mathrm{ad}_\mathfrak{g})_{N (x)} \otimes \mathrm{Id}_\mathfrak{g} + \mathrm{Id}_\mathfrak{g} \otimes S (\mathrm{ad}_\mathfrak{g})_x - N (\mathrm{ad}_\mathfrak{g})_x \otimes \mathrm{Id}_\mathfrak{g} - N \otimes (\mathrm{ad}_\mathfrak{g})_x   \big) ( N \otimes \mathrm{Id}_\mathfrak{g} - \mathrm{Id}_\mathfrak{g} \otimes S )(r) \\
       & = ( \mathrm{Id}_\mathfrak{g} \otimes (\mathrm{ad}_\mathfrak{g})_x) ( \mathrm{Id}_\mathfrak{g} \otimes S) ( N \otimes \mathrm{Id}_\mathfrak{g} - \mathrm{Id}_\mathfrak{g} \otimes S ) (r) = 0.
    \end{align}
\end{thm}

\begin{proof}
    It has been already recalled that the triple $(\mathfrak{g}, [~, ~]_\mathfrak{g}, \delta)$ is a Lie bialgebra if and only if the conditions (\ref{ad-inf1}) and (\ref{ad-inf2}) hold. Next, for any $x \in \mathfrak{g}$, we observe that
    \begin{align*}
        &(S \otimes S) \delta (x) - (S \otimes \mathrm{Id}_\mathfrak{g} + \mathrm{Id}_\mathfrak{g} \otimes S) \delta (S(x)) + \delta (S^2 (x)) \\
        &= S [x, p_i ]_\mathfrak{g} \otimes S (q_i) + S (p_i) \otimes S [x, q_i]_\mathfrak{g}- S [S(x), p_i]_\mathfrak{g} \otimes q_i - S (p_i) \otimes [S (x), q_i ]_\mathfrak{g} \\
        & \qquad  - [S(x), p_i]_\mathfrak{g} \otimes S (q_i) - p_i \otimes S [S (x), q_i]_\mathfrak{g} + [S^2 (x), p_i]_\mathfrak{g} \otimes q_i + p_i \otimes [S^2 (x), q_i]_\mathfrak{g} \\
        &=  S [x, p_i ]_\mathfrak{g} \otimes S (q_i) + S (p_i) \otimes S [x, q_i]_\mathfrak{g} - S [x, N(p_i)]_\mathfrak{g} \otimes q_i + [S(x), N (p_i)]_\mathfrak{g} \otimes q_i \\
        & \qquad - S (p_i) \otimes [S (x), q_i]_\mathfrak{g} - [S (x), p_i]_\mathfrak{g} \otimes S (q_i) - p_i \otimes S [x, N (q_i)]_\mathfrak{g} + p_i \otimes [S(x), N (q_i)]_\mathfrak{g} \\
        &= \big( \mathrm{Id}_\mathfrak{g} \otimes S(\mathrm{ad}_\mathfrak{g})_x - \mathrm{Id}_\mathfrak{g} \otimes (\mathrm{ad}_\mathfrak{g} )_{S (x)} \big) (S \otimes \mathrm{Id}_\mathfrak{g} - \mathrm{Id}_\mathfrak{g} \otimes N) (r)  \\
        & \qquad  - \big( S (\mathrm{ad}_\mathfrak{g})_x \otimes \mathrm{Id}_\mathfrak{g} - (\mathrm{ad}_\mathfrak{g} )_{S (x)} \otimes \mathrm{Id}_\mathfrak{g} \big) (N \otimes \mathrm{Id}_\mathfrak{g} - \mathrm{Id}_\mathfrak{g} \otimes S) (r).
    \end{align*}
This shows that $S$ is a Nijenhuis operator on the Lie coalgebra $(\mathfrak{g}, \delta)$ if and only if the condition (\ref{three-star}) holds. On the other hand, for $x \in \mathfrak{g}$, we also have
\begin{align*}
    &(N \otimes S) \delta (x) + (N \otimes \mathrm{Id}_\mathfrak{g} - \mathrm{Id}_\mathfrak{g} \otimes S) \delta (N (x)) - (N^2 \otimes \mathrm{Id}_\mathfrak{g}) \delta (x) \\
    &= N [x, p_i ]_\mathfrak{g} \otimes S (q_i) + N (p_i) \otimes S [x, q_i ]_\mathfrak{g} + N [N (x), p_i]_\mathfrak{g} \otimes q_i + N (p_i) \otimes [N (x), q_i ]_\mathfrak{g} \\
    & \qquad -[N(x), p_i]_\mathfrak{g} \otimes S (q_i) - p_i \otimes S [N (x), q_i]_\mathfrak{g} - N^2 ([x, p_i]_\mathfrak{g}) \otimes q_i - N^2 (p_i) \otimes [x, q_i]_\mathfrak{g} \\
    &=  \big(  \mathrm{Id}_\mathfrak{g} \otimes (\mathrm{ad}_\mathfrak{g})_{N (x)} + (\mathrm{ad}_\mathfrak{g})_{N (x)} \otimes \mathrm{Id}_\mathfrak{g} + \mathrm{Id}_\mathfrak{g} \otimes S (\mathrm{ad}_\mathfrak{g})_x - N (\mathrm{ad}_\mathfrak{g})_x \otimes \mathrm{Id}_\mathfrak{g} - N \otimes (\mathrm{ad}_\mathfrak{g})_x   \big) ( N \otimes \mathrm{Id}_\mathfrak{g} - \mathrm{Id}_\mathfrak{g} \otimes S )(r) \\
       & \qquad - ( \mathrm{Id}_\mathfrak{g} \otimes (\mathrm{ad}_\mathfrak{g})_x) ( \mathrm{Id}_\mathfrak{g} \otimes S) ( N \otimes \mathrm{Id}_\mathfrak{g} - \mathrm{Id}_\mathfrak{g} \otimes S )(r)
\end{align*}
Therefore, $N^*$ is admissible to the Nijenhuis Lie algebra $(\mathfrak{g}^*, [~,~]_{\mathfrak{g}^*}, S^*)$, i.e. the identity (\ref{adm-sec}) holds if and only if the condition (\ref{four-star}) holds. Hence the conclusion follows.
\end{proof}

Let $(\mathfrak{g}, [~,~]_\mathfrak{g}, N)$ be a Nijenhuis Lie algebra and $S : \mathfrak{g} \rightarrow \mathfrak{g}$ be a linear map admissible to the Nijenhuis Lie algebra $(\mathfrak{g}, [~,~]_\mathfrak{g}, N)$. Let $r \in \mathfrak{g} \otimes \mathfrak{g}$ be an element satisfying (\ref{ad-inf1}) and the equations:
\begin{align}
   \llbracket r_{12}, r_{13} \rrbracket_\mathfrak{g} + \llbracket r_{12}, r_{23} \rrbracket_\mathfrak{g} + \llbracket r_{13}, r_{23} \rrbracket_\mathfrak{g} =~& 0,  \label{acybe1} \\
   (N \otimes \mathrm{Id}_\mathfrak{g} - \mathrm{Id}_\mathfrak{g} \otimes S)(r) =~& 0, \label{acybe2} \\
    (S \otimes \mathrm{Id}_\mathfrak{g} - \mathrm{Id}_\mathfrak{g} \otimes N) (r) =~& 0. \label{acybe3}
\end{align}
Then it follows from the above theorem that the tuple $(\mathfrak{g}, [~,~]_\mathfrak{g}, N, \delta, S)$ is a Nijenhuis Lie bialgebra.

Given a Lie algebra $(\mathfrak{g}, [~,~]_\mathfrak{g})$, the equation (\ref{acybe1}) is said to be the {\bf classical Yang-Baxter equation} (CYBE) in the Lie algebra. Generalizing this concept in the context of Nijenhuis Lie algebras, we obtain the following.

\begin{defn}
    Let $(\mathfrak{g}, [~,~]_\mathfrak{g}, N)$ be a Nijenhuis Lie algebra and $S : \mathfrak{g} \rightarrow \mathfrak{g}$ be a linear map admissible to the Nijenhuis Lie algebra $(\mathfrak{g}, [~,~]_\mathfrak{g}, N)$. Then the equation (\ref{acybe1}) together with the equations (\ref{acybe2}), (\ref{acybe3}) is called the {\bf admissible classical Yang-Baxter equation} (admissible CYBE) in the Nijenhuis Lie algebra $(\mathfrak{g}, [~,~]_\mathfrak{g}, N)$ and for the admissible map $S$.
\end{defn}

With the above definition, we now obtain the following result.

\begin{prop}\label{prop-la}
    Let $(\mathfrak{g}, [~,~]_\mathfrak{g}, N)$ be a Nijenhuis Lie algebra and $S : \mathfrak{g} \rightarrow \mathfrak{g}$ be a linear map admissible to the Nijenhuis Lie algebra $(\mathfrak{g}, [~,~]_\mathfrak{g}, N)$. Let $r \in \mathfrak{g} \otimes \mathfrak{g}$ be an antisymmetric solution of the admissible classical Yang-Baxter equation. Then $(\mathfrak{g}, [~,~]_\mathfrak{g}, N, \delta, S)$ is a Nijenhuis Lie bialgebra, where $\delta$ is given by (\ref{coboun}).
\end{prop}

\medskip

It is well-known that a relative Rota-Baxter operator on a Lie algebra gives rise to an antisymmetric solution of CYBE \cite{bai-jpa}. More precisely, let $(\mathfrak{g}, [~,~]_\mathfrak{g})$ be a Lie algebra and $(\mathcal{V}, \rho)$ be a representation of it. Suppose $r: \mathcal{V} \rightarrow \mathfrak{g}$ is any linear map which we consider as an element of $(\mathfrak{g} \oplus \mathcal{V}^*) \otimes (\mathfrak{g} \oplus \mathcal{V}^*)$ through the identification $\mathrm{Hom} (\mathcal{V}, \mathfrak{g}) \cong \mathfrak{g} \otimes \mathcal{V}^* \subset (\mathfrak{g} \oplus \mathcal{V}^*) \otimes (\mathfrak{g} \oplus \mathcal{V}^*)$. Then $r$ is a relative Rota-Baxter operator if and only if the element $r - \tau (r)$ is an antisymmetric solution of the CYBE in the semidirect product Lie algebra $(\mathfrak{g} \oplus \mathcal{V}^*, [~,~]_\ltimes)$. To generalize this result in the context of Nijenhuis Lie algebras, we first consider the following.

\begin{defn}
    Let $(\mathfrak{g}, [~,~]_\mathfrak{g}, N)$ be a Nijenhuis Lie algebra and $(\mathcal{V}, \rho, S)$ be a Nijenhuis representation of it. A {\em relative Rota-Baxter operator} or an {\em $\mathcal{O}$-operator} associated to the Nijenhuis representation $(\mathcal{V}, \rho, S)$ is a linear map $r : \mathcal{V} \rightarrow \mathfrak{g}$ satisfying $N \circ r = r \circ S$ and the identity (\ref{rrb}).
\end{defn}

\begin{thm}
    Let $(\mathfrak{g}, [~,~]_\mathfrak{g}, N)$ be a Nijenhuis Lie algebra and $(\mathcal{V}, \rho, S)$ be a Nijenhuis representation of it. Let $Q : \mathfrak{g} \rightarrow \mathfrak{g}$ be any linear map and $\beta : \mathcal{V} \rightarrow \mathcal{V}$ be a linear map admissible to the Nijenhuis Lie algebra $(\mathfrak{g}, [~,~]_\mathfrak{g}, N)$ and the Lie algebra representation $(\mathcal{V}, \rho)$. Suppose $r :\mathcal{V} \rightarrow \mathfrak{g}$ is a linear map which we regard as an element of $(\mathfrak{g} \oplus \mathcal{V}^*) \otimes (\mathfrak{g} \oplus \mathcal{V}^*)$. Then the following are equivalent:
    \begin{itemize}
        \item[(1)] $r$ is a relative Rota-Baxter operator associated to the Nijenhuis representation $(\mathcal{V}, \rho, S)$ satisfying additionally $r \circ \beta = Q \circ r$,
        \item[(2)] $r - \tau (r)$ is an antisymmetric solution of the admissible CYBE in the semidirect product Nijenhuis Lie algebra $(\mathfrak{g} \oplus \mathcal{V}^*, [~, ~]_\ltimes, N \oplus \beta^*)$ and for the admissible map $Q \oplus S^*$.
    \end{itemize}
    In either case, $(\mathfrak{g} \oplus \mathcal{V}^*, [~,~]_\ltimes , N \oplus \beta^*, \delta , Q \oplus S^*)$ is a Nijenhuis Lie bialgebra, where the linear map $\delta : \mathfrak{g} \oplus \mathcal{V}^* \rightarrow ( \mathfrak{g} \oplus \mathcal{V}^*) \otimes ( \mathfrak{g} \oplus \mathcal{V}^*)$ is defined by $\delta = \delta_{r - \tau (r)}$.
\end{thm}

\begin{proof}
    It has been recalled that $r$ satisfies the condition (\ref{rrb}) if and only if $r -\tau (r)$ is an antisymmetric solution of the CYBE in the Lie algebra $(\mathfrak{g} \oplus \mathcal{V}^*, [~,~]_\ltimes)$. Next, it is easy to see that $r$ satisfies $N \circ r = r \circ S$ and $r \circ \beta = Q \circ r$ if and only if the map $Q \oplus S^* : \mathfrak{g} \oplus \mathcal{V}^* \rightarrow \mathfrak{g} \oplus \mathcal{V}^*$ is admissible to the Nijenhuis Lie algebra $(\mathfrak{g} \oplus \mathcal{V}^*, [~, ~]_\ltimes, N \oplus \beta^*)$ and the element $r - \tau (r)$ satisfies
    \begin{align*}
        ( (N \oplus \beta^*) \otimes \mathrm{Id}_{\mathfrak{g} \oplus \mathcal{V}^*} - \mathrm{Id}_{\mathfrak{g} \oplus \mathcal{V}^*} \otimes (Q \oplus S^*)) (r - \tau (r)) = 0,\\
        ( (Q \oplus S^*) \otimes \mathrm{Id}_{\mathfrak{g} \oplus \mathcal{V}^*} - \mathrm{Id}_{\mathfrak{g} \oplus \mathcal{V}^*} \otimes (N \oplus \beta^*)) (r - \tau (r)) = 0.
    \end{align*}
    Hence the first part follows. The last part is a consequence of Proposition \ref{prop-la}.
\end{proof}

\subsection{NS-Lie algebras} Nijenhuis Lie algebras are closely related to NS-Lie algebras \cite{das-twisted}. In the last part of this paper, we obtain some important results including representations and matched pairs of NS-Lie algebras and relate them with the corresponding notions for Nijenhuis Lie algebras. Various other results on NS-Lie algebras including their cohomology and possible bialgebra theory will be discussed in a separate article.

\begin{defn}
    An {\bf NS-Lie algebra} is a triple $(\mathfrak{p}, \diamond, \lfloor ~, ~ \rfloor)$ consisting of a vector space $\mathfrak{p}$ endowed with two bilinear operations $\diamond, \lfloor ~, ~ \rfloor : \mathfrak{p} \times \mathfrak{p} \rightarrow \mathfrak{p}$ in which the operation $\lfloor ~,~ \rfloor$ is antisymmetric and satisfy the following identities:
    \begin{align}
        &(x \diamond y ) \diamond z - x \diamond (y  \diamond z ) + \lfloor x, y \rfloor \diamond z = (y \diamond x ) \diamond z - y \diamond (x  \diamond z ), \label{nsl1} \tag{NSL1}\\
        \lfloor x, ~ \! &\llbracket y, z \rrbracket \rfloor + \lfloor y, \llbracket z, x \rrbracket \rfloor + \lfloor z, \llbracket x, y \rrbracket \rfloor + x \diamond \lfloor y, z \rfloor + y \diamond \lfloor z, x \rfloor + z \diamond \lfloor x, y \rfloor = 0, \label{nsl2} \tag{NSL2}
    \end{align}
    for $x, y, z \in \mathfrak{p}$. Here we use the notation
    \begin{align}\label{subadjacent-ns}
        \llbracket x, y \rrbracket := x \diamond y - y \diamond x + \lfloor x, y \rfloor, \text{ for } x, y \in \mathfrak{p}.
    \end{align}
\end{defn}

An NS-Lie algebra $(\mathfrak{p}, \diamond, \lfloor ~, ~ \rfloor)$ for which the operation $\diamond$ is trivial turns out to be a Lie algebra. On the other hand, if the operation $\lfloor ~,~ \rfloor$ is trivial then $(\mathfrak{p}, \diamond)$ becomes a pre-Lie algebra. Thus, an NS-Lie algebra unifies both Lie algebras and pre-Lie algebras. In general, an arbitrary NS-Lie algebra $(\mathfrak{p}, \diamond, \lfloor ~, ~ \rfloor)$ gives rise to a Lie algebra structure on the underlying vector space $\mathfrak{p}$ with the bracket $\llbracket ~,~ \rrbracket$ defined in (\ref{subadjacent-ns}). The Lie algebra $(\mathfrak{p}, \llbracket ~, ~\rrbracket)$ is said to be the {\em subadjacent Lie algebra} of the NS-Lie algebra $(\mathfrak{p}, \diamond, \lfloor ~, ~ \rfloor)$.

\begin{prop}
    Let $(\mathfrak{g}, [~,~]_\mathfrak{g}, N)$ be a Nijenhuis Lie algebra. Then the vector space $\mathfrak{g}$ inherits an NS-Lie algebra structure with the operations
    \begin{align*}
        x \diamond_N y := [N(x), y]_\mathfrak{g} ~~~~ \text{ and } ~~~~ \lfloor x, y \rfloor_N := - N [x, y]_\mathfrak{g}, \text{ for } x, y \in \mathfrak{g}.
    \end{align*}
    The NS-Lie algebra $(\mathfrak{g}, \diamond_N, \lfloor ~,~\rfloor_N)$ is said to be induced from the Nijenhuis Lie algebra  $(\mathfrak{g}, [~,~]_\mathfrak{g}, N)$.
\end{prop}

\begin{defn}
    Let $(\mathfrak{p}, \diamond, \lfloor ~, ~ \rfloor)$ be an NS-Lie algebra. A {\bf representation} of this NS-Lie algebra is a vector space $\mathcal{V}$ equipped with three linear operations $l, r, \psi : \mathfrak{p} \rightarrow \mathrm{End}(\mathcal{V})$ subject to satisfy the following conditions:
    \begin{align}
        l_{x \diamond y} - l_x l_y + l_{ \lfloor x, y \rfloor } =~& l_{y \diamond x} - l_y l_x, \\
        r_{x \diamond y} - r_y r_x + r_y \psi_x =~& l_x r_y - r_y l_x,\\
        \psi_{ \llbracket x, y \rrbracket } -  r_{\lfloor x, y \rfloor} =~& l_x \psi_y - l_y \psi_x + \psi_x (l_y - r_y + \psi_y) - \psi_y (l_x - r_x + \psi_x),
    \end{align}
    for all $x, y \in \mathfrak{p}$. A representation as above is often denoted by the quadruple $(\mathcal{V}, l, r, \psi)$ or simply by $\mathcal{V}$ when the operations are understood.
\end{defn}

It is easy to see that the notion of representations of an NS-Lie algebra unifies representations of both Lie algebras and pre-Lie algebras. Note that any NS-Lie algebra $(\mathfrak{p}, \diamond, \lfloor ~, ~ \rfloor)$ has a natural representation on the vector space $\mathfrak{p}$ itself, where the maps $l, r, \psi : \mathfrak{p} \rightarrow \mathrm{End} (\mathfrak{p})$ are respectively given by $l_x (y) = x \diamond y$, $r_x (y) = y \diamond x$ and $\psi_x (y) = \lfloor x, y \rfloor$, for $x, y \in \mathfrak{p}$. This is called the {\em adjoint representation} or the {\em regular representation}.

\begin{prop}\label{prop-semid-ns}
    Let $(\mathfrak{p}, \diamond, \lfloor ~, ~ \rfloor)$ be an NS-Lie algebra and $(\mathcal{V}, l, r, \psi)$ be a representation of it. Then the direct sum $\mathfrak{p} \oplus \mathcal{V}$ inherits an NS-Lie algebra structure with the operations
    \begin{align}\label{semid-ns}
        (x, u) \diamond_\ltimes (y, v) := (x \diamond y ~\!, ~\! l_x v + r_y u) ~~~~ \text{ and } ~~~~ \lfloor (x, u), (y, v) \rfloor_\ltimes := ( \lfloor x, y \rfloor ~\! , ~\!  \psi_x v - \psi_y u),
    \end{align}
    for $(x, u), (y, v) \in \mathfrak{p} \oplus \mathcal{V}$. This is called the {\em semidirect product}.
\end{prop}

The proof of the above proposition is straightforward. Here we omit the proof as we will discuss a more general result in Theorem \ref{bicrossed-ns}. It is important to remark that the converse of the above proposition is also true. More precisely, let $(\mathfrak{p}, \diamond, \lfloor ~, ~\rfloor)$ be an NS-Lie algebra and $\mathcal{V}$ be any vector space with the linear maps $l, r, \psi : \mathfrak{p} \rightarrow \mathrm{End} (\mathcal{V})$. Suppose the space $\mathfrak{p} \oplus \mathcal{V}$ endowed with the operations defined in (\ref{semid-ns}) forms an NS-Lie algebra. Then the quadruple $(\mathcal{V}, l, r, \psi)$ is a representation of the NS-Lie algebra $(\mathfrak{p}, \diamond, \lfloor ~, ~\rfloor)$.

\begin{prop}
    Let $(\mathfrak{p}, \diamond, \lfloor ~, ~ \rfloor)$ be an NS-Lie algebra and $(\mathcal{V}, l, r, \psi)$ be a representation of it. Define a map $\rho : \mathfrak{p} \rightarrow \mathrm{End} (\mathcal{V})$ by $\rho_x (v) := (l_x - r_x + \psi_x) v$, for $x \in \mathfrak{p}$ and $v \in \mathcal{V}$. Then $\rho$ defines a representation of the subadjacent Lie algebra $(\mathfrak{p}, \llbracket ~, ~ \rrbracket)$ on the vector space $\mathcal{V}$.
\end{prop}

\begin{proof}
    Since $(\mathfrak{p}, \diamond, \lfloor ~, ~ \rfloor)$ is an NS-Lie algebra and $(\mathcal{V}, l, r, \psi)$ is a representation of it, we have the semidirect product NS-Lie algebra $(\mathfrak{p} \oplus \mathcal{V}, \diamond_\ltimes, \lfloor ~, ~ \rfloor_\ltimes)$ given in Proposition \ref{prop-semid-ns}. The corresponding subadjacent Lie algebra is then given by $(\mathfrak{p} \oplus \mathcal{V}, \llbracket ~, ~ \rrbracket_\ltimes)$, where
    \begin{align*}
        \llbracket (x, u), (y, v) \rrbracket_\ltimes =~& (x, u) \diamond_\ltimes (y, v) - (y, v) \diamond_\ltimes (x, u) + \lfloor (x, u), (y, v) \rfloor_\ltimes \\
        =~& \big( \llbracket x, y \rrbracket ~ \!, ~ \! (l_x - r_x + \psi_x ) v - (l_y - r_y + \psi_y) u   \big),
    \end{align*}
    for $(x, u), (y, v) \in \mathfrak{p} \oplus \mathcal{V}$. The above expression of the Lie bracket shows that $\rho = l -r +\psi$ defines a representation of the subadjacent Lie algebra $(\mathfrak{p}, \llbracket ~, ~ \rrbracket)$ on the vector space $\mathcal{V}$.
\end{proof}

The following result shows that a Nijenhuis representation of a Nijenhuis Lie algebra gives rise to a representation of the induced NS-Lie algebra. More precisely, we have the following.

\begin{prop}\label{gen-prop}
    Let $(\mathfrak{g}, [~,~]_\mathfrak{g}, N)$ be a Nijenhuis Lie algebra and $(\mathcal{V}, \rho, S)$ be a Nijenhuis representation of it. Then the quadruple $(\mathcal{V}, l, r, \psi)$ is a representation of the induced NS-Lie algebra $(\mathfrak{g}, \diamond_N, \lfloor ~, ~\rfloor_N)$, where
    \begin{align*}
        l_x (v) := \rho_{N(x)} v, \quad r_x (v) := - \rho_x S(v) ~~~~ \text{ and } ~~~~ \psi_x (v) := - S (\rho_x v), \text{ for } x \in \mathfrak{g}, v \in \mathcal{V}.
    \end{align*}
\end{prop}

\begin{proof}
    Since $(\mathcal{V}, \rho, S)$ is a Nijenhuis representation of the Nijenhuis Lie algebra $(\mathfrak{g}, [~,~]_\mathfrak{g}, N)$, one may consider the corresponding semidirect product Nijenhuis Lie algebra $(\mathfrak{g} \oplus \mathcal{V}, [~,~]_\ltimes, N \oplus S)$. Hence the vector space $\mathfrak{g} \oplus \mathcal{V}$ can be given an NS-Lie algebra structure with the operations
    \begin{align*}
        (x, u) \diamond_{N \oplus S} (y, v) = [ (N (x), S (u)), (y, v) ]_\ltimes =~& \big(  [N (x), y]_\mathfrak{g} ~ \! , ~ \! \rho_{N (x)} v - \rho_y S (u)   \big)\\
        =~& (x \diamond_N y ~ \! , ~ \! l_x v + r_y u),\\
        \lfloor (x, u), (y, v) \rfloor_{N \oplus S} = - (N \oplus S) [(x, u), (y, v)]_\ltimes =~& \big( - N [x, y]_\mathfrak{g} ~ \!, ~ \! - S (\rho_x v) + S (\rho_y u)   \big) \\
        =~& (\lfloor x, y \rfloor_N ~ \! , ~ \! \psi_x v - \psi_y u),
    \end{align*}
    for $(x, u), (y, v) \in \mathfrak{g} \oplus \mathcal{V}$. The above two expressions show that the maps $l, r, \psi$ defines a representation of the induced NS-Lie algebra $(\mathfrak{g}, \diamond_N, \lfloor ~, ~ \rfloor_N)$ on the vector space $\mathcal{V}$.
\end{proof}

Keeping in mind the definition of representations of an NS-Lie algebra, we will now define matched pairs of NS-Lie algebras.

\begin{defn}
    A {\bf matched pair of NS-Lie algebras} is a tuple
    \begin{align*}
         \big(  (\mathfrak{p}_1, \diamond_1, \lfloor ~, ~ \rfloor_1),  (\mathfrak{p}_2, \diamond_2, \lfloor ~, ~ \rfloor_2), l ,r, \psi, L, R, \Psi   \big)
    \end{align*}
    consisting of two NS-Lie algebras $(\mathfrak{p}_1, \diamond_1, \lfloor ~, ~ \rfloor_1)$ and $ (\mathfrak{p}_2, \diamond_2, \lfloor ~, ~ \rfloor_2)$ with linear maps $l, r, \psi : \mathfrak{p}_1 \rightarrow \mathrm{End} (\mathfrak{p}_2)$ and $L, R, \Psi : \mathfrak{p}_2 \rightarrow \mathrm{End} (\mathfrak{p}_1)$ such that
    \begin{itemize}
        \item the quadruple $(\mathfrak{p}_2, l, r, \psi)$ is a representation of the NS-Lie algebra $(\mathfrak{p}_1, \diamond_1, \lfloor ~, ~ \rfloor_1)$,
        \item the quadruple $(\mathfrak{p}_1,L, R, \Psi)$ is a representation of the NS-Lie algebra $(\mathfrak{p}_2, \diamond_2, \lfloor ~, ~ \rfloor_2)$
    \end{itemize}
    and for all $x, y \in \mathfrak{p}_1$, $\alpha, \beta \in \mathfrak{p}_2$, the following compatibility conditions hold:
    \begin{align}
        l_x (\alpha \diamond_2 \beta) =~& (l_x \alpha) \diamond_2 \beta  + \alpha \diamond_2 (l_x \beta) + (\psi_x \alpha) \diamond_2 \beta - r_x (\alpha) \diamond_2 \beta
        + r_{R_\beta x} \alpha - l_{ (L_\alpha - R_\alpha + \Psi_\alpha) x } \beta, \label{mnsl1}\\
        r_x (\llbracket \alpha , \beta \rrbracket_2) =~& \alpha \diamond_2 r_x (\beta) - \beta \diamond_2 r_x (\alpha) + r_{L_\beta x} \alpha - r_{L_\alpha x} \beta, \label{mnsl2}\\
        L_\alpha (x \diamond_1 y) =~& (L_\alpha x) \diamond_1 y + x \diamond_1 (L_\alpha y) + (\Psi_\alpha x) \diamond_1 y - R_\alpha (x) \diamond_1 y + R_{r_y \alpha} x - L_{ (l_x -r_x + \psi_x) \alpha} y, \label{mnsl3}\\
        R_\alpha ( \llbracket x, y \rrbracket_1) =~& x \diamond_1 R_\alpha (y) - y \diamond_1 R_\alpha (x) + R_{l_y \alpha} x - R_{l_x \alpha} y, \label{mnsl4}\\
        l_x (\lfloor \alpha, \beta \rfloor_2) =~& \lfloor (l_x -r_x + \psi_x) (\alpha), \beta \rfloor_2 + \lfloor \alpha, (l_x - r_x + \psi_x )(\beta) \rfloor_2 + \psi_{ (L_\beta - R_\beta + \Psi_\beta) x} \alpha \label{mnsl5}\\
        & \quad - \psi_{ (L_\alpha - R_\alpha + \Psi_\alpha) x} \beta + \alpha \diamond_2 \psi_x \beta - \beta \diamond_2 \psi_x \alpha + r_{\Psi_\alpha x} \beta - r_{\Psi_\beta x} \alpha - \psi_x \llbracket \alpha , \beta \rrbracket_2, \nonumber\\
        L_\alpha ( \lfloor x, y \rfloor_1 ) =~& \lfloor (L_\alpha -R_\alpha + \Psi_\alpha) (x), y \rfloor_1 + \lfloor x, (L_\alpha - R_\alpha + \Psi_\alpha )(y) \rfloor_1 + \Psi_{ (l_y - r_y + \psi_y) \alpha} x \label{mnsl6}\\
        & \quad - \Psi_{ (l_x - r_x + \psi_x) \alpha} y + x \diamond_1 \Psi_\alpha y - y \diamond_1 \Psi_\alpha x + R_{\psi_x \alpha} y - R_{\psi_y \alpha} x - \Psi_\alpha \llbracket x , y \rrbracket_1, \nonumber
    \end{align}
    where $\llbracket x, y \rrbracket_1 = x \diamond_1 y - y \diamond_1 x + \lfloor x, y \rfloor_1$ and $\llbracket \alpha, \beta \rrbracket_2 = \alpha \diamond_2 \beta - \beta \diamond_2 \alpha + \lfloor \alpha, \beta \rfloor_2.$
        \end{defn}

        In the above definition, if the operations $\diamond_1, \diamond_2, l, r, L, R$ are trivial then $((\mathfrak{p}_1, \lfloor ~, ~ \rfloor_1), (\mathfrak{p}_2, \lfloor ~, ~ \rfloor_2), \psi, \Psi )$ forms a matched pair of Lie algebras \cite{majid-lie}. On the other hand, if $\lfloor ~, ~ \rfloor_1$, $\lfloor ~, ~ \rfloor_2$, $\psi, \Psi$ 
 are trivial then $((\mathfrak{p}_1, \diamond_1), (\mathfrak{p}_2, \diamond_2), l, r, L, R)$ becomes a matched pair of pre-Lie algebras \cite{bai-pre}. Therefore, our definition unifies both the matched pair of Lie algebras and the matched pair of pre-Lie algebras. In the following result, we give the bicrossed product construction associated with a given matched pair of NS-Lie algebras. More precisely, we have the following.

        \begin{thm}\label{bicrossed-ns}
            Let  $\big(  (\mathfrak{p}_1, \diamond_1, \lfloor ~, ~ \rfloor_1),  (\mathfrak{p}_2, \diamond_2, \lfloor ~, ~ \rfloor_2), l ,r, \psi, L, R, \Psi   \big)$ be a matched pair of NS-Lie algebras. Then $(\mathfrak{p}_1 \oplus \mathfrak{p}_2, \diamond_\Join, \lfloor ~, ~ \rfloor_\Join)$ is an NS-Lie algebra, where
            \begin{align}
                (x, \alpha) \diamond_\Join (y, \beta) :=~& ( x \diamond_1 y + L_\alpha y + R_\beta x ~ \!, ~ \! \alpha \diamond_2 \beta + l_x \beta + r_y \alpha), \label{bim1} \\
                \lfloor (x, \alpha) , (y, \beta) \rfloor_\Join :=~& (\lfloor x, y \rfloor_1 + \Psi_\alpha y -\Psi_\beta x ~ \!, ~ \! \lfloor \alpha, \beta \rfloor_2 + \psi_x \beta - \psi_y \alpha), \label{bim2}
            \end{align}
            for $(x, \alpha), (y, \beta) \in \mathfrak{p}_1 \oplus \mathfrak{p}_2$. This is called the {\bf bicrossed product} of the given matched pair of NS-Lie algebras.
        \end{thm}

        \begin{proof}
            To show that $(\mathfrak{p}_1 \oplus \mathfrak{p}_2, \diamond_\Join, \lfloor ~, ~ \rfloor_\Join)$ is an NS-Lie algebra, we need to verify the identities (\ref{nsl1}) and (\ref{nsl2}) for the above two operations. Since both the identities are linear in any input, verifying these identities for the elements of the form $(x, 0)$ or $(0, \alpha)$ with all possible combinations is enough. The result will now follow by straightforward calculations and by using (\ref{mnsl1})-(\ref{mnsl6}).
        \end{proof}

        \begin{remark}\label{conc-rem}
            Let $(\mathfrak{p}_1, \diamond_1, \lfloor ~, ~ \rfloor_1)$ and $  (\mathfrak{p}_2, \diamond_2, \lfloor ~, ~ \rfloor_2)$ be two NS-Lie algebras. Suppose there are linear maps $l, r, \psi : \mathfrak{p}_1 \rightarrow \mathrm{End} (\mathfrak{p}_2)$ and $L, R, \Psi : \mathfrak{p}_2 \rightarrow \mathrm{End} (\mathfrak{p}_1)$ such that the space $\mathfrak{p}_1 \oplus \mathfrak{p}_2$ endowed with the operations (\ref{bim1}), (\ref{bim2}) forms an NS-Lie algebra. Then  $\big(  (\mathfrak{p}_1, \diamond_1, \lfloor ~, ~ \rfloor_1),  (\mathfrak{p}_2, \diamond_2, \lfloor ~, ~ \rfloor_2), l ,r, \psi, L, R, \Psi   \big)$ is a matched pair of NS-Lie algebras.
        \end{remark}

        \begin{prop}
            Let $\big(  (\mathfrak{p}_1, \diamond_1, \lfloor ~, ~ \rfloor_1),  (\mathfrak{p}_2, \diamond_2, \lfloor ~, ~ \rfloor_2), l ,r, \psi, L, R, \Psi   \big)$ be a matched pair of NS-Lie algebras. Then the quadruple $((\mathfrak{p}_1, \llbracket ~, ~ \rrbracket_1), (\mathfrak{p}_2, \llbracket ~, ~ \rrbracket_2), l-r+\psi, L-R+\Psi )$ is a matched pair of (subadjacent) Lie algebras.
        \end{prop}

        \begin{proof}
            First, consider the bicrossed product NS-Lie algebra $(\mathfrak{p}_1 \oplus \mathfrak{p}_2 , \diamond_\Join, \lfloor ~, ~ \rfloor_\Join)$ given in Theorem \ref{bicrossed-ns}. The corresponding subadjacent Lie algebra $(\mathfrak{p}_1 \oplus \mathfrak{p}_2, \llbracket ~, ~ \rrbracket_\Join)$ is given by
            \begin{align*}
                \llbracket (x, \alpha), (y, \beta) \rrbracket_\Join = (x, \alpha) \diamond_\Join (y, \beta) - (y, \beta) \diamond_\Join (x, \alpha) + \lfloor (x, \alpha) , (y, \beta) \rfloor_\Join,
            \end{align*}
            for $(x, \alpha), (y, \beta) \in \mathfrak{p}_1 \oplus \mathfrak{p}_2$. By using the definitions of $\diamond_\Join$ and $\lfloor ~, ~ \rfloor_\Join$, one obtain that
            \begin{align*}
                  \llbracket (x, \alpha), (y, \beta) \rrbracket_\Join =~& \big(  \llbracket x, y \rrbracket_1 + (L_\alpha - R_\alpha + \Psi_\alpha) y - (L_\beta - R_\beta + \Psi_\beta) x ~, \\
                 & \qquad \qquad  \llbracket \alpha, \beta \rrbracket_2 + (l_x - r_x + \psi_x) \beta - (l_y -r_y + \psi_y) \alpha \big).
            \end{align*}
            The above expression of the Lie bracket concludes the result.
        \end{proof}

 The following result shows that a matched pair of Nijenhuis Lie algebras gives rise to a matched pair of induced NS-Lie algebras. This generalizes Proposition \ref{gen-prop}.

 \begin{prop}
     Let $( (\mathfrak{g}, [~,~]_\mathfrak{g}, N), (\mathfrak{h}, [~,~]_\mathfrak{h}, S), \rho, \nu)$ be a matched pair of Nijenhuis Lie algebras. Then the tuple $( (\mathfrak{g}, \diamond_N, \lfloor ~, ~ \rfloor_N) , (\mathfrak{h}, \diamond_S, \lfloor ~, ~ \rfloor_S), l, r, \psi, L, R, \Psi ) $ is a matched pair of induced NS-Lie algebras, where the maps $l, r, \psi : \mathfrak{g} \rightarrow \mathrm{End} (\mathfrak{h})$ and $L, R, \Psi: \mathfrak{h} \rightarrow \mathrm{End}(\mathfrak{g})$ are respectively given by
     \begin{align*}
         & ~~ l_x h = \rho_{N (x)} h, \quad r_x h = - \rho_x S(h), \quad \psi_x h = - S (\rho_x h),\\
         &L_h x = \nu_{S(h)} x, \quad R_h x = - \nu_h N (x), \quad \Psi_h x = - N (\nu_h x).
     \end{align*}
 \end{prop}

 \begin{proof}
     Since $( (\mathfrak{g}, [~,~]_\mathfrak{g}, N), (\mathfrak{h}, [~,~]_\mathfrak{h}, S), \rho, \psi)$ is a matched pair of Nijenhuis Lie algebras, it follows from Proposition \ref{prop-bicrossed-nlie} that the triple $(\mathfrak{g} \oplus \mathfrak{h}, [~,~]_\Join, N \oplus S)$ is a Nijenhuis Lie algebra. Hence the space $\mathfrak{g} \oplus \mathfrak{h}$ has an NS-Lie algebra structure with the operations
     \begin{align*}
         (x, h) \diamond_{N \oplus S} (y, k) =~& [ (N (x), S(h)), (y, k) ]_\Join \\
         =~& ( x \diamond_N y + \nu_{S(h)} y - \nu_k N (x) ~ \! , ~ \! [S(h), k]_\mathfrak{h} + \rho_{N (x)} k - \rho_y  S (h) ), \\
         \lfloor (x, h), (y, k) \rfloor_{N \oplus S} =~& - (N \oplus S) [ (x, h), (y, k) ]_\Join \\
         =~& (\lfloor x, y \rfloor_N - N (\nu_h y) + N (\nu_k x) ~ \! , ~ \! \lfloor h, k \rfloor_S - S (\rho_x k) + S (\rho_y h)),
     \end{align*}
     for all $(x, h), (y, k) \in \mathfrak{g} \oplus \mathfrak{h}$. The above two expressions prove the desired result (see Remark \ref{conc-rem}).
 \end{proof}

\medskip

\medskip

\noindent {\bf Acknowledgements.} The author thanks the Department of Mathematics, IIT Kharagpur for providing the beautiful academic atmosphere where the research has been carried out.



\medskip

\noindent {\bf Data Availability Statement.} Data sharing does not apply to this article as no new data were created or analyzed in this study.







\begin{thebibliography}{BFGM03}


\bibitem{azimi} M. J. Azimi, C. Laurent-Gengoux and J. M. Nunes da Costa, Nijenhuis forms on $L_\infty$-algebras and Poisson geometry, {\em Diff. Geom. Appl.} 38 (2015), 69-113.








\bibitem{baez-crans} J. C. Baez and A. S. Crans, Higher-Dimensional Algebra VI: Lie $2$-Algebras, {\em Theor. Appl. Categor.} 12 (2004), 492-528.

\bibitem{bai-jpa} C. Bai, A unified algebraic approach to the classical Yang-Baxter equation, {\em J. Phys. A: Math. Theor.} 40 (2007), 11073-11082.

\bibitem{bai-pre} C. Bai, Left-symmetric bialgebras and an analogue of the classical Yang-Baxter equation, {\em Commun. Contemp. Math.} 10 (2008), 221-260.





\bibitem{bai-bialgebra} C. Bai, L. Guo, G. Liu and T. Ma, Rota-Baxter Lie bialgebras, classical Yang-Baxter equations and special L-dendriform bialgebras, {\em Algebr. Represent. Theor.} 27 (2024), 1347-1372.






\bibitem{gra-bi} J. F. Cari\~{n}ena, J. Grabowski and G. Marmo, Quantum Bi-Hamiltonian systems, {\em Int. J. Mod. Phys. A} 15 (2000), 4797-4810. 

\bibitem{chari}  V. Chari and A. Pressley, A Quide to Quantum Groups, Cambridge University Press, Cambridge (1994).




\bibitem{das-twisted} A. Das,  Twisted Rota-Baxter operators and Reynolds operators on Lie algebras and NS-Lie algebras, {\em J. Math. Phys.} Vol. 62, Issue 9  (2021) 091701.

\bibitem{baishya-das} A. Baishya and A. Das, Cup product, Fr\"{o}licher-Nijenhuis bracket and the derived bracket associated to Hom-Lie algebras, arXiv preprint, arXiv:2409.01865.

\bibitem{das-nj} A. Das, Non-abelian cohomology of Nijenhuis Lie algebras and the inducibility of automorphisms and derivations, Submitted.













\bibitem{dorfman} I. Dorfman, Dirac structures and integrability of nonlinear evolution equations, Wiley, 1993.

\bibitem{doubek} M. Doubek, M. Markl and P. Zima, Deformation theory (lecture notes), {\em Archivum mathematicum} 43 (5) (2007), 333-371.

\bibitem{drin} V. Drinfeld, Quantum groups, Proc. ICM 86, Berkeley (1986), 798-820.





\bibitem{fro-nij-1} A. Fr\"{o}licher and A. Nijenhuis, Theory of vector-valued differential forms. Part I. {\em Indag. Math.}, 18 (1956), 338-360.


\bibitem{gers}  M. Gerstenhaber, On the deformation of rings and algebras, {\em Ann. Math. (2)} 79 (1964), 59-103.

\bibitem{gers-sch} M. Gerstenhaber and S. D. Schack, On the deformation of algebra morphisms and diagrams, {\em Trans. Amer. Math. Soc.} 279 (1983), no. 1, 1-50.







\bibitem{hoch} G. Hochschild, On the cohomology groups of an associative algebra, {\em Ann. Math. (2)} 46 (1945), 58-67.



\bibitem{jiang-sheng} J. Jiang and Y. Sheng, Representations and cohomologies of relative Rota-Baxter Lie algebras and applications, {\em J. Algebra} 602 (2022), 637-670.




\bibitem{koss} Y. Kosmann-Schwarzbach and F. Magri, Poisson-Nijenhuis structures,  {\em Annales de l'I.H.P. Physique th\'{e}orique} 53, no. 1 (1990), 35-81.




\bibitem{lada-markl} T. Lada and M. Markl, Strongly homotopy Lie algebras, {\it Comm. Algebra} 23 (1995), 2147-2161.

\bibitem{lang-sheng} H. Lang and Y. Sheng, Factorizable Lie bialgebras, quadratic Rota–Baxter Lie algebras and Rota–Baxter Lie bialgebras, {\em Commun. Math. Phys.} 397 (2023), 763–791.












\bibitem{loday-der} J.-L. Loday, On the operad of associative algebras with derivation, {\em Georgian Math. J.} 17 (2010), 347-372.

\bibitem{majid-lie} S. Majid, Matched pairs of Lie groups associated to solutions of the Yang-Baxter equation, {\em Pacific J. Math.} 141 (1990), 311-332.


\bibitem{markl} M. Markl, Deformation theory of algebras and their diagrams, CBMS Regional Conference Series in Mathematics, Volume 116 (2012), 129 pp.



\bibitem{nij-ric} A. Nijenhuis and R. Richardson, Cohomology and deformations in graded Lie algebras, {\em Bull. Amer. Math. Soc.} 72 (1966), 1-29.


\bibitem{nij-ric2} A. Nijenhuis and R. Richardson, Deformations of homomorphisms of Lie groups and Lie algebras, {\em Bull. Amer. Math. Soc.} 73 (1967), 175-179.


\bibitem{saha} B. Mondal and R. Saha, Nijenhuis operators on Leibniz algebras, {\em J. Geom. Phys.} 196 (2024), 105057.

\bibitem{ravanpak} Z. Ravanpak, NL bialgebras, {\em Adv. Theor. Math. Phys.} to appear.



\bibitem{sheng-o} R. Tang, C. Bai, L. Guo and Y. Sheng, Deformations and their controlling cohomologies of $\mathcal{O}$-operators, {\em Commun. Math. Phys.} 368 (2019), 665-700.








\bibitem{voro} Th. Voronov, Higher derived brackets and homotopy algebras, {\em J. Pure Appl. Algebra} 202 (2005), 133-153.



\bibitem{yang} Q. Yang, Some graded Lie algebra structures associated with Lie algebras and Lie algebroids, PhD Thesis, University of Toronto (1999).


\end{thebibliography}
\end{document}